\documentclass{imsart}
\usepackage{float}
\usepackage{pgf}
\usepackage{tikz}
\usepackage{comment}
\usetikzlibrary{shadows,backgrounds}
\usepackage{amsmath}
\usepackage{amsfonts}
\usepackage{amsthm}
\usepackage{xcolor}
\usepackage{bbm}
\usepackage{amssymb}
\usepackage{nomencl}
\usepackage[utf8]{inputenc}  
\usepackage[T1]{fontenc}   
\usepackage{tikz,pgfplots}
\pgfplotsset{compat=1.14}

\DeclareMathOperator{\tr}{Tr}

\DeclareMathOperator{\diag}{Diag}
\DeclareMathOperator{\conv}{Conv}
\DeclareMathOperator{\dimm}{dim}
\DeclareMathOperator{\vect}{Vect}
\DeclareMathOperator{\id}{Id}
\newtheorem{theorem}{Theorem}[section]
\newtheorem{assumption}{Assumption}
\newtheorem{corollary}[theorem]{Corollary}
\newtheorem{example}[theorem]{Example}
\newtheorem{definition}{Definition}
\newtheorem{lemma}[theorem]{Lemma}
\newtheorem{proposition}[theorem]{Proposition}
\newtheorem{remark}[theorem]{Remark}
\newcommand{\un}{\mathbbm{1}}
\usepackage{tikz,pgfplots}
\usepackage[titletoc]{appendix}
 \pgfplotsset{compat=1.14}
 
 \newcommand{\ajout}{\color{black}}
\def\restriction#1#2{\mathchoice
              {\setbox1\hbox{${\displaystyle #1}_{\scriptstyle #2}$}
              \restrictionaux{#1}{#2}}
              {\setbox1\hbox{${\textstyle #1}_{\scriptstyle #2}$}
              \restrictionaux{#1}{#2}}
              {\setbox1\hbox{${\scriptstyle #1}_{\scriptscriptstyle #2}$}
              \restrictionaux{#1}{#2}}
              {\setbox1\hbox{${\scriptscriptstyle #1}_{\scriptscriptstyle #2}$}
              \restrictionaux{#1}{#2}}}
\def\restrictionaux#1#2{{#1\,\smash{\vrule height .8\ht1 depth .85\dp1}}_{\,#2}} 
\newcommand{\RR}{\mathbb{R}}
\newcommand{\trans}{T}
\newtheorem{assump}{Assumption}
\newenvironment{assbis}[1]
  {%
  \addtocounter{assumption}{-1}   \begin{assumption}}
  {\end{assumption}}

\begin{document}







\begin{frontmatter}

\title{Concentration of Measure and Large Random Matrices \\ with an application to Sample Covariance Matrices
\protect\thanksref{T1}}

\runtitle{Concentration of Measure and Large Random Matrices}
\thankstext{T1}{This work is supported by the ``LargeDATA'' MIAI Chair at UGA, and the HUAWEI-GIPSA Lardist project.}
\begin{aug}
  \author{\fnms{Cosme}  \snm{Louart}\corref{}\thanksref{t1}\thanksref{t2}\ead[label=e1]{cosme.louart@gipsa-lab.grenoble-inp.fr}}
  \and
  \author{\fnms{Romain} \snm{Couillet}\corref{}\thanksref{t2}\thanksref{t3}\ead[label=e2]{romain.couillet@gipsa-lab.grenoble-inp.fr}\ead[label=u1,url]{http://romaincouillet.hebfree.org}}

  \thankstext{t1}{CEA, LIST, LVIC, nano-innov, 91120 Palaiseau}
  \thankstext{t2}{Univ. Grenoble Alpes, CNRS, Grenoble Institute of engenering, GIPSA-lab, 38000 Grenoble, France}  
  \thankstext{t3}{Univ. de Paris-Saclay, CentraleSup\'{e}lec, L2S, 91190 Gif-sur-Yvette, France}

  \runauthor{Louart}

  \affiliation{Laboratoire L2S, GIPSA-lab}

  \address{Louart Cosme : Nano-INNOV, Avenue de la Vauve, \\Bâtiment 861, 91120 Palaiseau.\\ 
          \printead{e1}}
  \address{Romain Couillet : Bureau D1178, 11 rue des math\'{e}matiques \\ Domaine Universitaire BP 46, 38402 Saint Martin d'Hères cedex. \\  
  \printead{e2} \\
      \printead{u1}}

\end{aug}
\begin{abstract}
The present work provides an original framework for random matrix analysis based on revisiting the concentration of measure theory from a probabilistic point of view.
By providing various notions of vector concentration ($q$-exponential, linear, Lipschitz, convex), a set of elementary tools is laid out that allows for the immediate extension of classical results from random matrix theory involving random concentrated vectors in place of vectors with independent entries. These findings are exemplified here in the context of sample covariance matrices but find a large range of applications in statistical learning and beyond, thanks to the broad adaptability of our hypotheses.
\end{abstract}
\end{frontmatter}



\tableofcontents

\renewcommand{\abstractname}{Acknowledgements}

\nomenclature[-1]{iif}{``if and only if''}
\nomenclature[01]{$E$}{Typical normed vector space over $\mathbb R$, endowed with the norm $\Vert \cdot\Vert$.}
\nomenclature[10]{$E^*$}{Dual space of $E$ (the set of linear maps from $E$ to $\mathbb R$).}
\nomenclature[11]{$\mathcal A$}{Typical non commutative algebra endowed with the algebra norm $\Vert \cdot\Vert$, the possibly non commutative product is written without operation character, we have $\forall x,y\in \mathcal A$, $\Vert xy\Vert \leq \Vert x \Vert \Vert y \Vert$.}
\nomenclature[110]{Conv$(A)$}{Convex hull of $A\subset E$ (i.e., Conv$(A)=\cap\{C \subset E,\, C \text{ convex}, C \supset A\}$).}
\nomenclature[1100]{$\partial A$}{Boundary of $A\subset E$, if $\bar A$ is the closure of $A$ and $\mathring A$, the interior of $A$, $\partial A=\bar A \setminus \mathring A$.}
\nomenclature[111]{$A_*$}{$A \setminus \{0\}$}
\nomenclature[112]{$\un_A$}{Indicator function of $A\subset E$, $\un_A :E \rightarrow \{0,1\}$ and $\un_A(x)=1 \Leftrightarrow x \in A$. If $A$ is not a set but an assertion (like $A=A(t)=(t\geq 1)$), $\un_A=1$ if $A$ is true and $\un_A=0$ if $A$ is false. If $\un$ presents no index, it designates a vector full of one with a convenient size for the context.}
\nomenclature[113]{$\mathfrak S_p$}{Set of permutations of $\{1,\ldots,p\}$ ; $\mathfrak S_{p,n} = \mathfrak S_p \times \mathfrak S_n$.}
\nomenclature[115]{$\prec$}{Majorization relation, see Definition~\ref{def:majorization}}
\nomenclature[114]{$x^{\downarrow}$}{Decreasing version of $x \in \mathbb R^p$, $\exists \sigma \in \mathfrak S_p$ such that $\forall i \in \{1,\ldots, p\}$, $[x^{\downarrow}]_i = x_{\sigma(i)}$ and $[x^{\downarrow}]_i\leq [x^{\downarrow}]_{i-1} $ for $ i \geq 2$.}
\nomenclature[00]{$\mathbb R$}{Set of real numbers ; $\mathbb R_+=\{x\in \mathbb R \ \vert \ x \geq 0\}$ ; $\mathbb R_- = - \mathbb R_+$. If $x \in \mathbb R^p$, one notes $[x]_i$, or more simply $x_i$, $1\leq i \leq p$, the $i^{\text{th}}$ entry of the vector $x$}
\nomenclature[000]{$(a,b]$}{Considering an interval $I\subset \mathbb{R}$, we employ the sign ``$($'' if left border of $I$ is open and ``$[$'' if it is closed ; the same rule works for the right border. To give an example, given $a,b\in \mathbb R$ : $(a,b]=\{x \ \vert \ a<x\leq b\}$}
\nomenclature[0000]{$\lfloor x \rfloor$}{Integer part of $x\in \mathbb R$, $\lfloor x \rfloor\in \mathbb N$ and verifies $\lfloor x \rfloor \leq x \leq \lfloor x \rfloor + 1 = \lceil x \rceil$.}
\nomenclature[01]{$\mathbb C$}{Set of complex numbers.}
\nomenclature[12]{$\mathcal M_{p,n}$}{Set of real matrices of size $p \times n$. If $M \in \mathcal M_{p,n}$, one notes $[M]_{i,i}$ or more simply $M_{i,j}$ the entry at the line $i$ and column $j$. If $n=p$, we simply note $\mathcal{M}_{p}=\mathcal M_{p,n}$.}

\nomenclature[121]{$\tr$}{Trace operator on $\mathcal M_{p}$, $\forall M \in \mathcal M_{p}, \tr M=\sum_{i=1}^p M_{i,i}$.}
\nomenclature[122]{$\cdot^T$}{Transpose operator on $\mathcal M_{p}$, $\forall M \in \mathcal M_{p,n}$, $[M^T]_{i,j} = M_{j,i}$, $1\leq i \leq p$, $1\leq j \leq n$.}
\nomenclature[123]{$I_p$}{Identity matrix of $\mathcal M_{p}$, ($[I_p]_i,j=0$ if $i \neq j$ and $[I_p]_{i,i}=1$, $1\leq i,j\leq j$).}
\nomenclature[123]{Diag}{Diagonal operator. If $M \in \mathcal M_{p,n}$, $\text{Diag}(M)=(M_{i,i})_{1\leq i\leq \min(p,n)}$ ; if $x\in \mathbb R^p$, $\text{Diag}_{q,n} (x) \in \mathcal M_{p,n}$, $[\text{Diag}_{p,n} (x)]_{i,i}=x_i$ if $1\leq i \leq \min(p,q,n)$ and $[\text{Diag}_{p,n} (x)]_{i,j}=0$ if $i\neq j$, for $1\leq i \leq p$, $1\leq j \leq n$.}
\nomenclature[1231]{$\text{Sp}(M)$}{Spectrum of the matrix $M$.}
\nomenclature[1232]{$Q_C$}{Resolvent of the matrix $C \in \mathcal M_p$. For $z \in \mathbb C \setminus \text{Sp}(C)$, $Q_C(z)=(C +z I_p)^{-1}$.}
\nomenclature[126]{$\mathcal S_p$}{Set of symmetric matrices of $\mathcal M_p$ : $S \in \mathcal S_n \Leftrightarrow  S_{i,j}=S_{i,j}, 1\leq i \leq p$ ; $S \in \mathcal S_p^+ \Leftrightarrow \forall u \in \mathbb R^p, u^TSu\geq 0$ ; $S \in \mathcal S_p^- \Leftrightarrow -S \in \mathcal S_p^+$. Given $S_1,S_2 \in \mathcal S_p$, we say that $S_1$ is greater  than $S_2$ and we note $S_1 \geq S_2$ if $S_1-S_2 \in \mathcal S_p^+$.}
\nomenclature[1261]{$S^{1/2}$}{Square root of the nonnegative symmetric matrix $S \in \mathcal{S}_p^+$ (with the diagonalization $S=P^T \Lambda P$, $P\in \mathcal O_p$, $\Lambda \in \mathcal D_p$, we define $S^{1/2}=P^T \Lambda^{1/2} P$ where $[\Lambda^{1/2}]_{i,i}=\Lambda_{i,i}^{1/2}$).}
\nomenclature[124]{$\mathcal O_{p}$}{Set of orthogonal matrices of $\mathcal M_{p}$ : $P \in \mathcal D_{p,n} \Leftrightarrow  P^{-1}=P^T$ ; $\mathcal O_{p,n}=\mathcal O_p \times \mathcal O_n$.}
\nomenclature[125]{$\mathcal D_{p,n}$}{Set of diagonal matrices of $\mathcal M_{p,n}$ : $D \in \mathcal D_{p,n} \Leftrightarrow  D=\text{Diag}_{p,n}(\text{Diag}(D))$ ; $D \in \mathcal D_{p,n}^+ \Leftrightarrow D_{i,i}\geq 0, 1 \leq i \leq \min(p,n)$ ; $D \in \mathcal D_{p,n}^- \Leftrightarrow  -D \in \mathcal D_{p,n}^+$. When $n=p$, we simply note $\mathcal D_p=\mathcal D_{p,n}$.}
\nomenclature[127]{$\mathcal P_{p}$}{Set of permutation matrices of $\mathcal M_{p}$ : $P \in \mathcal P_{p} \Leftrightarrow P \in \mathcal O_n \text{ and } (\exists \sigma \in \mathfrak S_p, \ P_{i,j}=1 \Leftrightarrow \sigma(i)=j)$ ; we also define $\mathcal P_{p,n}= \{(U,V) \in \mathcal P_p \times \mathcal P_n \ \vert \ U I_{p,n}V^T=I_{p,n}\}$ where $I_{p,n} = \diag_{p,n}(\un)$.}
\nomenclature[2]{$\Vert \cdot \Vert_q$}{$\ell_q$-norm on $\mathbb R^p$ for two integers $p,q \in \mathbb N_*$ ; $\Vert x \Vert_q = \left(\sum_{i=1}^p x_i^q \right)^{1/q}$.}
\nomenclature[20]{$\Vert \cdot \Vert$}{Classical norm of the vector space $E$ one is working on : if $E=\mathbb R^p$, the euclidean norm $\Vert \cdot \Vert_2$ ; if $E=M_{p,n}$, the spectral norm ($\forall M \in \mathcal M_{p,n} : \sup_{\Vert u \Vert = 1}\Vert Mu\Vert$).}
\nomenclature[21]{$\Vert \cdot \Vert_1$}{on $\mathbb R^p$, the $\ell_1$ ; on $M_{p,n}$, the nuclear norm ($\forall M \in \mathcal M_{p,n}, \Vert M \Vert_1 = \tr( (MM^T)^{1/2})$).}
\nomenclature[22]{$\Vert \cdot \Vert_F$}{Frobenius norm, $\forall M \in \mathcal M_{p,n} : \Vert M \Vert_F = \sqrt{\tr MM^T} = \sqrt{\sum_{i=1}^p\sum_{j=1}^nM_{i,j}^2}$.}
\nomenclature[23]{$\Vert \cdot \Vert_*$}{Dual norm on $E^*$. If $f \in E^*$, $\Vert f \Vert _* = \sup_{\Vert x \Vert \leq 1} f(x)$.}
\nomenclature[24]{$d_{\Vert \cdot \Vert}$}{Distance associated to the norm $\Vert \cdot \Vert$. Given two vectors $x,y \in E$ and two sets $A,B\subset E$, $d_{\Vert \cdot \Vert}(x,y)=\Vert x-y\Vert$, $d_{\Vert \cdot \Vert}(x,A)=\inf\{d_{\Vert \cdot \Vert}(x,y), y \in A \}$ and $d_{\Vert \cdot \Vert}(A,B)=\inf\{d_{\Vert \cdot \Vert}(x,B), x \in A \}$.}
\nomenclature[241]{$\mathcal B_t$}{Closed ball of $E$ of size $t>0$, $\mathcal B_t = \{0\}_t=\{x\in E \ \vert \ \Vert x \Vert \leq 1\}$, when $t=1$, we note $\mathcal B = \mathcal B_1$. We also use the notation $\mathcal B_{\Vert \cdot \Vert}(x,t)=\{x\}_t=\{y \in E \ \Vert \ \Vert x-y \Vert\leq t\}$, the index $\Vert \cdot \Vert$ could be of course a distance $d$ or simply unspecified when we implicitly consider the classical norm of $E$.}
\nomenclature[25]{$A_t$}{If $A\subset E$, $A_t$ is the closed set $\{x\in E \ \vert \ d(x,A) \leq t\}\supset A$.}
\nomenclature[26]{$S^t_f$}{Level set of $f : E \rightarrow \mathbb R$. For $t\in \mathbb R$, $S^t_f=\{x\in E \ \vert \ f(x)\leq t\}$.}
\nomenclature[3]{$\mathbb S^p$}{Sphere of $\mathbb{R}^{p+1}$ ($\mathbb S^p=\{x \in \mathbb{R}^{p+1} \ \vert \ \Vert x \Vert = 1\}$).}
\nomenclature[6]{$\mathbb P$}{We implicitly suppose all over the paper that there exists a probability space $(\Omega, \mathcal F, \mathbb P)$ where $\mathcal F$ is a sigma-algebra of the set $\Omega$ and $\mathbb P$, a probability measure defined on the elements of $\mathcal F$. The random vectors we consider are then $\mathbb P$-measurable applications defined on $\Omega$ and taking value in normed vector spaces endowed with the Borel $\sigma$-algebra. In that setting, given a random vector $X \in E$ (i.e. $X : \Omega \rightarrow E$), for any Borel set $A\subset E$, we note $\mathbb P (X \in A)=\mathbb P (\{\omega \in \Omega  , \ X(w)\in A\})$ and for any measurable function $f:E \mapsto \mathbb R$ and $t\in \mathbb R$, we note $\mathbb P(f(X)\geq t)=\mathbb P(\{\omega \in \Omega, \ f(X(\omega))\geq t\})$...}
\nomenclature[61]{$\mathbb E$}{The expectation operator. For any random vector $X \in E$ and any measurable function $f : E \rightarrow \mathbb R$, we define $\mathbb E [f(X)] = \int_\Omega f\circ X d\mathbb P $. When $E$ has finite dimension, it is possible to define $\mathbb E X$ by integrating all the coordinates of $X$.}
\nomenclature[611]{a.s.}{Almost surely, an event (i.e., an element of $\mathcal F$) $A$ is true almost surely iff $\mathbb P(A)=1$. We will often abusively mix up random vectors and classes of almost surely equal random vectors.}
\nomenclature[612]{i.i.d.}{``independent and identically distributed''.}
\nomenclature[6111]{$\sigma(X)$}{If $X \in \mathcal{M}_{p,n}$, $\sigma(X) \in \mathbb R_+^{\min(p,n)}$ is the vector constituted of the singular values of $X$ in increasing order (i.e. the eigenvalues of $(XX^T)^{1/2}$). If $X \in E$ is a random vector, $\sigma(X)$ is the $\sigma$-algebra generated by $X$ (i.e., the $\sigma$-algebra of sets of $\Omega$ containing all the sets $X^{-1}(B)$ when $B$ is a Borel set of $E$).}
\nomenclature[613]{$\mathbb E[\cdot \vert \cdot]$}{Conditional expectation. Given a random variable $X \in \mathbb R$ $\mathcal F$-measurable and $\mathcal G$, a sub $\sigma$-algebra of $\mathcal F$, $\mathbb E[X \vert \mathcal G]$ is the a $\mathcal G$-measurable random variable that satisfies for any additional $\mathcal G$-measurable random variable $Y$, $\mathbb E[Y\mathbb E[X \vert \mathcal G]] = \mathbb E[YX]$.
verifies for any $\mathcal G$-measurable random variable $Y$, $\mathbb E[YX]=\mathbb E[Y \mathbb E [X \vert \mathcal G]]$. Given two random variables $X,Y \in \mathbb R$, $\mathbb E[X \vert Y]= \mathbb E[X \vert \sigma(Y)]$.}
\nomenclature[614]{$\mathbb P(\cdot \vert \cdot )$}{Conditional probability. Given a Borel set $A\subset \mathbb R$ and a random variable $X \in \mathbb R$, $\mathbb{P}(A \vert X)=\mathbb E [\un_A \vert X]$.}
\nomenclature[62]{$\in, \pm $}{Concentration around a pivot or around a deterministic equivalent, see Definitions~\ref{def:concentration_autour} and \ref{def:deterministic equivalent}.}
\nomenclature[63]{$\propto$}{Lipschitz concentration, see Definitions~\ref{def:conc_variable_aleatoire} and~\ref{def:conc_lipschitzienne_vect_al}.}
\nomenclature[64]{$\propto_c$}{Convex concentration, see Definition~\ref{def:convexite_faible}.}
\nomenclature[65]{$\propto_\cdot^T$}{Transversal concentration, see Definition~\ref{def:concentrationfaible_transversale}.}
\nomenclature[65]{$\mathcal R_X$}{Observable diameter, see Remark~\ref{rem:diametre_observable_generalise}}
\nomenclature[90]{$d\mu$}{If $\mu$ is a probability law defined on $E$, for any function $f : E \rightarrow \mathbb R$, such that $\mu(f) = \int f d\mu =1$, we note $fd\mu$ the measure verifying for all Borel set $B$ : $fd\mu(B) = \int \un_B f d\mu$.}
\nomenclature[91]{$\lambda_p$}{Lebesgue measure on $\mathbb{R}^p$.}
\nomenclature[92]{$\sigma_p$}{Uniform measure on $\mathbb{S}^p$.}
\nomenclature[92]{$\nu^p$}{Exponential measure. If $p=1$ $\nu^1=\frac{e^\vert \cdot \vert}{2} d\lambda_1$, for $p\geq 1$, $\nu^p =\nu^1 \otimes \ldots \nu^1$ ($p$times).}
\nomenclature[92]{$\beta_q^p$}{Uniform measure on the ball $\mathcal B_{\Vert\cdot \Vert_q}$ of $\mathbb R^p$.}
\nomenclature[93]{$\mathcal N(0,I_p)$}{Distribution of Gaussian vectors of $\mathbb{R}^p$ with zero mean and covariance $I_p$.}
\nomenclature[4]{$\eta_{(E, \Vert \cdot \Vert)}$}{Norm degree, see Definition~\ref{def:norm_degree}.}
\nomenclature[94]{$m_\mu$}{Stieltjes transform of the probability law $\mu$ on $\mathbb R$. Let $D \subset \mathbb R$ be the maximal Borel set such that $\mu(\mathbb R \setminus D)=0$, then $\forall z \in \mathbb C \setminus D$, the Stieltjes transform is defined with the formula $m_\mu(z)=\int_D \frac{d\mu(w)}{z-w}$.}
\printnomenclature

\section*{Introduction}

Sample covariance matrices are key quantities in applied statistics in that they allow for the estimation of structural information in the second order statistics of the sampled vectors, and find a wide range of applications in fields as diverse as applied statistics (e.g., financial statistics, biostatistics), signal or data processing, wireless communications, etc. Precisely, for a set of $n$ independent random vectors $x_1,\ldots x_n \in \mathbb{R}^p$ stacked in a matrix $X=[x_1, \cdots x_n] \in \mathcal{M}_{p,n}$, the sample covariance matrix $S=\frac{1}{n}\sum_{i=1}^n x_ix_i^T=\frac{1}{n}XX^T \in \mathcal{M}_p$ provides an estimator for $\Sigma=\frac{1}{n}\mathbb{E}[XX^T]$. 

If the number of independently sampled vectors $n$ is large compared to the dimension $p$ of the vectors, then under usually mild assumptions $S$ converges to $\Sigma$. When $p$ and $n$ have the same order of magnitude though, again under classical assumptions, the operator norm difference $\|S-\Sigma\|$ usually does not vanish and $S$ is thus not a consistent estimator for $\Sigma$. In their now famous article \cite{MAR67}, Mar\u{c}enko and Pastur proved that, if the vectors $x_i$, in addition to being independent, have independent entries of zero mean and unit variance, then, as $n,p\to\infty$ with $p/n\to c\in(0,\infty)$, the normalized counting measure of the eigenvalues of $S$ converges a.s. to a limiting distribution having a continuous density, and now referred to as the Mar\u{c}enko--Pastur distribution. From this article on, many works have provided generalizations of \cite{MAR67}. This is the case for instance of \cite{SIL95}, where the authors assume that the $x_i$ can be written under the form $x_i=\Sigma^\frac12z_i$ for $z_i$ a vector with independent zero mean and unit variance entries. It is to be noted that the independence \cite{MAR67}, linear dependence \cite{SIL95}, or vanishing dependence \cite{ADA11} between the entries of $X$ is key in the approach pursued in these articles as it provides a necessary additional degree of freedom in the derivation of the proofs.

These earlier results thus found many practical applications in scientific fields involving large dimensional matrix models with mostly linearly dependent entries, most notably in applied statistics, electrical engineering and computer science. But the renewed interest for machine learning applications, spurred by the big data era, has recently brought forward the need to understand and improve algorithms and methods relying on random matrix models involving non-linear relations between their entries. In some scenarios, as with kernel matrices (that is, matrices $K\in\RR^{n\times n}$ with entries of the type $K_{ij}=f(x_i^\trans x_j)$ for some non-linear function $f$), an asymptotic equivalence between these matrices and classical matrices with linearly dependent entries can be proved \cite{ELK10,COU16,KAM17}, thereby transferring the asymptotic analysis of the former to that of the latter. In other scenarios though, such asymptotic equivalences are not available. This is in particular the case of so-called random feature maps and neural networks. In random feature maps, the vectors $x_i$ can be expressed under the form $x_i=\sigma(Wz_i)$ for some non-linear (referred to as the activation) function $\sigma:\RR\to\RR$, here applied entry-wise, $W$ a given matrix, and $z_i$ yet another random vector. The randomness in random feature maps arises from the fact that $W$ is usually chosen at random, often with independent and identically distributed entries (this way allowing to produce $p$ independent non-linear ``features'' of the vector $z_i$). Prior to training, neural networks (in particular feedforward neural nets) may usually be seen as a cascade of such random feature maps. In evaluating the performance of algorithms and methods based on random feature maps and neural networks, it is often of importance to understand the statistical behavior of $S$.

Still in the scope of the statistical analysis of data processing algorithms, where sample covariance matrices built upon \emph{real data vectors} $x_i$ are considered, it is also quite restrictive, if not disputable, to assume that $x_i$ can be written as $x_i=\mu+\Sigma^\frac12w_i$ for some deterministic $\mu$ and with $w_i$ having independent entries.

\bigskip

As we shall see in the course of the article, a very convenient assumption to be made on $x_i$ in order (i) to answer the aforementioned controversial real dataset modeling, (ii) to properly model random feature map and neural network and (iii) to largely generalize the Mar\u{c}enko--Pastur and related results, is to propose that \emph{$x_i$ satisfies a vector-concentration inequality}. In a nutshell, concentration inequalities being stable under (bounded) linear operations \emph{and non-linear Lipschitz operations}, the framework proposed in this study allows for a natural study of models of $S$ with independent $x_i$ of the form $f(z_i)$ with $z_i$ itself a concentrated random vector (for instance a Gaussian vector)\footnote{This framework is particularly interesting when one is studying neural nets since in that case the concentration is preserved through the layers that perform Lipschitz transformation of the output of previous one with operations of the form $z_{l+1} = \sigma(Wz_l)$ where $z_{l} and z_{l+1}$ are respectively the output of the $l ^{\textit{th}}$ and of the $(l+1)^{\textit{th}}$ layers, $\sigma$ Lipschitz and $\|W\|$ small compared to the dimension and the number of elements in the training set. Also, it can be showed that highly realistic images produced by generative adversarial neural networks are concentrated by construction~\cite{SED20}}.

The objective of the article is precisely to provide a consistent method for the analysis of $S$ with $x_i$ independent concentrated random vectors \emph{based on elementary concentration inequality principles}, particularly suited to practical applications in large dimensional machine learning. The article notably extends the results of \cite{Bai08t}\footnote{In \cite{Bai08t}, the authors only suppose that the second centered moment of $x_i^TBx_i/n$ is bounded for any deterministic matrix $B$ with bounded spectral norm -- in other words they assume that an analog to Hanson-Wright Theorem (see \cite{Ver17}) could be applied to their data. This assumption is a direct consequence to some of the hypotheses posed in the present article, and thus could a priori be claimed more general; however, we obtain in our conclusions a (quasi-asymptotic) estimation for the Stieltjes transform of the spectral measure of $S$ by means of a subgaussian concentration inequality, providing the optimal rate of convergence, which the authors of \cite{Bai08t} cannot reach under their limited hypotheses. This being said, it is interesting to explore the conditions under which ``Hanson-Wright like'' theorems arise in the first place: a significant part of the present article is precisely dedicated to this endeavor.}, \cite{PAJ07}\footnote{The authors there assume that the data are isotropic and follow a log-concave distribution, which our present framework encompasses as a special case --- in particular, any random vector $X$ of the form $X =f(Z)$ with $f$ $1$-Lipschitz and $Z$ following a log-concave distribution is valid.}, \cite{CHA15}\footnote{Here the authors introduce what they refer to as a ``weak tail projection property'' for the data, which basically imposes concentration inequalities to the norm of all the orthogonal projection of the data. Under this assumption, the authors study the asymptotic behavior of the extreme eigenvalues of the sample covariance matrix. As for \cite{Bai08t}, this weak tail projection property is an immediate consequence of the concentration hypothesis made in the present article, which we consider to be more natural --- first because the concentration of measure framework provides a richer set of random vector examples \cite{LED05}; second because it provides a better control on the convergence rate of most estimates considered in the article.}, all generalizing the independence assumption posed in the classical proof approaches of \cite{MAR67,SIL95}, and more specifically of \cite{ELK09} in which the author first exploited concentration of measure arguments. The technical arguments and findings of \cite{ELK09} are nonetheless quite specific and treated lemma after lemma, some of the main results being valid only for a restricted class of concentrated random vectors (such as elliptically distributed random vectors). We rather aim here at a self-contained generic framework for the manipulation of a large class of random matrix models using quite generic concentration identities. For instance, in the present article, we provide a similar result to Hanson-Wright concentration inequality concerning quantities of the type $x^TAx$ where $x\in \mathbb R^p$ is random and $A\in \mathcal M_p$ deterministic. With our hypothesis issued from concentration of measure theory, it is possible to consider random vectors $x$ with complexly (not linearly) entangled entries (see Theorems~\ref{the:conc_forme_quadrat_hypo_conc_lipsch} and \ref{the:conc_forme_quadrat}). The present work also follows after a previous article by the same authors \cite{LOU17b}, in which a particular model of concentrated random vectors $x_i=\sigma(Wz_i)$ was studied in the aim of analyzing the performances of a one hidden-layered non-linear random neural network, commonly referred to as an extreme learning machine \cite{HUA06}. Concentration was there induced by the randomness of $W$ and the Lipschitz character of $\sigma$. Yet, not imposing that the input data $z_i$ are themselves concentrated random vectors led to restrictions in the results of \cite{LOU17b}, to which the present article easily circumvents. 

Specifically, our main contributions are:
\begin{enumerate}
  \item a formal description of three different stable classes of concentrated random vectors, being --- from the larger to the smaller --- linearly concentrated vectors, convexly concentrated vectors, and Lipschitz concentrated vectors. Our approach goes beyond \cite{LED05} in that we start from the main theorems providing the class of concentration of a given random vector $X$, from which we then deduce the concentration of (sufficiently regular) real functionals $f(X)$, called the ``observations of $X$''. Notably, we do not rediscuss isoperimetrical inequalities in the sense of Ledoux (at the heart of the Lipschitz concentration of Gaussian and uniform spherical random vectors), nor do we elaborate on the popular Talagrand theorem\footnote{This major result gives the convex concentration (in our formalism) of random vectors of $\mathbb R^p$ with independent and bounded entries \cite{TAL95}.} (which establishes convex concentration for models of independent and bounded entries). 
  \item a general approach to obtain concentration inequalities for the norm of (linearly) concentrated random vectors, through the introduction a new central object called the \textit{norm degree} (Definition~\ref{def:norm_degree}). 
  \item an expression of the concentration of the product of Lipschitz concentrated random vectors (Theorem~\ref{the:concentration_des_transformations_lipschitz_multilineairement}).
  \item an expression of two Hanson--Wright-like results with fine-tuned concentration rates in the case of Lipschitz concentration (Theorem~\ref{the:conc_forme_quadrat_hypo_conc_lipsch}) or of convex concentration (Theorem~\ref{the:conc_forme_quadrat}).
  \item a control of the concentration of the resolvent $Q \equiv (S +z I_p)^{-1}$ of the random sample covariance matrix $S=\frac1nXX^T$ under convex concentration hypotheses on $X$, i.e., based on hypotheses derived from Talagrand's theorem.
  \item a very simple proof of the convergence of the spectral distribution of the sample covariance matrix $S = \frac{1}{n}XX^T$ with $X = (x_1,\ldots,x_n) \in \mathcal M_{p,n}$, in a general setting where 
  \begin{itemize}
    \item the columns of $x_1,\ldots, x_n$ do not have to be identically distributed (but must be independent)
    \item $\forall i \in \{1,\ldots, n\}$, the entries of $x_i$ do not have to be independent (nor linearly dependent)
    \item the expectations $\mathbb E[x_1],\ldots, \mathbb E[x_n]$ are allowed to be slight modifications of a small number of ``signals'' $s_1,\ldots,s_k$ (with $k\ll p, n$), each having a norm of order up to $O(\sqrt p)$. For instance, $\mathbb E[x_i]$ may, for all $i$, be small (say, $O(1)$-norm) deviations of a common $O(\sqrt p)$-norm mean vector $s$.
  \end{itemize}
\end{enumerate}
All those contributions are presented with a new set of notations designed to optimize the readability of our proposed framework. This original formalism to concentration of measure theory is thoroughly elaborated in (the lengthy but compelling) Section~\ref{sse:conc_variable_al}; this section rediscovers known results on the concentration of random variables --- many of which are collected in \cite{LED05} ---, which are then extended into new contributions to random vectors in Section~\ref{sse:conc_vect_al}.
\color{black}

\bigskip

The remainder of the article is structured as follows. To give a somewhat original (but convenient to our analysis) probabilistically oriented approach to the notion of concentration of measure, we present, in a first section, a collection of globally well known lemmas concerning random variables in an efficient formalism that we introduce to employ it, in a second part, to the case of random vectors (one original contribution of our approach is to study the stability of the concentration through basic operations like sums and products in algebras). The aim of those two sections is not only to prepare the ground for the study of the sample covariance but also to offer a generic toolbox beyond our present scope; for this reason we maintain as general hypotheses as possible. In the last section we then devise so-called deterministic equivalents for the sample covariance matrix model under study and we eventually illustrate the robustness of our results to both artificial and real datasets. 

\section*{Preamble}

\begin{remark}[Definition of $S$]
As a first remark, note that we abusively define the sample covariance matrix $S$ as $S=\frac1nXX^T$ rather than the conventional $\frac1nXX^T-\frac1n\mathbb{E}X \mathbb{E}X^T$. This choice is not completely marginal in that not subtracting the matrix $\mathbb{E}X \mathbb{E}X^T$ from $S$ in general brings additional complications (because this matrix will in general have unbounded norm as $p,n\to\infty$); our concentration of measure approach however efficiently deals with this term. From a practical standpoint, the removal of $\mathbb{E}X \mathbb{E}X^T$ means that this matrix can be computed, which is in general not the case. Alternatively, centering $S$ by the empirical average $\frac1nX1_n1_n^TX^T$ presumes that the $x_i$ are identically distributed which, as is common in classification applications in machine learning, is also not a desirable assumption.
\end{remark}

In the course of the article, we will be particularly interested in the eigenvalues $l_1,\ldots,l_p$ of $S$. More precisely, we will consider the spectral distribution $F$ of $S$ defined as the following normalized counting measure of the eigenvalues of $S$ (a random probability measure):
\begin{align*}
  F=\frac{1}{p}\sum_{i=1}^p \delta_{l_i},
\end{align*}
where $\delta_x$ is the Dirac measure centered at $x$. As is conventional in large dimensional random matrix theory, we shall retrieve information of $F$ through an approximation of its Stieltjes transform $m_{F}$, defined as:
\begin{align*}
  m_{F}(z) = \int_{w}  \frac{1}{w-z} dF(w)=\frac{1}{p}\tr \left(\left(S - z I_p\right)^{-1}\right),
\end{align*}
where $z\in\mathbb{C}$ belongs to the complementary of the support of $F$. Since the matrix $S$ is nonnegative definite, it suffices to study $m_{F}(z)$ for $z \in \mathbb{R}_-$ to recover $m_F$ by analytic extension. For convenience, rather than working on $\mathbb{R}^-$ we shall consider $m_{F}(-z)$ for $z$ in $\mathbb{R}^+$.

\bigskip

The random matrix $Q_S(z)=\left(S + z I_p\right)^{-1}$, referred to as the resolvent of $S$, will thus be an object of fundamental importance in the remainder. It shall be denoted $Q$ when non-ambiguous. The convenience of the resolvent $Q$ as a cornerstone of random matrix theory analysis is due in part to its simple boundedness properties:
\begin{lemma}\label{lem:controle_Q}
  Given a matrix $R \in \mathcal{M}_{p,n}$, a nonnegative definite symmetric matrix $C \in \mathcal{M}_{p}$ and $z \in \mathbb{R}_+$, we have the following bounds:
  \begin{align*}
    &\left\Vert Q_{C}(z)\right\Vert \leq \frac1z & 
    &\left\Vert Q_{C}(z) C\right\Vert \leq 1& 
    &\left\Vert Q_{\frac1nRR^T}(z) R\right\Vert \leq \frac{\sqrt n}{\sqrt z}. 
  \end{align*}  
\end{lemma}
\begin{proof}
  The upper bound for $\left\Vert Q_C(z)\right\Vert$ follows from the smallest eigenvalue of $C+zI_p$ being larger than $z$. The second result is a consequence of the simple identity $Q_{C}(z) C + zQ_{C}(z) = I_p$ and the fact that $Q_{C}(z)$ is symmetric positive definite. Combining the first two results, we have the bound $$\left\Vert Q_{\frac1nRR^T}(z)\frac1n RR^T Q_{\frac1nRR^T}(z)\right\Vert\leq \frac1z,$$ providing the last result.
\end{proof}

\bigskip

Beyond the eigenvalues of $S$, our interest (also driven by numerous applications in electrical engineering and data processing) will also be on the eigenvectors of $S$. For $U=[u_1,\ldots,u_l]$ an eigen-basis for the eigenspace associated to the multiplicity-$l$ eigenvalue $\lambda$ of $S$, remark that $UU^T=\frac1{2\pi\imath}\oint_{\Gamma_\lambda} Q(-z)dz$ for $\Gamma_\lambda$ a negatively oriented complex contour surrounding $\lambda$ only. As such, beyond studying the trace of $Q$ (and thus the Stieltjes transform of $S$), our interest is also on characterizing $Q$ itself. 

Precisely, we shall study so-called deterministic equivalents for $Q$, the precise definition of which is given in Definition~\ref{def:deterministic equivalent} and that can be described as deterministic matrices $\tilde Q$ verifying $\tr A (Q-\tilde Q) \rightarrow 0$ when $A$ has unit norm (depending on the tightness of the concentration around $\tilde Q$, the norm of $A$ considered will either be the Frobenius norm $\left\Vert A\right\Vert_F=\sqrt{\tr AA^T}$, either the nuclear norm $\left\Vert A\right\Vert_1=\tr (AA^T)^{\frac{1}{2}}$). Note that if we control $\tr A (Q-\tilde Q)$, we also control $u^T(Q-\tilde Q)v=\tr vu^T(Q-\tilde Q)$ for two deterministic vectors $u,v \in \mathbb{R}^p$ with unit norm (in that case $\left\Vert vu^T\right\Vert_F=\left\Vert vu^T\right\Vert_1=\left\Vert u\right\Vert \left\Vert v\right\Vert$). 
In the following lines, we provide an outline for our subsequent development. First, let us note that a naive approach would be to think that $\left(\Sigma + zI_p\right)^{-1}$ might be a deterministic equivalent for $Q$. This turns out to be incorrect under general assumptions. Instead, letting $\tilde Q = \left(\Sigma'  + z I_p\right)^{-1}$ where $\Sigma'$ is some deterministic matrix to determine, we may first compute the difference:
\begin{align*}
   \tilde{Q}-\mathbb{E}Q
  &=\mathbb{E}\left[Q\left(\frac1nXX^T - \Sigma'\right)\tilde{Q}\right] =\frac{1}{n}\sum_{i=1}^n  \mathbb{E}\left[Q (x_ix_i^T - \Sigma')\tilde{Q}\right].
\end{align*}
Here, to go further, we need to make explicit the dependence between $x_i$ and the matrix $Q$ in order to evaluate the expectation of the product $Qx_i$. Let us denote $X_{-i} \in \mathcal{M}_{p,n-1}$ the matrix $X$ deprived of its $i$-th column, which leads us to defining the matrices $S_{-i}=\frac1nX_{-i}X_{-i}^T$ and $Q_{-i}=(S_{X_{-i}}+z I_p)^{-1}$ (which is not $Q_{S_{X_{-i}}}$ as $n$ is not turned in $n-1$). To handle the dependence between $x_i$ and $Q$, we will massively exploit in the paper the classical Schur identities:
\begin{align}\label{eq:lien_q_qj}
  &Q=Q_{-i} -\frac{1}{n}\frac{Q_{-i}x_ix_i^TQ_{-i}}{1+\frac1nx_i^TQ_{-i}x_i}&
  &\text{and}&
  &Qx_i=\frac{Q_{-i}x_i}{1+\frac1nx_i^TQ_{-i}x_i}.
\end{align}
The second inequality allows us to disentangle the relation between $Q$ and $x_i$ in the product $Qx_i$ with a similar but easier to apprehend product $Q_{-i}x_i$ and a factor $1/\left(1+\frac1nx_i^T Q_{-i} x_i\right)$ easily controllable thanks to a first call to concentration inequalities (see subsequently Property~\ref{the:conc_forme_quadrat}). This leads us to:
\begin{align*}
  \tilde{Q}-\mathbb{E}Q
  &=\frac{1}{n}\sum_{i=1}^n  \mathbb{E}\left[Q_{-i} \left(\frac{x_ix_i^T}{1+\frac1nx_i^T Q_{-i}x_i} - \Sigma'\right)\tilde{Q}\right] + \frac{1}{n^2}\sum_{i=1}^n \mathbb{E}\left[Q_{-i}x_ix_i^TQ \Sigma' \tilde Q \right].
\end{align*}
We will see that, due to the supplementary factor $1/n$, the norm of the rightmost random matrix will be negligible compared to that of the other right-hand side matrix. Thus, if one assumes, say, that the random vectors $(x_i)_{1\leq i \leq n}$ follow the same law  (not our general assumption in the article), one would choose naturally $\Sigma' = \mathbb{E}[x_ix_i^T] / (1+\delta)$ where $\delta= \frac{1}{n}\mathbb{E}[x_i^TQ_{-i}x_i]=\frac1n\tr \mathbb{E}[Q_{-i}]\mathbb{E}[x_ix_i^T]$. Then, having established that $\frac1n\mathbb{E}\tr A Q$ (and thus $\frac1n\mathbb{E}\tr A Q_{-i}$) is close to $\frac1n\tr A\tilde Q$, in particular here for $A=\mathbb{E}[x_ix_i^T]$, one may establish an implicit equation for $\delta$ not involving expectations over $Q$ (or $Q_{-i}$). The remaining issue, at the core of our present analysis, is now to find a convenient setting in the modeling of the $x_i$ for which such a choice of $\Sigma'$ would guarantee that $\frac{1}{p}\tr(\mathbb{E}Q-\tilde Q)$ indeed vanishes. That will be related to the concentration of ``chaos-like'' quantities (see \cite[Section 6.1]{Ver17}) such as $x_i^TQ_{-i}x_i$ but also of quantities such as $\tr AQ$ or $u^TQv$ as we shall see.

In the first part of the article, we will show that a comfortable approach is to structure our results as an outgrowth of the concentration of measure phenomenon. To give a brief insight into our approach, let us give the original theorem of the theory that we owe to Paul Pierre Levy in the beginning of the twentieth century and concerns the concentration of the uniform distribution $\sigma_p$ on the sphere $\mathbb S^{p}$. To demystify the result we mention that it is closely linked to the isoperimetrical inequality. 

\begin{theorem}[Normal concentration of $\sigma_p$, \protect{\cite[Theorem 2.3]{Led01}}]\label{the:concentration_vecteur_spherique}
  Given a degree $p \in \mathbb{N}$ and a random vector $Z \sim \sigma_p$, for any $1$-Lipschitz function $f : \mathbb R^{p+1} \rightarrow \mathbb R$, we have the inequality~:
  \begin{align}\label{eq:concentraion_sphere}
    \mathbb{P}\left(\left\vert f(Z) - m_f\right\vert\geq t\right)\leq 2 e^{-(p-1)t^2/2},
  \end{align}
  where $m_f$ is a median of $f(Z)$ verifying by definition $\mathbb{P}(f(Z)\geq m_f), \mathbb{P}(f(Z)\leq m_f) \geq \frac{1}{2}$.
\end{theorem}
There exist plenty of other distributions that verify this inequality like the uniform distribution on the ball, on $[0,1]^n$, or the Gaussian distribution $\mathcal N(0,I_p)$ that are presented in \cite{Led01} ; more generally, for Riemannian manifolds, the concentration can be interpreted as a positive lower bound on the Ricci curvature (see Gromov appendix in \cite{Mil86} or \cite[Theorem 2.4]{Led01}).
Our study will not evoke the design of such distributions and the validity of their concentration since this work has already been treated in the past; we will rather directly assume a similar concentration inequality to \eqref{eq:concentraion_sphere} for the data matrix $X$ and then, with the tools developed in Sections~\ref{sse:conc_variable_al} and~\ref{sse:conc_vect_al}, we will infer the concentration of the Stieltjes transform $m_{F}(z)$ for any $z > 0$ and devise a good estimator of it.
To prepare the applications to come, the dimension $p$ must be thought to be quasi-asymptotic. In that sense the tightening of the concentration when the metric diameter of the distribution stays constant is a remarkable specificity of the concentration phenomenon ; this is furthermore a necessary condition, required for our study in Section~\ref{sec:emp_cov}. 
The concentration inequality verified on the sphere will structure our paper in a way that we will try to express our result with that form as often as possible to pursue Levy's idea; we will thus choose the short notation $f(Z) \in m_f \pm 2e^{-(p-1)\,\cdot \,^2/2}$ to express it in a simple way.


As Milman has advocated from the beginning of the seventies (\cite{Mil86}), the concentration of a random object appears to be an essential feature leading immediately to a lot of implications and controls on the object. We do not present here the historical introduction of concentration of the measure. Rather than the geometric, distribution-oriented approach, we directly adopt a probabilistic point of view on the notion.
That does not mean that the presentation of the theory will be incomplete. On the contrary, we display here, almost always with their proofs, all the important theorems and propositions necessary to apprehend the theory and set rigorously the study of Section~\ref{sec:emp_cov} about the sample covariance. 
The usual approach to the concentration of measure is to start with geometric inequalities ruling high dimensional Banach spaces and then track from these powerful results some probabilistic properties on the real functionals, the ``observable world''. We propose here a reversed approach where we start from the probabilistic results on $\mathbb{R}$ which offer us some interesting reasoning schemes all the same. 
Then, once the reader convinced by the direct computation improvements offered by the theory, we perform the fundamental step consisting in considering high dimensional concentration properties.
Many of the results are derived from the complete presentation of the theory made by Ledoux in \cite{Led01}.

\section{Concentration of a random variable}\label{sse:conc_variable_al}

\subsection{Definition and first Properties}

\begin{definition}[Concentration function]\label{def:fonction_de_concentration}
  Any non-increasing and left continuous function $\alpha : \mathbb{R}_+ \rightarrow [0,1]$, is called a concentration function.
\end{definition}
Given a random variable $Z$ and a concentration function $\alpha$, we choose first to express the \textit{$\alpha$-concentration} of $Z$ through the introduction of an independent copy $Z'$. 
\begin{definition}[Concentration of a random variable]\label{def:conc_variable_aleatoire} 
  The random variable $Z$ is said to be $\alpha$\emph{-concentrated}, and we write $Z \propto \alpha$, iff for any independent copy $Z'$~:
  \begin{align}\label{eq:conc_diff}
    \forall t >0 \ : \  \ \mathbb{P}\left( \left\vert Z-Z'\right\vert \geq t\right) \leq \alpha(t).
  \end{align}
\end{definition}
 In Definition~\ref{def:fonction_de_concentration}, the different properties required for the concentration function ($\alpha\in [0,1]$, non-increasing, left-continuous) are important features of any function $t \mapsto \mathbb{P}\left(\left\vert X\right\vert \geq t\right)$ when $X$ is a random variable. In particular, if we often deal with concentration functions $\alpha$ that take values outside of $(1,+\infty)$, it is implicitly understood that we consider in the calculus $t \mapsto \min\{\alpha(t),1\}$ instead of $\alpha$.

In Definition~\ref{def:conc_variable_aleatoire}, we see that $\alpha$ basically limits the variations of the random variable $Z$. The faster $\alpha$ decreases, the lower the variations of $Z$.
Instinctively one might hope that this limitation of the variation would be equivalent to a concentration around some central quantity that will be called a \textit{pivot} of the concentration in the following sense~:
\begin{definition}[Concentration around a pivot]\label{def:concentration_autour}
  Given $a\in \mathbb{R}$, the random variable $Z$ is said to be  $\alpha$\emph{-concentrated around the pivot $a$}, and we write $Z \in a \pm \alpha$, iff :
  \begin{align*}
    \forall t >0 \ : \ \mathbb{P}\left( \left\vert Z -a\right\vert \geq t\right) \leq \alpha(t).
  \end{align*}
\end{definition}
We call the parameter $a$ a \textit{pivot} because the concentration around $a$ provides similar concentration around other values close to $a$. Given $\theta >0$, let us denote $\tau_{\theta}$ the operator defined on the set of concentration functions that verifies for any concentration function $\alpha$~:
\begin{align*}
  \tau_{\theta} \cdot \alpha(t) = \left\{\begin{aligned}
                                      &1 \text{  if  } t\leq \theta \\
                                      &\alpha(t-\theta) \text{  if  } t > \theta.
                                     \end{aligned}\right.
\end{align*}
Then we have the simple Lemma~:
\begin{lemma}\label{lem:concentration_autour_des_valeurs_proches_du_pivot}
  Given a random variable $Z \in \mathbb{R}$, a concentration function $ \alpha$, two real numbers $a,b\in \mathbb{R}$ and $\theta >0$, we have the implication~:
  \begin{align*}
    \left\{\begin{aligned}
        &Z \in a \pm \alpha\\
        & \left\Vert a-b\right\Vert\leq \theta
    \end{aligned}\right.&  
    &\Longrightarrow&
    &Z \in b \pm \tau_\theta \cdot \alpha.
  \end{align*} 
\end{lemma}

At first sight, there exists no pivot $a\in \mathbb{R}$ such that Definition~\ref{def:conc_variable_aleatoire} and Definition~\ref{def:concentration_autour} are equivalent. We can however find an interesting relation considering the case $a=m_Z$ where $m_Z$ is a median of $Z$, verifying by definition $\mathbb{P}(Z\geq m_Z)\geq 1/2$ and $\mathbb{P}(Z\leq m_Z)\geq 1/2$. 
\begin{proposition}[Concentration around the median, from \protect{\cite[Corollary 1.5]{Led01}}]\label{pro:conc_med}
  Given a random variable $Z$, a median $m_Z$ of $Z$ and a concentration function $\alpha$, we have the implications :
  \begin{align*}
    Z\propto \alpha&
    &\Longrightarrow&
    &Z \in m_Z \pm 2\alpha&
    &\Longrightarrow&
    &Z\propto 4\alpha( \, \cdot \, /2)
  \end{align*}
  where $\alpha(\, \cdot \, /2)$ is defined as being the function $t \mapsto \alpha(t/2)$.
\end{proposition}
We see here that if $Z$ is $\alpha$-concentrated, the tail of $Z$, i.e., the behavior of $Z$ far from the median, is closely linked to the decreasing speed of $\alpha$.
\begin{proof}
   We need to consider the fact that $\mathbb{P}\left( Z = m_Z\right)=\epsilon$ may be non zero. Therefore there exist $\epsilon_1, \epsilon_2$ such that $\epsilon=\epsilon_1 + \epsilon_2$ and~:
  \begin{align*}
    \mathbb{P}\left( Z < m_Z\right)=\frac{1}{2}-\epsilon_1&
    &\mathbb{P}\left( Z > m_Z\right)=\frac{1}{2}-\epsilon_2.
  \end{align*}  
  Let us take $t>0$. The first result follows from the inequalities~:
  \begin{align*}
    \mathbb{P}\left(\left\vert Z-Z'\right\vert \geq t\right)
    &=\mathbb{P}\left(\left\vert Z-Z'\right\vert \geq t, Z' < m_Z\right)+\mathbb{P}\left(\left\vert Z-Z'\right\vert \geq t, Z' > m_Z\right) \\
    &\hspace{ 0.5cm}+\mathbb{P}\left(\left\vert Z-m_Z\right\vert \geq t, Z' = m_Z\right) \\
    &\geq\left( 1/2-\epsilon_1\right)\mathbb{P}\left(Z \geq t + Z' \ | \ Z' < m_Z\right)  \\
    & \hspace{0.5cm}+\left( 1/2-\epsilon_2\right)\mathbb{P}\left(Z' \geq t + Z \ | \ Z' > m_Z\right)  \\
    & \hspace{0.5cm}+\epsilon \ \mathbb{P}\left(\left\vert  Z - m_Z\right\vert  \geq t\right) \\
    &\geq \left(1/2-\epsilon\right) \left(  \mathbb{P}\left(Z \geq t + m_Z \right) + \mathbb{P}\left(m_Z \geq t + Z \right)\right) \\
    &\hspace{0.5cm}+\epsilon \ \mathbb{P}\left(\left\vert   Z - m_Z\right\vert  \geq t\right) \\
    &= \frac{1}{2}\,\mathbb{P}\left(\left\vert Z-m_Z\right\vert \geq t  \right) .
  \end{align*}
  The second result is true for any real $a\in \mathbb{R}$ replacing $m_Z$~:
  \begin{align*}
    \mathbb{P}\left( \left\vert Z-Z'\right\vert \geq t\right)
    &\leq \mathbb{P}\left( \left\vert Z-m_Z + m_Z-Z'\right\vert \geq t\right) \\
    &\leq 2\mathbb{P}\left( \left\vert Z-m_Z\right\vert \geq t/2\right).
  \end{align*}
\end{proof} 
The reason why Definition~\ref{def:conc_variable_aleatoire} was presented firstly and thus given the importance of the naturally underlying definition, even though Definition~\ref{def:concentration_autour} might seem more intuitive, lies in its immediate compatibility to the composition with any Lipschitz function and more generally with any uniformly continuous function. The uniform continuity of a function is sized by its \textit{modulus of continuity}. Although we presently work on $\mathbb{R}$, we give a general definition of the continuity modulus between two normed vector spaces because it will be useful in the next sections.

\begin{definition}[Uniform continuity]\label{def:fonction_holderienne}
  Any non decreasing function $\omega : \mathbb{R}_+ \rightarrow \mathbb R_+$ continuous and null in $0$ is called a modulus of continuity. Given two normed vector space $(E,\left\Vert \cdot \right\Vert_E)$ and $(F,\left\Vert \cdot\right\Vert_F)$, a function $f : E \rightarrow F$ is said to be continuous under the modulus of continuity $\omega$ if~: 
  \begin{align*}
    \forall x,y \in E \ \ : \ \ \ \left\Vert f(x) - f(y)\right\Vert_F \leq  \omega(\left\Vert x- y\right\Vert_E).
  \end{align*}
  When $\forall t >0$, $\omega(t)=\lambda t^\nu$ for $\lambda>0$ and $\nu \in (0,1]$, we say that $f$ is $(\lambda,\nu)$-H\"older continuous, $\nu$ is called the \emph{H\"older exponent} and $\lambda$ is called the \emph{Lipschitz coefficient} ; indeed, if $\nu=1$, $f$ is said to be \emph{$\lambda$-Lipschitz}. Of course a function admits a modulus of continuity iff it is uniformly continuous.
\end{definition}
\begin{lemma}\label{lem:conc_transfo_lipschitz}
  Let us consider a random variable $Z$, a concentration function $\alpha$, and a function $f : \mathbb{R} \rightarrow \mathbb{R}$ continuous under a modulus of continuity $\omega$.
  We allow ourselves to write $\omega^{-1}$ the pseudo inverse of $\omega$ defined (for $\omega$ bijective or not) as : $\omega^{-1}(w) = \inf \{t,\omega(t)\geq w \} $. We have the implication~:
  \begin{align*}
    Z \propto \alpha \  \Longrightarrow \ f(Z) \propto  \alpha\left(\omega^{-1}\left(\, \cdot \, \right)\right).
  \end{align*}  
\end{lemma}
\begin{proof}
  It is straightforward to write~:
  \begin{align*}
    \mathbb{P}\left(\left\vert f(Z) -f(Z')\right\vert\geq t\right)
    &\leq\mathbb{P}\left(\omega (\vert Z - Z'\vert)\geq t\right) \\
    &\leq\mathbb{P}\left(\vert Z - Z'\vert\geq \omega^{-1}(t)\right)
    \leq \alpha \left(\omega^{-1}\left(t\right)\right),
  \end{align*}
  since for any $w,t>0$, $\omega(t) \geq w \Longrightarrow t \geq \omega^{-1}(w)$ by definition of the pseudo inverse. 
\end{proof}
\begin{remark}\label{rem:incomp_def2_lipsch}
  Definition~\ref{def:concentration_autour} is not so compatible with the $\omega$-continuous transformations. Indeed, the simple developments that one finds in the proof above cannot be performed when strictly assuming that $Z\in m_Z \pm \alpha$. In this case, one would rather combine Lemma~\ref{lem:conc_transfo_lipschitz} with Proposition~\ref{pro:conc_med} to find~:
  \begin{align*}
    Z \in m_Z \pm \alpha \  \Longrightarrow \ f(Z) \in m_{f_Z} \pm 2\alpha\left(\frac{1}{2}\omega^{-1}\left(\, \cdot \, /2 \right)\right).
  \end{align*}  
\end{remark}

The stability with respect to $\omega$-continuous functions reflects the fact that a $\omega$-continuous function contains the spreading of the distribution up to the modulus of continuity. 

One property that we could expect from $\alpha$-concentration is a stability towards the sum. Up to a multiplying factor of $2$, this is the case :
\begin{lemma}\label{lem:conc_somme}
  Given two random variable $Z_1$ and $Z_2$ and two concentration functions $\alpha,\beta$, we have the implication~:
  \begin{align*}
    Z_1 \propto \alpha \ \  \text{and} \ \ Z_2 \propto \beta& 
    &\Longrightarrow& 
    &Z_1 + Z_2 \propto \alpha(\, \cdot \, /2) + \beta(\, \cdot \, /2).
  \end{align*}  
\end{lemma} 
\begin{proof}
  Recall that $Z_1'$ and $Z_2'$ are two independent copies respectively of $Z_1$ and $Z_2$. There is no reasons for $Z_1'$ to be independent of $Z_2$ (resp., $Z_2'$ of $Z_1$). The idea is to decompose the threshold $t>0$ in $\frac{t}{2}+\frac{t}{2}$ ~:
  \begin{align*}
    &\mathbb{P}\left( \left\vert Z_1+Z_2 - Z_1' - Z_2'\right\vert \geq t\right) \\
    & \hspace{1cm}\leq \mathbb{P}\left( \left\vert Z_1 - Z_1'\right\vert  \geq \frac t2 \right)
    +\mathbb{P}\left( \left\vert Z_2 - Z_2'\right\vert\geq \frac t2 \right)
     \ \leq \ \alpha \left(\frac t2\right) +\beta \left(\frac t2\right).
  \end{align*}
\end{proof}
\begin{remark}\label{rem:comp_conc_aut_somme}
  This time the idea of the demonstration works the same with the setting of Definition~\ref{def:concentration_autour}, and we have~:
  \begin{align*}
    Z_1 \in a \pm \alpha \ \  \text{and} \ \ Z_2 \in b \pm \beta&
    &\Longrightarrow&
    &Z_1 + Z_2 \in a + b \pm \alpha(\, \cdot \, /2) + \beta(\, \cdot \, /2).
  \end{align*}
\end{remark}
In a first approach, $\alpha$-concentration performs badly with the product, as it requires a bound on both random variables involved which greatly reduces the number of possible applications.
\begin{lemma}\label{lem:conc_produit}
  Given two bounded random variable $Z_1$ and $Z_2$ such that $\left\vert Z_1\right\vert\leq K_1$ and $\left\vert Z_2\right\vert\leq K_2$ and two concentration functions $\alpha, \beta$~:
  \begin{align*}
    Z_1 \propto \alpha \ \ \text{and} \ \ Z_2 \propto \beta&
    &\Longrightarrow&
    &Z_1Z_2 \propto  \alpha \left(\, \frac{\cdot}{2K_2} \,\right) +  \beta \left(\, \frac{\cdot}{2K_1} \,\right).
  \end{align*}  
\end{lemma}
\begin{proof}
  Given $t>0$~:

  \begin{align*}
    \mathbb{P}\left(\left\vert Z_1Z_2 -Z_1'Z_2'\right\vert\geq t\right)
    &\leq \mathbb{P}\left(\left\vert Z_1(Z_2 -Z_2')\right\vert\geq \frac{t}{2}\right)+\mathbb{P}\left(\left\vert (Z_1 -Z_1')Z_2'\right\vert\geq \frac{t}{2}\right)  \\
    &\leq \mathbb{P}\left(\left\vert Z_2 -Z_2'\right\vert\geq \frac{t}{2K_1}\right)+\mathbb{P}\left(\left\vert Z_1 -Z_1'\right\vert\geq \frac{t}{2K_2}\right).
  \end{align*}
\end{proof}
Actually, the setting of Definition~\ref{def:concentration_autour} is more convenient here than the setting of Definition~\ref{def:conc_variable_aleatoire} because it allows us to only require one random variable to be bounded~:
\begin{lemma}\label{lem:conc_autour_prod}
  Given two random variables $Z_1$ and $Z_2$ such that $\left\vert Z_1\right\vert\leq K_1$, two pivot $a,b\in \mathbb R$ and two concentration functions $\alpha, \beta$, if $b\neq 0$ one has the implication~:
  \begin{align*}
    Z_1 \in a \pm \alpha \ \ \text{and} \ \ Z_2 \in b \pm\beta&
    &\Longrightarrow&
    &Z_1Z_2 \in  ab \pm \alpha \left(\, \frac{\cdot}{2 \left\vert b\right\vert} \,\right) +  \beta \left(\, \frac{\cdot}{2K_1} \,\right),
  \end{align*}
  if $b=0$, then the concentration of $Z_1$ and $Z_2$ implies that $Z_1Z_2 \in 0 \pm \beta (\frac{\cdot}{K_1})$. 
\end{lemma}

We may even go further and dispense with the bounding hypothesis to get in this case a slightly more complicated concentration form~:
\begin{proposition}\label{pro:conc_autour_prod_ss_borne}
  Given two random variable $Z_1$ and $Z_2$, two pivot $a,b \in \mathbb R$ and two concentration functions $\alpha, \beta$ such that $Z_1 \in a \pm \alpha$ and $Z_2 \in b \pm\beta$, if $a\neq 0$ and $b\neq 0$, then $Z_1Z_2$ is concentrated around $ab$ with ~: 
  \begin{align*}
    Z_1Z_2 \in  ab \pm \alpha \left(\, \sqrt{\frac{\cdot}{3}} \,\right)+\alpha \left(\, \frac{\cdot}{3 \left\vert b\right\vert} \,\right) + \beta \left(\, \sqrt{\frac{\cdot}{3}} \,\right)+  \beta \left(\, \frac{\cdot}{3 \left\vert a\right\vert} \,\right).
  \end{align*}
  If $a=0$ and $b\neq 0$, we get $Z_1Z_2 \in  0 \pm \alpha \left(\, \sqrt{\frac{\cdot}{2}} \,\right)+\alpha \left(\, \frac{\cdot}{2 \left\vert b\right\vert} \,\right) + \beta \left(\, \sqrt{\frac{\cdot}{2}} \,\right)$ and if $a=b=0$, $Z_1Z_2 \in  0 \pm \alpha \left(\, \sqrt{\, \cdot \,}\right) +  \beta \left(\, \sqrt{\, \cdot \,} \,\right)$.
\end{proposition}
\begin{proof}
  We just prove the result for $a,b \neq 0$ since the other cases are simpler. The idea is to use the algebraic identity~: $xy-ab=(x-a)(y-b) + a(y-b) + b(x-a)$ and the implication $xy \geq t \Rightarrow x\geq \sqrt{t} \ \text{or} \ y \geq \sqrt{t} $ (for $x,y,t\geq 0$):
  \begin{align*}
    \mathbb{P}\left(\left\vert Z_1Z_2 - ab\right\vert \geq t\right)
    &\leq \mathbb{P}\left(\left\vert Z_1 - a\right\vert \geq \sqrt{\frac{t}{3}}\right) + \mathbb{P}\left(\left\vert Z_2 - b\right\vert \geq \sqrt{\frac{t}{3}}\right) \\
    &\hspace{0.5cm}+\mathbb{P}\left(\left\vert Z_1 - a\right\vert\left\vert b\right\vert \geq \frac{t}{3}\right) + \mathbb{P}\left(\left\vert Z_2 - b\right\vert \left\vert a\right\vert \geq \frac{t}{3}\right).
  \end{align*}
\end{proof}

In the case of the square of a random variable or even an integer power of any size, the concentration given by Proposition~\ref{pro:conc_autour_prod_ss_borne} can simplify.
\begin{proposition}\label{pro:conc_puiss_var_al_conc_autour}
  Given $m\in \mathbb{N}$ and a random variable $Z \in a \pm \alpha$ with $a\in \mathbb{R}$ and $\alpha$, one has the concentration~:
  \begin{align*}
    Z^m \in a^m \pm \alpha\left(\frac{\cdot}{2^m\left\vert a\right\vert^{m-1}}\right) + \alpha\left(\left(\frac{\cdot}{2}\right)^{\frac{1}{m}}\right).
  \end{align*}
\end{proposition}
\begin{proof}
  Let us employ the algebraic identity~:
  \begin{align*}
    Z^m=\left(Z-a +a\right)^m = \sum_{i=0}^m \binom{m}{i } a^{m-i} (Z-a)^i= a^m + a^m \sum_{i=1}^m \binom{m}{i } \left(\frac{Z-a}{a}\right)^i.
  \end{align*}
  If $\left\vert \frac{Z-a}{a}\right\vert \leq 1$, for any $i\in \{1,\ldots m\}$, $\left\vert \frac{Z-a}{a}\right\vert^i\leq \left\vert \frac{Z-a}{a}\right\vert$ and conversely, if $\left\vert \frac{Z-a}{a}\right\vert\geq 1$, then $\left\vert \frac{Z-a}{a}\right\vert^i \leq \left\vert \frac{Z-a}{a}\right\vert^m$. This entails~:
  \begin{align*}
    \left\vert Z^m -a^m\right\vert \leq (2 \left\vert a\right\vert)^m \left(\left\vert \frac{Z-a}{a}\right\vert + \left\vert \frac{Z-a}{a}\right\vert^m\right)
  \end{align*}
  and therefore~:
  \begin{align*}
    \mathbb{P}\left(\left\vert Z^m-a^m\right\vert\geq t\right)
    \leq \mathbb{P}\left(\left\vert Z-a\right\vert\geq \frac{t}{2^m\left\vert a\right\vert^{m-1}}\right) + \mathbb{P}\left(\left\vert Z-a\right\vert\geq \left(\frac{t}{2}\right)^{\frac{1}{m}}\right).
  \end{align*}
\end{proof}


We conclude this section by setting the continuity of the concentration property. We adopt the classical formalism for the convergence of a random variable (or vector). 
\begin{definition}\label{def:converg_loi}
  We say that a sequence of random variables (or random vectors) $Z_n$ converges in law (or ``in distribution'' or ``weakly'') to $Z$ if for any real valued continuous function $f$ with compact support~:
  \begin{align*}
    \lim_{n\rightarrow \infty}\mathbb{E} [f(Z_n)]= \mathbb{ E} [f(Z)].
  \end{align*}  
\end{definition}
Although we rather find in the literature the definition mentioning \textit{continuous and bounded} functions, this equivalent definition relying on the class of continuous functions with compact support is more adapted to our needs (see Proposition~\ref{pro:conc_faible_trAQ}). We start with the preliminary lemma~:
\begin{lemma}[\cite{Ouvr09}, Proposition 14.17]\label{lem:conv_loi_fonct_repartition}
  Let us consider a sequence of random variable $Z_n$, $n\in \mathbb{N}$, and a random variable $Z$ with cumulative distribution functions respectively noted $F_{Z_n}$, $n\in \mathbb{N}$, and $F_Z$. The sequence $(Z_n)_{n\geq 0}$ converges in law to $Z$ iff for any $t\in \mathbb{ R}$ such that $F_Z$ is continuous on $t$, $(F_{Z_n}(t))_{n\geq 0}$ converges to $F_Z(t)$.
\end{lemma}
\begin{proposition}\label{pro:converg_loi}
  Consider a sequence of random variables $(Z_n)_{n\geq 0}$ that converges in law to a random variable $Z$, a sequence of pivot $(a_n)_{n\geq 0}$ converging to a pivot $a\in \mathbb{  R}$ and a sequence of concentration functions $(\alpha_n)_{n\geq 0}$ that point-wise converges to a continuous concentration function $\alpha$. If we suppose that, for any $n \in \mathbb{N}$, $Z_n\in a_n \pm \alpha_n$ then $Z \in a \pm \alpha$.
\end{proposition}
\begin{proof}
  For any $n\in \mathbb{N}$, let us note $Y_n=\left\vert Z_n-a_n\right\vert$ and $Y= \left\vert Z-a\right\vert$. We wish to show first that $(Y_n)_{n\geq 0}$ converges in law to $Y$. Let us consider for that purpose a continuous function $f: \mathbb{R} \rightarrow \mathbb{R}$ with compact support $S$ and $\varepsilon>0$. We know from the Heine-Cantor theorem that $f$ is uniformly continuous, therefore there exists $\eta>0$ such that~:
  \begin{align*}
    \left\vert  x-y\right\vert  \leq \eta \ \  \ \Longrightarrow \left\vert   f(x)-f(y)\right\vert\leq \frac{ \varepsilon}{2}.
  \end{align*}
  Moreover, since $\lim_{n\rightarrow \infty} a_n = a$, there exists $n_0 >0$ such that for any $n \geq n_0$, $ \left\vert a_n - a\right\vert \leq \eta$. Eventually, if we introduce the function $g : x \rightarrow \left\vert x-a\right\vert$, we know that $f \circ g$ has a compact support $S'=a+S \cup a-S$, and thus there exists $n_1 \in \mathbb{ N}$ such that if $n\geq n_1$~:
  \begin{align*}
    \left\vert \mathbb{E} [f (\left\vert  Z_n-a\right\vert)]- \mathbb{  E} [f(\left\vert  Z-a\right\vert)]\right\vert=\left\vert \mathbb{E} [f\circ g(Z_n)]- \mathbb{ E} [f \circ g(Z)]\right\vert \leq \frac{\varepsilon}{2}.
  \end{align*}
  Therefore if we consider $n \geq \max(n_0,n_1)$~:
  \begin{align*}
    \left\vert  \mathbb{E} [f(Y_n)] - \mathbb{  E} [f(Y)]\right\vert  \ \ 
    &\leq\left\vert   \mathbb{E} [f(|Z_n-a_n|)-f(|Z_n-a|)]\right\vert \\  
    &\hspace{0.8cm}+ \left\vert \mathbb{E}[f (\left\vert  Z_n-a\right\vert)- f(\left\vert   Z-a\right\vert)]\right\vert 
    \ \ \leq \ \varepsilon.
  \end{align*}
  Then we know from Lemma~\ref{lem:conv_loi_fonct_repartition} that for any $t$ such that the cumulative distribution function $F_Y=\mathbb{P}\left(\left\vert Z-a\right\vert\leq \, \cdot \,\right)  $ is continuous around $t$~:
  \begin{align*}
    \mathbb{P}\left(\left\vert Z-a\right\vert\geq t\right)  = \lim_{n\rightarrow \infty} \mathbb{P}\left(\left\vert Z_n-a_n\right\vert\geq t\right) \leq \lim_{n \rightarrow \infty} \alpha_n(t) = \alpha(t).
  \end{align*}  
  Since $\alpha$ is continuous and $t \rightarrow \mathbb{P}\left(\left\vert Z-a\right\vert\geq t\right)$ is decreasing, we recover the preceding inequality for any $t>0$, and $Z \in a \pm \alpha$.
\end{proof} 
In the setting of Definition~\ref{def:conc_variable_aleatoire}, we can show the continuity of the concentration the same way introducing this time the random variables $Y_n=\left\vert Z_n-Z_n'\right\vert$ ($Z_n'$ being a sequence of independent copies of $Z_n$). We present the next proposition without proof.
\begin{proposition}\label{pro:continuite_conc_def_1}
  In the setting of Proposition~\ref{pro:converg_loi}, if for any $n\in \mathbb{  N}$ $Z_n \propto \alpha_n$, then $Z \propto \alpha$.
\end{proposition}
\subsection{Exponential concentration}\label{sse:exponential_conc}

In \cite{Led01}, Ledoux defines a random variable as \textit{normally concentrated} when the concentration function is of the form $Z\propto C e^{-(\, \cdot \, /\sigma)^2}$ for two given constants $C\geq 1$ and $\sigma>0$. The form of the concentration function has no real importance on small dimensions since the result ensues from a mere bound on the Gaussian Q-function; it is however surprising that this form naturally appears for the uniform distribution on the sphere (see Theorem~\ref{the:concentration_vecteur_spherique}) and others in high dimensions.

In order to present a general picture that will be helpful later (when dealing with products of random variables), let us include the general case of \textit{$q$-exponential concentrations} that we present with the formalism of Definition~\ref{def:concentration_autour}. 
\begin{definition}\label{def:concentration_exp}
  Given $q>0$, a random variable $Z$ is said to be $q$-\emph{exponentially concentrated} with \emph{head parameter} $C\geq 1$ and \emph{tail parameter} $\sigma>0$ iff there exists a pivot $a\in \mathbb{R}$ such that $Z \in a \pm C e^{-( \cdot / \sigma)^q}$. 
\end{definition}

\begin{example}\label{exe:conc_variable_gaussienne}
  A random variable $Z$ following a Gaussian distribution with zero mean (i.e. zero median) and unit variance is $2$-exponentially concentrated with a tail parameter equal to $\sqrt{2}$ : $Z \in 0 \pm 2 e^{-(\cdot)^2/2}$.
\end{example}

In practice, the random variables will depend on a random vector whose dimension, say $p$, tends to infinity, the tail parameter is then a function of $p$ that represents the \textit{asymptotic} speed of concentration since it has the same order as the standard deviation of $Z$.
If the $q$-exponential concentration functions are employed reasonably, the asymptotic information can be transmitted from the head parameter to the tail parameter so that the head parameter would stay mainly uninformative and close to $1$.
For that purpose, the following lemma gives us an easy way to bound a given concentration to a close one with a head parameter equal to $e$. 

\begin{lemma}\label{lem:head_parameter_to_tail}
  Given $x,q >0 $ and $ C\geq e$, we have the inequality~:
  $$\min(1,C e^{-x}) \leq e e^{-x/2\log(C)}.$$
\end{lemma}
\begin{proof}
    If $x \leq 2\log(C)$ the inequality is clear, and if $x\geq 2\log(C)\geq 2$, we deduce the result of the lemma from the equivalence 
    \begin{align*}
        \log(C)-x \leq - \frac{x}{2\log(C)}&
        &\Longleftrightarrow&
        &x \geq \frac{2\log(C)^2}{2\log(C) -1},
    \end{align*}
    since $\frac{2\log(C)^2}{2\log(C) -1} \leq \frac{2 \log(C)^2}{\log(C)} \leq x$.
\end{proof}

To place ourselves under the hypotheses of Lemma~\ref{lem:head_parameter_to_tail} and as it appears rather convenient in several propositions below we will suppose from now on that the tail parameter $C$ is greater than $e$.

Exponential concentrations offer simple expressions of the concentration through shifting the pivot thanks to the following lemma.
\begin{lemma}\label{lem:pivotement_de_la_concentration_avec_concentration_exponentielle}
  Given the parameters $C \geq e$ and $q,\sigma,\theta>0$~:
  \begin{align*}
     \forall t>0 \ : \ \ \tau_\theta \cdot C e^{-(\cdot \,/\sigma)^q} \leq \max\left(e^{(\theta/\sigma)^q},C\right) e^{-(\cdot \, /2\sigma)^q}.
   \end{align*} 
\end{lemma}
\begin{proof}
  When $t\geq 2 \theta$, the increasing behavior of $t \mapsto t^q$ ensures that $\left(t-\theta\right)^q\geq (t/2)^q$ which implies the functional inequality on $[2\theta, \infty)$. When $t \leq 2 \theta$:
  \begin{align*}
    \max\left(e^{(\theta/\sigma)^q},C\right) e^{-(t \, /2\sigma)^q}\geq e^{(\theta/\sigma)^q} e^{-(2\theta \, /2\sigma)^q}\geq 1 \geq \tau_\theta \cdot C e^{-(\cdot \,/\sigma)^q},
  \end{align*} 
  (by definition of $\tau_\theta$).
\end{proof}
\color{black}
This lemma, combined with lemma~\ref{lem:concentration_autour_des_valeurs_proches_du_pivot}, clarifies the notion of \textit{tail parameters}. The next corollary shows that it can be seen as the diameter of a ``black hole'' centered around any pivot of the concentration, in a sense that each value inside this ``black hole'' can be considered as a satisfactory pivot for the concentration ; this will be called later, in the case of random vectors, the \textit{observable diameter} of the distribution, following Gromov terminology \cite{Gro79}.

\begin{corollary}\label{cor:sens_param_de_queue}
  Given $C\geq e$ and three positive parameters $\sigma,\lambda,q>0$, two real $a$ and $b$ such that $\left\vert a-b\right\vert  \leq \lambda \sigma$ and a random variable $Z$, one has the implication~:
  \begin{align*}
    Z \in a \pm C e^{-( \cdot /\sigma)^q}&
    &\Longrightarrow&
    &Z \in b \pm C' \exp\left(-\left(\frac{\, \cdot \, }{2\sigma}\right)^q\right),
  \end{align*}  
  where $C'=\max(C,e^{\lambda^q})$.
\end{corollary}
Note that the interesting aspect of the result is the independence of the head parameter $C'=\max(C,e^{\lambda^q})$ to the tail parameter $\sigma$. Moreover the tail parameter stays unmodified when $q<1$. 

We now have all the elements to show that, due to the high concentration of exponentially concentrated random vectors, every median plays a pivotal role among the different constants that can localize the concentration.

\begin{proposition}[\protect{\cite[Proposition~1.8]{Led01}}]\label{pro:mediane_pivot}
  Given a random variable $Z$, and a median $m_Z$ of $Z$, if we suppose that $Z \in a \pm Ce^{-(\,\cdot \, / \sigma)^q}$ for a pivot $a\in \mathbb{R}$, then~:
  \begin{align*}
    Z  \in m_Z \pm 2C \exp \left(-\left(\frac{\cdot}{2\sigma}\right)^q\right).
  \end{align*}

\end{proposition}

\begin{proof}
  For some $\varepsilon>0$, we choose $t_0>\sigma\left(\log(2C)+\varepsilon\right)^{1/q}$. We know that $\mathbb{P}\left( \left\vert Z -a\right\vert \geq t_0\right) < \frac{1}{2}$ and consequently $\left\vert a-m_Z\right\vert \leq t_0$. Indeed if we suppose that $m_Z \geq a + t_0$, then~:
  \begin{align*}
    1/2 \leq \mathbb{P}\left( Z  \geq m_Z \right)\leq \mathbb{P}\left( Z - a\geq t_0 \right) \leq \mathbb{P}\left( \left\vert Z - a\right\vert\geq t_0 \right) ,
   \end{align*} 
  and we get the same absurd result if we suppose that $a \geq m_Z + t_0$. We can thus conclude thanks to Corollary~\ref{cor:sens_param_de_queue} (with $C'=\max(C,\exp(\frac{t_0^q}{\sigma^q}))=2Ce^\varepsilon$), letting $\varepsilon$ tend to zero.

\end{proof} 
The tails of $q$-exponentially concentrated random variables can be controlled rather easily and roughly thanks to the next lemma that is based on the same simple mathematical inequalities that lead to Corollary~\ref{cor:sens_param_de_queue}.
\begin{lemma}\label{lem:conentration_superieure}
  Given a random variable $Z$, two parameters $C\geq e$ and $\sigma>0$ an exponent $q>0$ and a pivot $a\in \mathbb{R}$, if $Z \in a \pm C e^{-(\, \cdot\, /\sigma)^q}$ then~:
  \begin{align*}
    \forall t \geq 2 \left\vert a\right\vert: \ \mathbb{P}\left( \left\vert Z\right\vert\geq t\right) \leq C e^{-(t/2\sigma)^q}
  \end{align*}
  (for any $t> 0$, $\mathbb{P}\left( \left\vert Z\right\vert\geq t\right) \leq C e^{(|a|/\sigma)^q}e^{-(t/2\sigma)^q}$).
\end{lemma}


Very interestingly, exponential concentration is of great computation convenience to manage H\"older's inequality. For instance a general issue is to bound~:
\begin{align*}
  \mathbb{E} \left[ (Z_1-a_1)^{r_1} \cdots (Z_m-a_m)^{r_m}\right].
\end{align*}
For any $\theta_1, \cdots, \theta_m \in (0,1)$ such that  $\theta_1+ \cdots + \theta_m=1$, H\"older's inequality gives us directly~:
\begin{align}\label{eq:holder}
  \mathbb{E} \left[ (Z_1-a_1)^{r_1} \cdots (Z_m-a_m)^{r_m}\right] \leq \prod_{i=1}^m\left(\mathbb{E} \, \vert Z_i-a_i \vert ^{\frac{r_i}{\theta_i}}\right)^{\theta_i} .
\end{align}
As we will see in the next proposition, the quantities $\mathbb{E} \, \vert Z_i-a_i \vert ^{r}$ can be bounded easily when $Z_i=a_i \pm C e^{-(\, \cdot \, /\sigma)^{q_i}}$, and we will even show in the next proposition that the bounds on $\mathbb{E} \, \vert Z_i-a_i \vert ^{r}$ for $r>0$ can become a characterization (it is actually a \textit{pseudo-characterization} since there is no equivalence) of $q$-exponential concentrations.
\begin{proposition}[Moment characterization of concentration, \protect{\cite[Proposition~1.10]{Led01}}]\label{pro:carcterisation_ac_moments_conc_q_expo}
  Given a random variable $Z$, a pivot $a\in \mathbb{R}$, two exponents $r,q>0$, and two parameters $C\geq e$ and $\sigma>0$, we have the implications :
  \begin{align*}
    Z \propto C e^{-( \cdot/\sigma)^q}
    \ \Rightarrow \
    \forall r \geq q : \mathbb{E}\left[\left\vert Z - Z'\right\vert^r\right] \leq C \Gamma\left(\frac{r}{q}+1\right)\sigma^r
    \ \Rightarrow \
    Z \propto C e^{-\frac{(\cdot /\sigma)^q}{e}}
  \end{align*}
  and 
  \begin{align*}
    Z \in a \pm C e^{-( \cdot/\sigma)^q}
    \ \Rightarrow \
    \forall r \geq q : \mathbb{E}\left[\left\vert Z - a\right\vert^r\right] \leq C \Gamma\left(\frac{r}{q}+1\right)\sigma^r
    \ \Rightarrow \
    Z \in a \pm C e^{-\frac{(\cdot /\sigma)^q}{e}},
  \end{align*}
  where $\Gamma : r \mapsto \int_0^\infty t^{r-1}e^{-t} dt$, (if $n\in \mathbb N$, $\Gamma(n+1) = n!$).
\end{proposition}


In both results, the first implication consists in bounding an expectation with a probability; it will involve the Fubini relation, giving for any positive random variable $Z$~:
\begin{align*}
  \mathbb{E} Z = \int_{Z} \left(\int_0^{\infty} \un_{t \leq Z} dt\right) dZ = \int_0^{\infty} \mathbb{P}\left(Z\geq t\right)dt,
\end{align*}
where $\un_{t \leq Z}$ is equal to $1$ if $t \leq Z$ and to $0$ otherwise. 

The second implication consists in bounding a probability with an expectation; it is a consequence of Markov's inequality, for any non decreasing function $f : \mathbb{R}\rightarrow \mathbb{R}$~:
\begin{align*}
  \mathbb{P}\left(Z\geq t\right) \leq \frac{\mathbb{E}f(Z)}{f(t)}.
\end{align*}
These two key indications given we can start the proof.
\begin{proof}
  We just prove the first implication, as it will be clear that both can be proved the same way. Let us first suppose that $r\geq q$. Knowing that $Z \propto  e^{-(ct)^q}$, we consider $Z'$, an independent copy of $Z$, and we can bound~:
  \begin{align*}
    \mathbb{E}\left[ \vert Z-Z' \vert^{r}\right]
    &=\int_{0}^{\infty} \mathbb{P}\left( \vert Z-Z' \vert^{r} \geq t\right) dt
    =\int_{0}^{\infty} rt^{r-1}\mathbb{P}\left( \vert Z-Z' \vert \geq t\right) dt \\
    &\leq C r \int_{0}^{\infty} t^{r-1}  e^{-(t/\sigma)^q}dt 
    = C\sigma^{r}  r\int_{0}^{\infty} t^{r-1}  e^{- t^q}dt,
  \end{align*}
  and, since $r\geq q$~:
  \begin{align*}
    r\int_{0}^{\infty} t^{r-1}  e^{- t^q}dt
    &= \frac{r}{q}\int_{0}^{\infty} t^{\frac{r}{q}-1} e^{- t}dt = \Gamma\left(\frac{r}{q}+1\right).
  \end{align*}
  
  Now, assuming the second term of the implication chain, we know from Markov's inequality that $\forall r \geq q$~:
  \begin{align*}
    \mathbb{P}\left(\left\vert Z-Z'\right\vert \geq t\right)
    &\leq \frac{\mathbb{E}\left[\left\vert Z-Z'\right\vert^r\right]}{t^r}
    \leq C \Gamma\left(\frac{r}{q}+1\right) \left(\frac \sigma t \right)^r
    \leq C \left(\frac{r}{q(t/\sigma)^q}\right)^{r/q}.
  \end{align*} 
  If $t\geq e^{\frac{1}{q}} \sigma$, we can then set $r= \frac{qt^q}{e\sigma^q}\geq q$, and we get~:
  \begin{align*}
    \mathbb{P}\left(\left\vert Z-Z'\right\vert \geq t\right)
    \leq C e^{-(t/\sigma)^q/e}.
  \end{align*}
  If $t\leq e^{\frac{1}{q}} \sigma$, we still know that $\mathbb{P}\left(\left\vert Z-Z'\right\vert \geq t\right)\leq 1$, and we conclude that~:
  \begin{align*}
    \forall t>0 \ \ : \ \ \ \mathbb{P}\left(\left\vert Z-Z'\right\vert \geq t\right)
    \leq \max(C,e)e^{-(t/\sigma)^q/e}.
  \end{align*}
\end{proof}

\begin{remark}\label{rem:borne_moments_r<q}
  In the last proposition, we did not provide a bound on $\mathbb{E}\left[\left\vert Z-Z'\right\vert^r\right]$ and $\mathbb{E}\left[\left\vert Z-a\right\vert^r\right]$ for $0\leq r <q$ since it is irrelevant to the characterization of $q$-exponential concentration. We may nonetheless easily bound these quantities with $C\sigma^r$. Getting inspiration from our previous derivations, we can indeed obtain when $1\leq r \leq q$~:
  \begin{align*}
    \mathbb{E}\left[ \vert Z-Z' \vert^{r}\right]
    &\leq C\sigma^{r}  \int_{0}^{\infty} r t^{r-1}  e^{- t^q}dt 
    \leq C\sigma^{r}  \int_{0}^{\infty} t^{r} q t^{q-1} e^{- t^q}dt \\
    &\leq C\sigma^{r}  \int_{0}^{1} q t^{q-1} e^{- t^q}dt + C\sigma^{r}  \int_{1}^{\infty} t^{q} q t^{q-1} e^{- t^q}dt\\
    &= C\sigma^{r}  \int_{0}^{\infty} qt^{q-1} e^{- t^q}dt 
    \ \ = \ C\sigma^r.
  \end{align*}
  And when $r\leq 1 \leq q$, we conclude with Jensen's inequality~:
  \begin{align*}
    \mathbb{E}\left[ \vert Z-Z' \vert^{r}\right]
    \leq \mathbb{E}\left[ \vert Z-Z' \vert\right]^{r}
    \leq C^r \sigma^r \leq C \sigma^r. 
  \end{align*}
\end{remark}

The following Lemma gives an alternative result to the aforementioned sufficiency of bounds on $\mathbb  E \left[ \vert Z - a\vert^r\right]$ (or on $\mathbb  E \left[ \vert Z - Z'\vert^r\right]$) for $r \in \mathbb N$.
\begin{lemma}\label{lem:relaxation_borne_moments}
    Given a random variable $Z\in \mathbb R_+$ and three parameters $C>e$, $q,\sigma>0$ one has the implication~:
    \begin{align*}
      \forall m \in \mathbb N,~ \mathbb E[Z^m] \leq C \left(\frac{m}{q}\right)^{\frac{m}{q}} \sigma^m&
      &\Longrightarrow&
      &\forall r \geq 0,~ \mathbb E[Z^r] \leq C e^{\frac 1 e} \left(\frac{r}{q}\right)^{\frac{r}{q}}  \left(\sigma(2/\bar q)^{\frac{1}{q}}\right)^r,
    \end{align*}
    where $\bar q = \min(q,1)$.
  \end{lemma} 
  \begin{proof}
When $r \leq 1$, we already know thanks to Jensen's inequality, by concavity of $t \mapsto t^r$, that~:
\begin{align*}
  \mathbb E [Z^r] \leq (\mathbb E[Z])^r \leq C \left(\frac{1}{q}\right)^{\frac{r}{q}} \sigma^r \leq C e^{\frac 1 e} \left(\frac{r}{q}\right)^{\frac{r}{q}}  \left(\frac{\sigma}{q^{\frac{1}{q}}}\right)^r,
 \end{align*}
 since $\forall t>0$, $t^t \geq \frac{1}{e^{\frac{1}{e}}}$.

 When $r\geq 1$, one can invoke the well known result concerning $\ell^r$ norms, where in our case, $\left\Vert Z\right\Vert_{\ell^r}=\mathbb{E}\left[\left\vert Z\right\vert^r\right]^{1/r}$. Let us consider the general case where we are given $p_1,p_2>0$ such that $p_1 \leq r \leq p_2$ and we consider $\theta \in \left(0,1\right)$ satisfying $1/r=\theta/p_1 + (1-\theta)/p_2$. We then have the inequality :
  \begin{align*}
    \left\Vert Z\right\Vert_{\ell^r} \leq \left\Vert Z\right\Vert_{\ell^{p_1}}^\theta \left\Vert Z\right\Vert_{\ell^{p_2}}^{1-\theta}.
  \end{align*}
  This implies~:
  \begin{align*}
    \mathbb{E}\left[Z^r\right]^{\frac{1}{r}}
    &\leq \left(C^{\frac{1}{p_1}} \sigma \left(\frac{p_1}{q }\right)^{\frac{1}{q}}\right)^\theta  \left(C^{\frac{1}{p_2}} \sigma \left(\frac{p_2}{q}\right)^{\frac{1}{q}}\right)^{1-\theta}  
    \leq\frac{C^{\frac{1}{r}} \sigma}{q^{\frac{1}{q}}} p_1^{\frac{\theta}{q}}p_2^{\frac{1-\theta}{q}}.
  \end{align*}
  We would like to bound $p_1^{\frac{\theta}{q}}p_2^{\frac{1-\theta}{q}}$ with $r^{\frac{1}{q}}$.
  Unfortunately, for $\theta \in \left]0,1\right[$, $p_1^{\theta}p_2^{1-\theta}> 1/(\theta/p_1 + (1-\theta)/p_2)=r$ (this is due to the inequality of arithmetic and geometric means, itself a consequence of the concavity of the $\log$ function). However, for the particular setting under study~:
  \begin{align*}
    \left(\frac{\theta}{p_1} + \frac{1-\theta}{p_2}\right)p_1^\theta p_2^{1-\theta}
    &\leq \left(\frac{1}{p_2} +\frac{(p_2-p_1)\theta}{p_1p_2}\right) p_2 
     \ \ \leq \  1 +\frac{p_2-p_1}{p_1}.
  \end{align*}
  As a consequence, taking $p_1 = \lfloor r \rfloor$ and $p_2 = \lceil r \rceil$, one obtains~: $$\mathbb{E}\left[\left\vert Z\right\vert^r\right]\leq C \sigma^r \left(\frac{2r}{q}\right)^{\frac{r}{q}}.$$ 
\end{proof}

\begin{remark}\label{rem:borne_esp_Z}
  In Proposition~\ref{pro:carcterisation_ac_moments_conc_q_expo}, we saw that if $Z \in a \pm Ce^{-(\, \cdot \, /\sigma )^q}$ then $\left\vert \mathbb{E}Z - a\right\vert\leq \mathbb{E}\left\vert Z-a \right\vert \leq C \left(\frac{1}{q}\right)^{1/q} \sigma <\infty$. Therefore any $q$-exponentially concentrated random variable admits a finite expectation.
\end{remark}

Now that we know it exists, we are going to show that the expectation $\mathbb{E}Z$ plays the same pivotal role as any median.
\begin{corollary}[\protect{\cite[Proposition~1.9]{Led01}}]\label{cor:moyenne_pivot}
  With the notations of the previous proposition, one has~:
  \begin{align*}
    Z \in a \pm Ce^{-(\, \cdot \, /\sigma )^q}&
    &\Longrightarrow&
    &Z \in \mathbb{E}Z \pm e^{\frac{C^q}{q}} e^{-(\, \cdot \,/2\sigma)^q}.
  \end{align*}
\end{corollary}
\begin{proof}
  We suppose that $Z \in Ce^{-(\, \cdot \, /\sigma )^q}$ for a pivot $a\in \mathbb{R}$. Proposition~\ref{pro:carcterisation_ac_moments_conc_q_expo} applied in the case $r=1$ gives us $\left\vert a - \mathbb{E}Z\right\vert\leq C (\frac{1}{q})^\frac{1}{q} \, \sigma$. One can then invoke Corollary~\ref{cor:sens_param_de_queue}, to get the concentration~:
  \begin{align*}
    Z \in \mathbb{E}Z \pm C' e^{-(\, \cdot \,/2\sigma)^q},
  \end{align*}
  with $C'=\max(C,e^{\frac{C^q}{q}})$. It is then interesting to note that the function $q \mapsto \frac{C^q}{q}$ has a minimum in $\frac{1}{\log(C)}$ where it takes the value $e \log C$. Then we see that $$C=e^{\log C}\leq e^{e\log C} \leq e^{\frac{C^q}{q}},$$ and we can simplify the head parameter to obtain the result of the corollary.
\end{proof}
If $Z\in a\pm Ce^{-(\, \cdot \, /\sigma)^p}$, we can then employ $\mathbb{E}Z$ as a pivot of $Z$ to get a bound on the centered moments~:
\begin{corollary}\label{cor:borne_moments_centres_conc_q_exp}
  With the notations of Corollary~\ref{cor:moyenne_pivot}, if we suppose that $Z\in a \pm C e^{-(\, \cdot \, /\sigma)^q}$, then we have~:
  \begin{align*}
    &\mathbb{E}\left\vert Z- \mathbb{E}Z  \right\vert^r \leq  e^{\frac{C^q}{q}}  \left(2 \sigma\right)^r\left(\frac{r}{q}\right)^{\frac{r}{q}} .
  \end{align*}
\end{corollary}

 The previous development of $q$-exponential concentration primarily aims at providing a versatile and convenient ``toolbox'' (note that most introduced inequalities could have been enhanced, however to the expense of clarity) for the subsequent treatment of large dimensional random vectors, rather than random variables. This analysis will be performed through resorting to concentrated \emph{functionals} of the random vectors, i.e., real images of through a mapping with different levels of regularity (linear, Lipschitz, convex..). For large dimensional vectors, one is mostly interested in the \emph{order} of the concentration, thus the various constants appearing in most of the previous propositions and lemmas do not have any particular interest; of major interest instead is the independence of the concentration with respect to the random vector dimension, as observed for instance in Theorem~\ref{the:concentration_vecteur_spherique} and that we will extend to other type of random vectors in what follows.

\section{Concentration of a random vector of a normed vector space}\label{sse:conc_vect_al}


In Section~\ref{sec:emp_cov}, the random vectors under study are either in $E=\mathbb{R}^p$ endowed with the Euclidean norm $\left\Vert z \right\Vert=\sqrt{\sum_{i=1}^pz_i^2}$ or the $\ell_1$-norm $\Vert z \Vert_1 = \sum_{i=1}^p \vert z_i \vert$, or in $E=\mathcal{M}_{pn}$, $p,n \in \mathbb{N}$, endowed with two possible norms\footnote{The notion of concentration presented in this subsection can be immediately extended to vector spaces endowed with \textit{semi norms}}~:
\begin{itemize}
  \item the spectral norm defined as $\left\Vert M\right\Vert= \sup_{\left\Vert z\right\Vert\leq 1} \left\Vert Mz\right\Vert$,
  \item the Frobenius norm $\left\Vert M \right\Vert_F=\sqrt{\sum_{\genfrac{}{}{0pt}{2}{1\leq k\leq p}{1 \leq i \leq n}}M_{k,i}^2}$
 \end{itemize} 
(where $M \in \mathcal M_{p,n}$). Note that the Frobenius norm can be seen as a Euclidean norm on $\mathbb{R}^{pn}$ with the bijection that concatenates the column of a matrix.

To generalize the notion of concentration to the case of a random vector of a normed vector space $(E, \left\Vert \cdot\right\Vert)$, one might be tempted to follow the idea of Definition~\ref{def:concentration_autour} and say that a vector $Z \in E$ is $\alpha$-concentrated if one has for instance $\mathbb{P}\left(\left\Vert Z- \tilde Z\right\Vert\geq t\right)\leq \alpha(t)$ for a deterministic vector $\tilde Z$, well chosen. This would describe a notion of a concentration \textit{around a vector}.

However, this basic notion would not be compatible with the fundamental example of the uniform distribution on the sphere of radius $\sqrt{p}$ presented in Theorem~\ref{the:concentration_vecteur_spherique} or the Gaussian vectors of identity population covariance matrices. 
When the dimension grows, those random vectors concentrate around a growing sphere which is the exact opposite behavior of being concentrated around a point. Yet, they present strong dimension free concentration properties that we will progressively identify through the presentation of three fundamental notions~:
\begin{enumerate}
    \item The \textit{linear concentration} which is the concentration of $u(Z-\tilde Z)$ for some deterministic vector $\tilde Z \in E$ (the so-called \textit{deterministic equivalent}) and for any bounded linear form $u :E \mapsto \mathbb R$. For instance, we know from Theorem~\ref{the:concentration_vecteur_spherique} that any random vector $Z$ uniformly distributed on the sphere admits $\tilde Z = \mathbb E Z =0$ as a deterministic equivalent. This means that most drawings of $Z$ are close to the equator when the dimension grows.
    \item The \textit{Lipschitz concentration} which is the concentration of $f(Z)-f(Z')$ for any i.i.d. copy $Z'$ of $Z$ and any Lipschitz map $f :E \mapsto \mathbb R$.
    \item The \textit{convex concentration} which is the concentration of $f(Z)-f(Z')$ for any Lipschitz \textit{and weakly convex} map $f :E \mapsto \mathbb R$. This notion is of course weaker than the Lipschitz concentration, its presentation here is only justified by the fundamental Theorem~\ref{the:talagrand} owed to Talagrand that provides concentration properties on random vectors with independent and bounded entries. It is less ``stable'' than the Lipschitz concentration, meaning there exist very few transformations that preserve convex concentration. As a consequence one usually naturally returns to linear concentration after some refinement. For instance, supposing that $X$ is convexly concentrated, only a linear concentration can be obtained for the resolvent $Q = (XX^T/n + zI_p)^{-1}$. 
\end{enumerate}
Although a seemingly basic notion, linear concentration still has quite interesting features that justify an independent treatment at the beginning of this section.

\subsection{Linear Concentration}
\subsubsection{Definition and first properties}
Considering the linear functionals of random vectors allows us, in particular, to introduce the notion of \textit{deterministic equivalents}, which play the role of the pivots we presented in the concentration of random variables.

\begin{definition}\label{def:deterministic equivalent}
  Given a random vector $Z\in E$, a deterministic vector $\tilde Z\in E$ and a concentration function $\alpha$, we say that $Z$ is \emph{linearly $\alpha$-concentrated} around the \emph{deterministic equivalent} $\tilde Z$ if for any bounded linear form $u : E \rightarrow \mathbb{R}$ with a unit operator norm (i.e., $\forall z\in E, \ \vert u(z)\vert\leq \left\Vert z\right\Vert$)~:
  \begin{align*}
    u(Z) \in u(\tilde Z) \pm \alpha. 
  \end{align*}
  We note in that case : $Z \in \tilde Z \pm \alpha \ \ \text{in} \ (E,\left\Vert \cdot\right\Vert)$.
\end{definition}
\begin{remark}\label{rem:compatibilite_avec_def_2}
  Definition~\ref{def:deterministic equivalent} is clearly compatible with Definition~\ref{def:concentration_autour} for the case where $E=\mathbb{R}$. Indeed, in $\mathbb{R}$ the linear forms are the scalar functions and for any coefficient $\lambda \in \mathbb{R}_*$ and any random variable $Z\in \mathbb R$ we have the equivalence~:
  \begin{align*}
     Z\in a \pm \alpha&
     &\Longleftrightarrow&
     &\lambda Z \in \lambda a \pm \alpha(\, \cdot \, / \lambda).
  \end{align*}
\end{remark}
\color{blue}
\color{black}
\begin{remark}\label{rem:avantage_fonctionnelles_lin\'{e}aires}
  The advantage of \textit{linear functionals} is that they preserve the expectation which gives us the simplest deterministic equivalent. For instance if $Z$ is linearly $q$-exponentially concentrated, then we know from Corollary~\ref{cor:moyenne_pivot} that $\mathbb{E}Z$ is a deterministic equivalent for $Z$ since for any continuous linear form $u$, $\mathbb{E} [u(Z)]=u(\mathbb{E}[Z])$. 
  For instance, thanks to Theorem~\ref{the:concentration_vecteur_spherique}, one knows that for any $p \in \mathbb{N_*}$, if $Z \sim \sigma_p$ then $Z \in 0 \pm 2e^{-\, \cdot \, ^2/2}$ since $\mathbb{E}[Z]=0$.
\end{remark}

Since the Frobenius norm is larger than the spectral norm, for any random matrix $M\in \mathcal M_{p,n}$, we have the implication~:
\begin{align*}
  M\in \tilde M \pm \alpha \ \ \text{in} \ (\mathcal M_{p,n}, \left\Vert \cdot\right\Vert_F)&
  &\Longrightarrow&
  &M\in \tilde M \pm \alpha \ \ \text{in} \ (\mathcal M_{p,n}, \left\Vert \cdot\right\Vert),
\end{align*}
for some deterministic matrix $\tilde M$ and some concentration function $\alpha$. When the choice of the norm is not ambiguous (when $E=\mathbb{R}^p$ in particular), we will allow ourselves not to specify the norm.
\begin{remark}\label{rem:indentification_forme_lin_vect}
  The same way that linear forms in $\mathbb{R}^p$ are fully described with the scalar product, in $\mathcal M_{p,n}$ the linear forms are fully defined by the functions $f_A : M \mapsto \tr A M$ for $A\in \mathcal M_{p,n}$ where $f_A$ is said to have a unit norm if $\left\Vert f_A\right\Vert_*=1$, i.e.,
  \begin{itemize}
    \item in $(\mathcal M_{p,n}, \left\Vert \, \cdot \,\right\Vert_F)$ : $\left\Vert A\right\Vert_F= \sqrt{\tr (AA^T)}=1$, 
    \item in $(\mathcal M_{p,n}, \left\Vert \, \cdot \,\right\Vert)$ : $\left\Vert A\right\Vert_1= \tr (AA^T)^{\frac{1}{2}}=  1$.
  \end{itemize}
\end{remark}

Linear concentration is provided by classical concentration inequalities like Bernstein's or Hoeffding's inequalities. 
 \begin{example}\label{exe:bernstein}[Bernstein's inequality, \cite[Theorem 2.8.2]{Ver17}]
Given $p$ independent random variables $(Z_i)_{1\leq i \leq p} \in \mathbb{R}^{p}$, and three parameters $C\geq e$ and $c,q >0$, we have the implication~:
\begin{align*}
   \forall i \in \{1,\ldots, p\}, Z_i \in \mathbb{E} [Z_i] \pm C e^{- \, \cdot \, /\sigma}&
   &\Rightarrow&
   & Z \in \mathbb{E}[Z] \pm 2Ce^{-c(\, \cdot \, /\sigma)^2} + 2C e^{- c\, \cdot/ \sigma},
\end{align*}
where $c$ is a numerical constant depending only on $C$ and $C = (Z_1,\ldots, Z_p)$.
\end{example}

\begin{example}\label{exe:hoeffding}[Hoeffding's inequality, \cite[Theorem 2.2.6]{Ver17}]
Given $p$ independent random variables $Z_1,\ldots,Z_p\in [0,1]$, the random vector $Z=(Z_1, \ldots,Z_p) \in \mathbb{R}^p$ is linearly concentrated and verifies $Z \in \mathbb{E}[Z] \pm 2 e^{-2 \, \cdot \,^2}$ in $(\mathbb{R}^p, \left\Vert \cdot\right\Vert)$.
\end{example}
We will see in Subsection~\ref{ssse:concentration_convexe} about convex concentration a generalization of Hoeffding's theorem with the theorem of Talagrand.

Dealing with $q$-exponential concentrations, we can get an analogous result to Corollary~\ref{cor:sens_param_de_queue} that allowed us to interchange the pivot $a$ with any other real in a ball around $a$ with a diameter of the same order as the tail parameter. 
The tail parameter will be called in the case of a random vector an \textit{observable diameter}. Given a random vector $Z\in E$, we will say that $Z$ is \emph{$q$-exponentially concentrated} around $\tilde Z$ with a \emph{head parameter $C$} and an \emph{observable diameter $\sigma$} if $Z \in \tilde Z \pm C e^{-(\, \cdot \, / \sigma)^q}$.
The \emph{observable diameter} is the diameter of the ``observations''  of the distribution that could be seen as linear projections (for other types of concentrations to be introduced subsequently, they will alternatively be related to 1-Lipschitz or quasi-convex maps) on $\mathbb{R}$. Intuitively, it is the diameter of the observation in the ``real world'' ; refer to \cite{gro99} for a both more precise and more general definition.
  
\begin{lemma}\label{lem:comportement_pivot_equivalent_deterministe}
  Let us consider a random vector $Z\in E$, two deterministic vectors $\tilde Z,\tilde Z' \in E$ and three parameters $C\geq e$, $\lambda,\sigma>0$. If $\Vert \tilde Z - \tilde Z'\Vert\leq \lambda \sigma$ then we have the implication~:
  \begin{align*}
    Z \in \tilde Z \pm C e^{-(\, \cdot \, / \sigma)^q}&
    &\Longrightarrow&
    &Z \in \tilde Z' \pm \max(C, e^{\lambda^q}) e^{-(\, \cdot \, / \sigma)^q}.
  \end{align*}
\end{lemma}

\subsubsection{Product of Linearly concentrated random vectors}
One might expect for concentrated random vectors a similar proposition to Proposition~\ref{pro:conc_autour_prod_ss_borne} giving the concentration of a product of random variables. In the case of vectorial objects, we are looking for the concentration of a scalar product of two random vectors $X,Y\in E$ or, closer to our present interest as announced in the preamble, for the concentration of $u^TQx$ where $u$ is deterministic, $Q$ is a concentrated random matrix and $x$ a concentrated random vector. If one looks closely at the proof of Proposition~\ref{pro:conc_autour_prod_ss_borne}, it becomes obvious that although some steps can be fully adapted, the method gets stuck when trying to bound~:
\begin{align*}
  \mathbb{P}\left(\left\vert u\left( (X- \mathbb{E}X)(Y- \mathbb{E}Y)\right)\right\vert\geq t\right).
\end{align*}
It is tempting here to invoke the Cauchy-Schwartz inequality~:
\begin{align*}
  u\left((X- \mathbb{E}X)(Y- \mathbb{E}Y)\right)\leq \Vert u \Vert \left\Vert X- \mathbb{E}X\right\Vert\left\Vert Y- \mathbb{E}Y\right\Vert,
\end{align*}
where $\Vert u \Vert$ is the operator norm of $u$. However, as we explained it with the example of spherical and Gaussian vectors in the introduction of this subsection, unlike $u( X- \mathbb{E}X)$ and $u( Y- \mathbb{E}Y)$ the quantities $\left\Vert X- \mathbb{E}X\right\Vert$ and $\left\Vert Y- \mathbb{E}Y\right\Vert$ are far from being concentrated for concentrated random vectors $X$ and $Y$ of practical use. 
We will see in Proposition~\ref{pro:tao} that it is still possible to express the concentration of the norms and obtain consequently loose bounds for the concentration of $(X-\mathbb E X)(Y - \mathbb E Y)$ as presented in Examples~\ref{exe:conc_lineaire_produit_odot} and~\ref{exe:concentration_lineaire_covariance_empirique}.

Before that, to present a setting where the concentration is satisfactory, we suppose in addition to the concentration that the two vectors are independent and one is bounded.
\begin{proposition}\label{pro:conc_f(Z_1,Z_2)_independant}
  Let us consider two normed vector spaces $(E_1,\left\Vert \cdot\right\Vert_1)$ and $(E_2,\left\Vert \cdot\right\Vert_2)$, two \emph{independent} random vectors $Z_1\in E_1$ and $Z_2\in E_2$ and a bilinear form $f : E_1\times E_2 \rightarrow \mathbb{R}$ such that for any $(z_1,z_2)\in E_1\times E_2$~: 
  \begin{align*}
    \left\vert  f(z_1,z_2)\right\vert \leq \left\Vert z_1\right\Vert_1 \left\Vert z_2\right\Vert_2.
  \end{align*}  
  If there exist two concentration functions $\alpha, \beta : \mathbb R_+ \rightarrow \mathbb R_+$ and two deterministic vectors $(\tilde Z_1, \tilde Z_2)\in E_1 \times E_2$ such that~:
  \begin{align*}
    Z_1 \in \tilde Z_1 \pm \alpha  \ \text{in} \ (E_1, \left\Vert \cdot\right\Vert_1)&
    &\text{and}&
    &Z_2 \in \tilde Z_2  \pm \beta \ \text{in} \ (E_2, \left\Vert \cdot\right\Vert_2),
  \end{align*}
  and if $\left\Vert  Z_2\right\Vert_2$ is bounded by a real $K_2>0$, then~:
  \begin{align*}
    f(Z_1,Z_2)\in f\left(\tilde Z_1,\tilde Z_2\right) \pm \alpha \left(\frac{\, \cdot \,}{2K_2}\right)+ \beta \left(\frac{\, \cdot \,}{2\Vert \tilde Z_1\Vert_1}\right). 
  \end{align*}
\end{proposition}
\begin{proof}
  Given $t>0$, we just compute~:
  \begin{align*}
    &\mathbb{P}\left(\left\vert f(Z_1,Z_2)-f(\tilde Z_1,\tilde Z_2)\right\vert\geq t\right)\\
     &\hspace{1cm}\leq \mathbb{E}\left[ \mathbb{P}\left(\left\vert f(Z_1-\tilde Z_1,Z_2) \right\vert\geq \frac{t}{2} \ \vert \ Z_2\right)\right] 
    + \mathbb{P}\left(\left\vert f(\tilde Z_1,Z_2- \tilde Z_2)\right\vert\geq \frac{t}{2}\right) \\
    &\hspace{1cm}\leq \alpha \left(\frac{t}{2K_2}\right)+ \beta \left(\frac{t}{2 \Vert \tilde Z_1\Vert_1}\right),
  \end{align*}
  since $x\mapsto f(x,Z_2)$ and $y\mapsto f(\tilde Z_1,y)$ are both linear and respectively $K_2$-Lipschitz and $\Vert \tilde Z_1\Vert$-Lipschitz.
\end{proof}
If we place ourselves in an algebra $\mathcal A$ endowed with an algebra norm $\Vert \cdot \Vert$ (verifying $\Vert Z_1 Z_2 \Vert \leq \Vert Z_1\Vert \Vert Z_2 \Vert$, $\forall Z_1, Z_2 \in \mathcal A$), in the setting of Proposition~\ref{pro:conc_f(Z_1,Z_2)_independant}, we also have the concentration of the random vector $Z_1Z_2$.
\begin{corollary}\label{cor:conc_lineaire_pdt_independant_borne}
  Given two independent random vectors $Z_1,Z_2$ of an algebra $\mathcal A$, some constant $K_2$, two deterministic vectors $\tilde Z_1,\tilde Z_2 \in \mathcal A$ and two concentration function $\alpha,\beta : \mathbb R_+ \rightarrow \mathbb R_+$, we have the implication~:
  \begin{align*}
      \left\{\begin{aligned}
      &Z_1 \in \tilde Z_1 \pm \alpha \\
      &Z_2 \in \tilde Z_2 \pm \beta \\
      &\Vert Z_2 \Vert \leq K_2
      \end{aligned} \right.&
      &\Longrightarrow&
      &Z_1Z_2 \in \tilde Z_1 \tilde Z_2 \pm \alpha \left(\frac{\, \cdot \,}{2K_2}\right)+ \beta \left(\frac{\, \cdot \,}{2\Vert \tilde Z_1\Vert_1}\right).
  \end{align*}
\end{corollary}

\subsubsection{Size of the norm}
If a random vector $Z$ is linearly concentrated around a deterministic equivalent $\tilde Z$, it is possible to control the norm $\Vert Z -\tilde Z \Vert$ if the norm $\left\Vert \, \cdot \,\right\Vert$ can be defined as the supremum on a set of linear forms. For instance, in $\mathbb{R}^p$ endowed with the sup norm $\left\Vert \cdot \right\Vert_{\infty}$ ($\left\Vert x \right\Vert_{\infty} = \sup\{\vert x_i \vert, 1 \leq i \leq p\}$), if $Z \in \tilde Z \pm \alpha$~:
\begin{align}\label{eq:concentration_norme_infinie}
    \mathbb{P}\left(\Vert Z - \tilde Z \Vert_{\infty} \geq t\right) 
    &=\mathbb{P}\left(\sup_{1\leq i \leq p} e_i^T(Z - \tilde Z) \geq t\right) \\
    &\leq p \sup_{1\leq i \leq p}\mathbb{P}\left( e_i^T(Z - \tilde Z) \geq t\right) \leq p \alpha(t),
\end{align}
where $(e_1, \ldots,e_p)$ is the canonical basis of $\mathbb R^p$ ($e_i$ is a vector full of $0$ with $1$ on the $i^{\text{th}}$ coordinate).
To manage the infinity norm, the supremum is taken on a finite set $\{e_1, \ldots e_p\}$; things are more complex if we look at the Euclidean norm because this time one comes to use the identity $\left\Vert x \right\Vert = \sup\{ u^T x, \Vert u \Vert \leq 1\}$ where the supremum is taken on the unit ball.
To tackle this loss of cardinality, it is convenient here to introduce the so-called $\epsilon$-nets to discretize the ball in order to approach sufficiently the norm and at the same time find a good bound for the probability (see \cite{Tao11}).  
We leave the proof in the appendix.

  \begin{proposition}\label{pro:tao}
  Let us consider a normed vector space $(E, \left\Vert \, \cdot \,\right\Vert)$ of finite dimension such that there exists a subspace $H$ of the dual space $(E^*,\left\Vert \, \cdot \, \right\Vert_*)$, and a ball $\mathcal B_H =\{f \in H, \left\Vert f\right\Vert_*\leq 1\} \subset H$ verifying for any $z\in E$~: 
  \begin{align}
     \left\Vert z\right\Vert=\sup_{f \in \mathcal B_H} f(z).
  \end{align}
  Given a random vector $Z \in E$, a deterministic equivalent $\tilde Z$ and a concentration function $\alpha$, if we suppose that $Z \in \tilde Z \pm \alpha$ then we have the concentration~:
  \begin{align}\label{eq:concentration_norme_d_H}
    \left\Vert Z -\tilde Z\right\Vert \in 0 \pm 8^{{\rm dim}(H)} \alpha \left( \frac{\cdot}{2}\right),
  \end{align}
  where ${\rm dim}(H)$ is the dimension of $H$. 
\end{proposition}
\begin{remark}\label{rem:avantage_enonce_prop_tao}
  One always has $\left\Vert z\right\Vert=\sup_{\left\Vert f\right\Vert_*\leq 1} f(z)$, so one might be tempted to systematically consider $H=E^*$. For instance in $(\mathbb{R}^p, \left\Vert \, \cdot \, \right\Vert)$, $H$ is taken to be equal to $\mathbb{R}^p$ and ${{\rm dim}(H)}=p$. The same way, in the Euclidean space $(\mathcal M_{p,n}, \left\Vert \, \cdot \,\right\Vert_F)$, we take $H=\mathcal M_{p,n}$ (and ${{\rm dim}(H)}=pn$). However, we can reduce greatly the dimension ${\rm dim}(H)$ in the case of $(\mathcal M_{p,n}, \left\Vert \, \cdot \,\right\Vert)$ since for any $M \in \mathcal M_{p,n}$~:
  \begin{align*}
    \left\Vert M\right\Vert = \sup_{\left\Vert u\right\Vert, \left\Vert v\right\Vert\leq 1} \left\vert u^TMv\right\vert.
  \end{align*}
  Thus it is wise to consider $H=\{f_{vu^T}, u,v \in \mathbb{R}^p\times \mathbb{R}^n\}$ (see Remark~\ref{rem:indentification_forme_lin_vect} for a definition of $f_A$ ; here, $\left\Vert f_{vu^T}\right\Vert_*=\left\Vert vu^T\right\Vert_1=\left\Vert u\right\Vert \left\Vert v\right\Vert$) and the dimension is then only ${\rm dim}(H)=p+n$.
\end{remark}
Taking into account the two results \eqref{eq:concentration_norme_infinie} and \eqref{eq:concentration_norme_d_H}, we are led to introduce an indicator characteristic to the norm that gives the speed of the concentration.  
\begin{definition}[Norm degree]\label{def:norm_degree}
  Given a normed vector space $(E, \Vert \cdot \Vert)$, and a subset $H\subset E^*$, let us define the degree $\eta_H$ of $H$ as~:
  \begin{itemize}
      \item $\eta_H=\log(\# H)$ if $H$ is finite
      \item $\eta_H=\dimm(\vect H)$ if $H$ is infinite 
  \end{itemize}
  where $\# H$ is the number of elements in $H$ and $\vect H$ is the subspace of $E^*$ generated by $H$.
  When it is defined, we note $\eta(E, \Vert \cdot \Vert)$ or more simply $\eta _{\Vert \cdot \Vert}$ the degree of $\Vert \cdot \Vert$ that is defined as\footnote{If we allow $\eta_{(E,\|\cdot\|)}$ to take infinite values, it is always defined thanks to Hahn-Banach Theorem}~:
  \begin{align*}
      \eta_{\Vert \cdot \Vert}=\eta(E, \Vert \cdot \Vert)=\inf \left\{\eta_H, H\subset E^* \ | \ \forall x \in E, \Vert x \Vert = \sup_{f \in \mathcal B_H}f(x)\right\}.
  \end{align*}
\end{definition}
In the setting of the last proposition~:
\begin{align}\label{eq:conc_norm_ac_eta}
    \left\Vert Z -\tilde Z\right\Vert \in 0 \pm e^{c\eta(E,\left\Vert \cdot\right\Vert)} \alpha \left( \frac{\cdot}{2}\right),
\end{align}
where $c$ is a numerical constant. In the case of the $q$-exponential concentration, it is possible to rearrange the concentration of $\Vert Z - \tilde Z \Vert$ to obtain a head parameter of order $1$. 
\begin{proposition}\label{pro:tao_conc_exp}
  Given a random vector $Z \in E$, a deterministic vector $\tilde Z \in \mathbb  R^p$ and three parameters $C \geq e$, $q,\sigma>0$, we have the implication~:
  \begin{align*}
      Z \in \tilde Z \pm C e^{-(\, \cdot \, /\sigma)^q}&
      &\Longrightarrow&
      & \Vert Z - \tilde Z \Vert \in 0 \pm Ce^{-(\, \cdot \,/2\sigma)^q/2c\eta_{\Vert \cdot \Vert}},
  \end{align*}
  where $c$ is the same numerical constant as in~\eqref{eq:conc_norm_ac_eta}. 

  Reciprocally~:
  \begin{align*}
      \Vert Z - \tilde Z \Vert \in 0 \pm Ce^{-(\, \cdot \,/\sigma)^q}&
      &\Longrightarrow&
      &Z \in \tilde Z \pm C e^{-(\, \cdot \, /\sigma)^q}.
  \end{align*}
\end{proposition}
The second result of the proposition is trivial (and quite useless) and the first result is just a simple consequence of \eqref{eq:conc_norm_ac_eta} combined with Lemma~\ref{lem:head_parameter_to_tail}.
\begin{example}\label{exe:norm_degree}
    We can give some examples of norm degrees~:
    \begin{itemize}
        \item $\eta \left( \mathbb R^p, \Vert \cdot \Vert_\infty \right) = \log(p)$  \ \ (for $H = \{x \mapsto e_i^Tx, 1\leq i\leq p\}$)
        \item $\eta \left( \mathbb R^p, \Vert \cdot \Vert \right) = p$  \ \ (for $H = \{x \mapsto u^Tx, u \in \mathcal B_{\mathbb R^p}\}$)
        \item $\eta \left( \mathcal M_{p,n}, \Vert \cdot \Vert \right) = n+p$ \ \  (for $H = \{M \mapsto u^TMv, (u,v)\in \mathcal B_{\mathbb R^p} \times \mathcal B_{\mathbb R^n}\}$)
        \item $\eta \left( \mathcal M_{p,n}, \Vert \cdot \Vert_F \right) = np$   \ \ (for $H = \{M \mapsto \tr(AM), A \in \mathcal M_{n,p}, \|A\|_F\leq 1\}$).
    \end{itemize}
\end{example}
In the particular case of $q$-exponential concentrations, the norm degree allows us to bound the expectation of the norm of $Z-\tilde Z$ thanks to Proposition~\ref{pro:tao_conc_exp}.
\begin{corollary}\label{cor:tao}
  Given a random vector $Z \in E$, if we suppose that $Z \in \tilde Z \pm C e^{- (\, \cdot \, /\sigma)^q}$ and $q\geq 1$, we can bound~:
  \begin{align*}
    \mathbb{E}\left\Vert Z-\tilde Z\right\Vert \leq C' \sigma \eta_{\Vert \cdot \Vert}^{1/q},
  \end{align*}
  where $C'$ is a constant depending on $C$.
\end{corollary}
\begin{example}\label{exe:borne_esp_norm_vecteur_lin_conc}
Let $Z \in \mathbb{R}^p$ and $M \in \mathcal M_{p,n}$ be two random vectors. Then,
\begin{itemize}
  \item  if $Z\in \tilde Z \pm 2 e^{-t^2/2}$ in $(\mathbb{R}^p,\left\Vert \cdot\right\Vert)$~: $\mathbb{E}\left\Vert Z\right\Vert\leq\Vert \tilde  Z\Vert + C \sqrt{p}$
  \item if $M \in \tilde M \pm 2 e^{-t^2/2}$ in $(\mathcal M_{p,n}, \left\Vert \, \cdot \,\right\Vert)$~: 
  $\mathbb{E}\left\Vert M\right\Vert \leq \Vert \tilde M\Vert + C\sqrt{p+n},$
  \item if $M \in \tilde M \pm 2 e^{-t^2/2}$ in $(\mathcal M_{p,n}, \left\Vert \, \cdot \,\right\Vert_F)$~:
  $\mathbb{E}\left\Vert M\right\Vert \leq \Vert \tilde M\Vert + C\sqrt{pn}$.
\end{itemize}
\end{example}

Now that we can control the quantities $\Vert Z-\tilde Z \Vert$ when $\tilde Z$ is a deterministic equivalent of $Z$, it is tempting to extend Corollary~\ref{cor:conc_lineaire_pdt_independant_borne} to the concentration of the product of any linearly concentrated random vectors. The examples presented below are just here to give an idea of what could be obtained, they are not relevant in practice since the bounds are too loose. Unlike Lipschitz concentration as it will be presented in next subsection, linear concentration is not suited to study the concentration of the product of random vectors.

\begin{example}\label{exe:conc_lineaire_produit_odot}
   Let us note $\odot$ the product in $\mathbb{R}^p$ verifying for any $x=(x_1,\ldots,x_p)$ and $y=(y_1,\ldots,y_p)$, $x \odot y = (x_1 y_1, \ldots, x_p y_p)$. It gives an algebra structure to $\mathbb R^p$ where $\Vert \cdot \Vert_\infty$ and $\Vert \cdot \Vert_2$  are both algebra norms ($\Vert x \odot y \Vert_2 \leq \Vert x \odot y \Vert_1 = \sum_{i=1}^p \vert x_i y_i\vert \leq \Vert x \Vert_2 \Vert y \Vert_2 $ thanks to the Cauchy Schwarz inequality). Therefore we have for any vector $Z\in \tilde Z \pm 2 e^{- \,\cdot \,^2/2}$ in $(\mathbb{R}^p,\Vert \cdot \Vert_2)$~:
   \begin{itemize}
       \item $\frac{Z^{\odot 2}}{p}= \frac{Z \odot Z}{p} \in \frac{\tilde Z ^{\odot^2}}{p} \pm C e^{-c \cdot} +  C e^{-c (\sqrt{p} \,\cdot \,/\Vert \tilde Z\Vert)^2}$ in $(\mathbb{R}^p, \Vert \cdot \Vert_2)$ 
       \item $\frac{Z^{\odot 2}}{\log p}  \in \frac{\tilde Z ^{\odot^2}}{\log p}\pm Ce^{-c \cdot} +  C e^{-c (\sqrt{\log p}\, \cdot \,/\Vert \tilde Z\Vert_\infty)^2} $ in $(\mathbb{R}^p, \Vert \cdot \Vert_\infty)$,
   \end{itemize}
   where $C\geq e$ and $c>0$ are two numerical constants. 
\end{example}
\begin{example}[Concentration of the sample covariance]\label{exe:concentration_lineaire_covariance_empirique}
    Given a matrix $X \in \mathcal   M_{p,n}$, and three parameters $C,q \geq 1$ and $c >0$, if we suppose that $X \in \tilde X \pm 2 e^{-\, \cdot \, ^2/2}$ in $\left( \mathcal  M_{p,n}, \left\Vert   \cdot\right\Vert _F\right) $, then~:
    \begin{itemize}
      \item $\frac{XX^T}{ n^2 } \propto C e^{-c   \, \cdot \,/ \gamma} + C e^{-c ( n\,\cdot \,/\Vert \tilde X\Vert_F)^2}$  in $\left( \mathcal  M_{p,n}, \left\Vert   \cdot\right\Vert _F\right) $
      \item $\frac{ XX^T}{ n } \propto C e^{-c \, \cdot \,/ \bar \gamma} + C e^{-c ( n\,\cdot \,/\Vert \tilde X\Vert)^2}$  in $\left( \mathcal  M_{p,n}, \left\Vert   \cdot\right\Vert\right) $
      \item $\frac{ XX^T}{ \log n } \propto C \exp\left(- \frac{c \, \cdot}{1+\frac{\log p}{\log n}} \right)+ C e^{-c ( \sqrt{\log n}\,\cdot \,/\Vert \tilde X\Vert_\infty)^2}$  in $\left( \mathcal  M_{p,n}, \left\Vert   \cdot\right\Vert _\infty\right) $,
    \end{itemize}  
    where $\gamma=\frac{p}{n}$, $\bar \gamma= \gamma + 1 \geq \max(\gamma, 1)$ and $C\geq e$, $c>0$ are two numerical constants. 
\end{example}

As rich it could be the notion of linear concentration is insufficient when dealing with the resolvent of random matrices, starting with the resolvent of the sample covariance $Q = (XX^T/n + I_p)^{-1}$. It is possible to infer the concentration of the sample covariance  in $( \mathcal M_p, \left\Vert \cdot\right\Vert)$ from the concentration of $X$ as we saw in Example~\ref{exe:concentration_covariance_empirique}, we can even track the concentration of the resolvent since the inverse of a matrix of $\mathcal M_p$ can be written as a polynomial of degree $p$, but the observable diameter will then be of diverging order. To solve these issues, one needs a stronger notion of concentration to be introduced next : the \textit{Lipschitz} concentration.


\subsection{Lipschitz Concentration}

\subsubsection{Definition and fundamental examples}
Theorem~\ref{the:concentration_vecteur_spherique} given in the preamble provides us with the concentration of the Lipschitz and even uniformly continuous functionals of random vectors uniformly distributed on the sphere, that is to say, cases that go far beyond the linear case and that let us hope for interesting inference on the resolvent of random matrices. The class of Lipschitz concentrated vectors is a subclass of the class of linearly concentrated random vectors, so it satisfies more properties and is also far more stable. Indeed it supports any Lipschitz transformations but also concatenation (as depicted in Proposition~\ref{pro:concentration_(X,Y)_independant} for the linear concentration) and tensorial product under some assumptions (see Theorem~\ref{the:concentration_des_transformations_lipschitz_multilineairement} below).

As firstly evoked in Lemma~\ref{lem:conc_transfo_lipschitz}, the concentration of a random vector $Z$ can be expressed through the concentration of any random variable $f(Z) \in \mathbb{R}$ when $f$ is Lipschitz. As before, this approach of concentration has the asset of bringing back the concentration on any normed vector space to a mere concentration on $\mathbb{R}$, that we deeply studied at the beginning of the section. 
The following definition allows us to generalize the notion of concentration to any metric space as presented in \cite{Led01}.
\begin{definition}[Lipschitz Concentration of a random vector]\label{def:conc_lipschitzienne_vect_al}
  Given a concentration function $\alpha$, a random vector $Z$ is said to be Lipschitz \textit{$\alpha$-concentrated} iff one of the following three assertions is verified for any $1$-Lipschitz function $f : E \rightarrow \mathbb{R}$~:
  \begin{itemize}
     \item $f(Z) \propto \alpha, $ and we will note in that case $ Z \propto \alpha$\\
     \item $f(Z) \in m_f \pm \alpha, $ and we will note in that case $ Z \overset{m}{\propto} \alpha$\\
     \item $f(Z) \in \mathbb{E}[f(Z)] \pm \alpha, $ and we will note in that case $ Z \overset{\mathbb{E}}{\propto} \alpha$,
   \end{itemize} 
   where $m_f$ is a median of $f(Z)$.
\end{definition}
The Lipschitz concentration is the strongest notion of concentration we will present and it is the one that received most of the interest from the scientific community; therefore, we allow ourselves to omit the term ``Lipschitz'' when mentioning this kind of concentration.
In our paper, such Lipschitz concentration of random vectors can only be obtained through Theorem~\ref{the:concentration_vecteur_spherique} or Theorem~\ref{the:concentration_vecteur_gaaussien} below, and they are both set on the normed vector space $(\mathbb{R}^p, \left\Vert \cdot\right\Vert)$ (or the analogous one $(\mathcal M_{p,n}, \left\Vert   \cdot\right\Vert_F)$. We will thus allow ourselves to omit the precision about the normed vector space on which is made the concentration when we are on $(\mathbb{R}^p, \left\Vert \cdot\right\Vert)$, on $(\mathcal M_{p,n}, \left\Vert \cdot\right\Vert_F)$ or when we are on the formal normed vector space $(E, \left\Vert \cdot\right\Vert)$.


\begin{remark}\label{rem:justification_definition}
  We know from Lemma~\ref{lem:conc_transfo_lipschitz} that Definition~\ref{def:conc_lipschitzienne_vect_al} is compatible with Definition~\ref{def:conc_variable_aleatoire} when $E=\mathbb{R}$, so there are no conflicts between the different uses of the notation $\propto$ for random vectors and random variables.
\end{remark}
\begin{remark}\label{rem:lien_entre_diff_types_de_concentrations_lipschitz}
    Given a random vector $Z \in E$ and a concentration function $\alpha$, we know thanks to Lemma~\ref{pro:conc_med}~:
    \begin{align*}
        Z \propto \alpha
        & \Longrightarrow&
        &Z \overset{m}{\propto} 2\alpha&
        &\Longrightarrow&
        &Z \propto 4 \alpha(\cdot/2).
    \end{align*}
    Also, by definition of the linear concentration~:
    \begin{align*}
        Z \overset{\mathbb{E}}{\propto} \alpha&
        &\Longrightarrow&
        &Z \in \mathbb E Z \pm \alpha,
    \end{align*}
    and we can link this concentration to the other two thanks to Proposition~\ref{pro:mediane_pivot} and Corollary~\ref{cor:moyenne_pivot} in the case of a $q$-exponential concentration~:
    \begin{align*}
        Z \overset{m}{\propto} C e^{-(\cdot/\sigma)^q}&
        &\Longrightarrow&
        &Z \overset{\mathbb{E}}{\propto} e^{C^q/q}e^{-(\cdot/2\sigma)^q}&
        &\Longrightarrow&
        &Z \overset{m}{\propto} 2e^{C^q/q}e^{-(\cdot/4\sigma)^q},
    \end{align*}
    for any $q,\sigma>0$ and $C \geq e$.
\end{remark}


We thought useful to present the theorem of concentration of Gaussian vectors in our new formalism, to add this setting to the historical example of the concentration on the sphere.
\begin{theorem}\label{the:concentration_vecteur_gaaussien}
A canonical Gaussian vector $Z$ is normally concentrated independently of its dimension. For any $p \in \mathbb{N}$~:
  \begin{align*}
    Z \sim \mathcal N(0,I_p)&
    &\Longrightarrow& 
    &Z \overset{m}{\propto} 2 e^{-\, \cdot \,^2/2} \  \text{ and } \ Z \overset{\mathbb{E}}{\propto} 2 e^{-\, \cdot \,^2/2},
  \end{align*}
  where $\mathcal N(0,I_p)$ is the distribution of the \emph{canonical Gaussian vectors of dimension $p$} that have independent zero mean and unit variance Gaussian entries.
\end{theorem}
A structural proof with a geometrical approach from the Poincar\'{e} lemma tracks the concentration of Gaussian vectors from the concentration of the uniform distribution on the sphere. A more functional approach based on the log-Sobolev inequalities can be found in \cite{Led01}. An alternative proof is found in \cite{Tao11}, originally proposed by Maurey and Pisier, which does not provide the optimal constants but is more efficient and simple.

\begin{remark}\label{rem:observable_diameter_gaussienne}
  The independence of the concentration of a random vector to its dimension can be interpreted as a conservation of its \emph{observable diameter} through dimensionality when the \emph{actual diameter} increases. This second diameter, that can be referred to as the \emph{metric diameter}, can be naturally defined as the expectation of the distance between two independent random vectors drawn from the same distribution. Theorem~\ref{the:concentration_vecteur_gaaussien} states that the observable diameter of a Gaussian distribution in $\mathbb{R}^p$ is of order $1$, that is to say $\frac{1}{\sqrt{p}}$ times less than the diameter (that is of order $\sqrt{p}$). The same result holds for the uniform distribution on the sphere of $\mathbb{R}^p$ and for any distribution that would be called for that reason \emph{concentrated}.
\end{remark}

Theorem~\ref{the:concentration_vecteur_gaaussien} can be generalized to any random vector $X \in \mathbb R^p$ with density $d\mathbb P_X(x) = e^{-U(x)}d\lambda_p(x)$ where $U : \mathbb R^p \rightarrow \mathbb R$ is a positive functional with the hessian bounded inferiorly by, say $cI_p$, $c>0$. In that case, $X \propto 2e^{-c \cdot^2/2}$, (see \cite[Theorem 2.7]{Led01})   

Let us add for the general picture a result from Talagrand \cite[Theorem 2.4]{TAL94} (or \cite[Proposition 4.18]{Led01}) concerning the concentration of the exponential distribution, that we shall denote $\nu^p$, which is the distribution of random vectors of $\mathbb R^p$ with independent entries having density $\frac 1 2 e^{-\vert \cdot \vert}d\lambda_1$. 

\begin{theorem}\label{the:conc_distribution_exponentielle}\cite[Proposition 4.18]{Led01}
There exist two numerical constants $C\geq 1$ and $c>0$, such that for any $p\in \mathbb{N}$~:
  \begin{align*}
      Z \sim \nu^p&
  &\Longrightarrow&
  &Z \propto C e^{-c\, \cdot} .
  \end{align*}
\end{theorem}

As an example of $q$-exponential concentrations when $q\in [1,2]$, one may consider vectors uniformly distributed on the balls of $\mathbb R^p$, that is $\mathcal B_{\Vert\cdot \Vert_q} = \{x \in \mathbb R^p \ \vert \ \Vert x \Vert_q = (\sum x_i^q)^{1/q}\leq 1\}$; let us note $\beta_q^p$ this distribution.

\begin{theorem}\label{the: concentration_dist_unif_boule} \cite[Proposition 4.21]{Led01}
Given $q\in[1,2]$, there exist two numerical constants $C\geq 1$ and $c>0$, such that for any $p\in \mathbb{N}$~:
  \begin{align*}
      Z \sim \beta_q^p&
  &\Longrightarrow&
  &Z \propto C e^{-cp\, \cdot^q} .
  \end{align*}
\end{theorem}


Definition~\ref{def:conc_lipschitzienne_vect_al} only presents the concentration of Lipschitz functionals of $Z$ (if $f$ is $\lambda$-Lipschitz then $f/\lambda$ is $1$-Lipschitz and the product with a constant is easy to manage, we find $f(Z)\propto \alpha(\, \cdot \, /\lambda)$), but it is possible to show the concentration of any uniformly continuous functional of $Z$~:

\begin{proposition}[Concentration of the uniformly continuous transformations]\label{pro:conc_fonctionnelles_holderienne}
  Given two normed vector spaces $E$ and $G$, a random vector $Z\in E$, a concentration function $\alpha$, a modulus of continuity $\omega$, a function $\phi : E \rightarrow G$, $\omega$-continuous, we have the implication~:
  \begin{align*}
    Z \overset{m}{\propto} \alpha &
    &\Longrightarrow&
    &\phi(Z) \overset{m}{\propto} \alpha(\omega^{-1}(\, \cdot \, )).
  \end{align*}
\end{proposition}
\begin{proof}
  Let us introduce a $1$-Lipschitz function $g : G \rightarrow \mathbb R$, we note $f = g \circ \phi$. We introduce $m_f$, a median of $f(Z)$ and the sets $A_-=\{z,f(z)\leq m_f\}$ and $A_+=\{z,f(z)\geq m_f\}$ (they verify by definition $\mathbb{P}(Z\in A_+),\mathbb{P}(Z\in A_-)\geq \frac{1}{2}$). 
  The image through $f$ of the boundary $\partial A_+ = \partial A_- = A_+ \cap A_-$ is equal to $\left\{m_f\right\}$. Since the boundary is closed, for any $z \in A_+$, there exists a sequence $z_n\in \partial A_-$ such that $\left\Vert z-z_n\right\Vert \underset{n \rightarrow \infty}{\longrightarrow} d(z, \partial A_-)=d(z, A_-)$, then since $f$ is uniformly continuous like $\phi$, we can bound~:
  \begin{align*}
     \left\vert f(z)-m_f \right\vert 
     &= \lim_{n \rightarrow \infty}\left\vert f(z)-f(z_n)\right\vert \leq  \lim_{n \rightarrow \infty} \omega (\left\Vert z -z_n\right\Vert ) \\
     &\leq\omega (d(z,A_-)) =\omega(\left\vert d(z,A_-)-d(z,A_+)\right\vert) ,
  \end{align*}
  and the same inequality is also verified for any $z\in A_-$. This entails~:
  \begin{align*}
    \mathbb{P}\left( \left\vert f(Z) -m_f\right\vert\geq  t\right) 
    &\leq  \mathbb{P}\left( \left\vert d(Z,A_-) -d(Z,A_+)\right\vert \geq  \omega^{-1}\left(t\right)\right),
  \end{align*}
  and we can then conclude since $\tilde f : z \mapsto d(z,A_-) - d(z, A_+)$ is $1$-Lipschitz and $\tilde f(Z)$ admits $0$ as a median. 
\end{proof}

  \begin{remark}\label{rem:concentration_transformation_hold_gaussien}
    Theorems~\ref{the:concentration_vecteur_spherique}, \ref{the:concentration_vecteur_gaaussien}, \ref{the:conc_distribution_exponentielle}~and~\ref{the: concentration_dist_unif_boule} combined with Proposition~\ref{pro:conc_fonctionnelles_holderienne} give us immediately a $q$-exponential concentration of all the random vectors $F(Z)$ where $F : \mathbb{R}^p \rightarrow \mathbb{R}^d$ is uniformly continuous and $Z \sim \sigma_p$, $Z\sim\mathcal N(0,I_p)$, $Z\sim\nu^p$, $Z\sim \beta_q^p$. This describes a wide range of random vectors that remarkably do not need to present independent entries.
\end{remark}



\subsubsection{Concentration of the concatenation}
We want to answer the question: if $X$ and $Y$ are concentrated, is $(X,Y)$ also concentrated with a comparable observable diameter ? In general it is false, as one can see in Example~\ref{ex:contre_ex_conc_(X,Y)}. 
We first present a way to infer the concentration of $(X,Y)$ when $X$ and $Y$ are independent and we don't make further assumptions on their concentration functions (the particular case of the $q$-exponential concentration is studied in Appendix~\ref{ssa:conc_conc_conv}).

To get interesting inferences from the concentration of the measure theory, one has to base its initial inequalities on a theorem of the kind of Theorem~\ref{the:concentration_vecteur_spherique}, \ref{the:concentration_vecteur_gaaussien} or~\ref{the:conc_distribution_exponentielle} providing distributions with an observable diameter far smaller than the metric parameter (see Remark~\ref{rem:observable_diameter_gaussienne} for precisions). The following proposition is just a tool that can be some help when one wants to study a \textit{limited} (with regard to the dimension) concatenation of concentrated vectors.
In what follows $E$ and $F$ are two normed vector spaces respectively equipped with the norms $\left\Vert \cdot\right\Vert_E$ and $\left\Vert \cdot\right\Vert_F$. We again note $\left\Vert \cdot\right\Vert_{\ell_1}$ the norm of $E\times F$ defined as $\left\Vert (x,y)\right\Vert_{\ell_1} = \left\Vert x\right\Vert_E + \left\Vert y\right\Vert_F$.
\begin{proposition}\label{pro:concentration_(X,Y)_independant}
  Given two independent random vectors $X \in E$ and $Y\in F$, if we suppose that $X$ and $Y$ are concentrated then $(X,Y)$ is also concentrated. Given two concentration functions $\alpha,\beta : \mathbb{R}_+ \mapsto \mathbb R_+$, and any $\lambda \in (0,1)$~:
  \begin{align*}
    \left\{\begin{aligned}
      &X \propto  \alpha \\
      &Y \propto  \beta
    \end{aligned}\right.&
    &\Longrightarrow&
    &(X,Y) \propto  \alpha \left( \lambda \, \cdot \, \right) + \beta \left( (1- \lambda) \, \cdot \, \right),  
  \end{align*}
  If we suppose that $\alpha$ is invertible and piecewise differentiable, we also have the implication~:
  \begin{align*}
    \left\{\begin{aligned}
      &X \propto \alpha \\
      &Y \propto \beta
    \end{aligned}\right.&
    &\Longrightarrow&
    &(X,Y) \propto  \alpha  + \beta -\alpha' * \beta  ,  
  \end{align*}
  where $*$ is the convolution operator ($f*g(t) = \int_{\mathbb R} f(u)g(t-u) du$). Since $\alpha$ and $\beta$ are only defined on $\mathbb R_+$, we implicitly compute the convolution with a null continuation of $\alpha$ and $\beta$ on $\mathbb R_-$. The implications are also true if we work with concentrations around the medians ($X \overset{m}{\propto} \alpha$) or around the means ($X \overset{\mathbb{E}}{\propto} \alpha$).
\end{proposition}
In the second result, be careful that, since the concentration functions are decreasing, $\alpha'$ is a negative function.
\begin{proof}
  For the proof of the first implication, refer to \cite[Proposition 1.11]{Led01}. Let us consider a $1$-Lipschitz function $f : E \times F \rightarrow \mathbb{R}$. We will work with the concentration around the means since it is easier and we will note for simplicity $\mathbb{E}f = \mathbb{E}\left[f(X,Y)\right]$ and $\mathbb{E}\left[f | Y\right] = \mathbb{E}\left[f(X,Y) \ | \ Y\right]$ ; we can bound~:
  \begin{align*}
    \mathbb{P}\left(\left\vert f(X,Y) - \mathbb{E}f\right\vert\geq t\right)
    &\leq \mathbb{E}\left[\mathbb{P}\left(\left\vert f(X,Y) - \mathbb{E}\left[f| Y\right]\right\vert\geq t - \left\vert \mathbb{E}f - \mathbb{E}\left[f | Y\right]\right\vert\ | \ Y\right)\right]\\
    &\leq \mathbb{E}\left[\alpha\left(t-\left\vert \mathbb{E}f - \mathbb{E}\left[f | Y\right]\right\vert\right)\right] \, + \,  \mathbb{P}\left(\left\vert \mathbb{E}f - \mathbb{E}\left[f | Y\right]\right\vert\geq t\right) \\
    & \leq \int_0^1 \mathbb{P}\left(\alpha \left(t - \left\vert \mathbb{E}f - \mathbb{E}\left[f | Y\right]\right\vert\right) \geq u\right) du \, + \beta(t)\\
    & \leq \int_0^1 \mathbb{P}\left(t - \left\vert \mathbb{E}f - \mathbb{E}\left[f | Y\right] \right\vert\leq \alpha^{-1}(u)\right) du \, + \beta(t)\\
    &\leq \int_{\alpha(t)}^1 \beta\left(t - \alpha^{-1}(u)\right) du \, + \alpha(t) + \beta(t) \\
    &=\int_{t}^0 \alpha'(u)\beta\left(t - u \right) du \, + \alpha(t) + \beta(t).
  \end{align*}
\end{proof}
\begin{remark}\label{rem:diametre_observable_generalise}
  The two results of Proposition~\ref{pro:concentration_(X,Y)_independant} can be compared if we consider a generalization of the notion of observable diameters, as defined for the exponential concentration before Lemma~\ref{lem:comportement_pivot_equivalent_deterministe}, and that is slightly different from the observable diameter introduced by Gromov in \cite[Chapter~3.1/2]{gro99}. Still, it is of the same order when the dimension is large. Given a concentration function $\alpha$, we note $\mathcal R_\alpha=\int_0^\infty \alpha,$ and for any random vector $X$ the observable diameter $\mathcal R_X$ is defined as $\mathcal R_X= \inf\{\mathcal R_\alpha \ |\ X \propto \alpha\}$. Our definition comes from the fact that if, say, $X \propto \alpha$, then for any $1$-Lipschitz function $f$ and any independent copy $X'$~:
  \begin{align*}
    \mathbb{E}\left[\left\vert f(X) - f(X')\right\vert\right] = \int_0^\infty \mathbb{P}\left(\left\vert f(X) - f(X')\right\vert\geq t\right) dt \leq \int_0^\infty \alpha = \mathcal R_\alpha.
  \end{align*}
  With the first result given by Proposition~\ref{pro:concentration_(X,Y)_independant}, we find~:
  \begin{align*}
    \mathcal R_{(X,Y)} \leq  \frac{\mathcal R_X}{1-\lambda} + \frac{\mathcal R_Y}{\lambda},
  \end{align*}
  and with the second result~:
  \begin{align*}
    \mathcal R_{(X,Y)} \leq \mathcal R_X + 2 \mathcal R_Y,
  \end{align*}
  since for any differentiable concentration functions $\alpha,\beta$~: $$\int_0^\infty \alpha' * \beta = \int_{-\infty}^{+\infty}\alpha'*\beta = \int_{-\infty}^{+\infty}\alpha' \int_{-\infty}^{+\infty}\beta = - \alpha(0) \mathcal R_\beta.$$
  The second inequality is clearly better than the first one, however we are far from the stability of the observable diameter that we find in Theorems~\ref{the:concentration_vecteur_spherique} and~\ref{the:concentration_vecteur_gaaussien}.
\end{remark}
One can assert easily that the concentration could be generalized to non-independent random vectors since we know for instance that $(X,f(X))$ is concentrated as a $2$-Lipschitz transformation of $X$ when $f : E \rightarrow E$ is $1$-Lipschitz. However, there exists a lot of examples where $(X,Y)$ is far from being concentrated despite $X$ and $Y$ being concentrated.

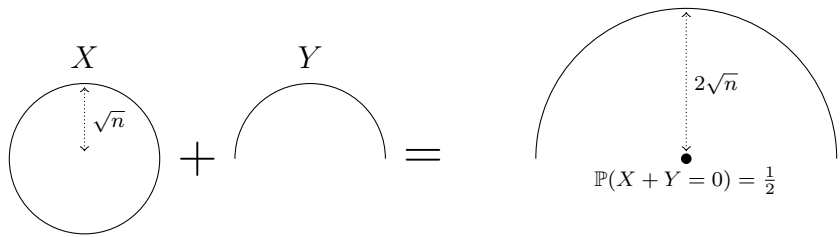
\begin{figure}[!ht]
\begin{center}

\begin{tikzpicture}
\draw (0,1) circle (1cm);
\draw[densely dotted,<->] (0,1.1) -- (0,1.95) ;
\draw (0,1.5) node[anchor = west] {$\sqrt{n}$};
\draw (0,2.1) node[anchor=south] {\large $X$};
\draw (1.5,1) node {\huge $+$};
\draw (4,1) arc (0:180:1cm);
\draw (4.5,1) node {\huge $=$};
\draw (3,2.1) node[anchor=south] {\large$Y$};
\draw (10,1) arc (0:180:2cm);
\draw[densely dotted,<->] (8,1.1) -- (8,2.95) ;
\draw (8,2) node[anchor = west] { $2\sqrt{n}$};
\fill (8,1) circle[radius=2pt] node[anchor=north]{ $\mathbb P(X+Y=0) = \frac{1}{2}$};

\end{tikzpicture}
\end{center}
\caption{The sum of two concentrated random vectors can be non concentrated}
\label{fig:somme_non_conc}
\end{figure}
\begin{example}\label{ex:contre_ex_conc_(X,Y)}
  Given $p\geq 0$, let us consider a random vector $Y\sim \sigma_p$. We define the random vector $X$ as being equal to $Y$ on $\sqrt{p} \, \mathbb{S}^p_+ = \sqrt{p} \, \mathbb{S}^p \cap (\mathbb R_+ \times \mathbb R^p) $ and equal to $-Y$ on $\sqrt{p} \, \mathbb{S}^p_- = \sqrt{p} \, \mathbb{S}^p \cap (\mathbb R_- \times \mathbb R^p)$. Note that $X\sim \sigma_p \sim 2 \un_{\sqrt{p}\mathbbm{S}^p_+}\sigma_p$  and $Y$ are both $2 e^{- \, \cdot \, ^2/2}$-concentrated (see Theorem~\ref{the:concentration_vecteur_spherique}) and therefore $\mathcal R_X \leq R_Y \leq \mathcal R_Y \leq \int_0^\infty 2 e^{-t^2/2} dt= \sqrt{2\pi}$. As we see on the schematic Figure~\ref{fig:somme_non_conc}, the distribution of $X+Y$ is completely different from the distribution of $X$ and $Y$ since it is discontinuous. If one looks at the variations of the random variable $\left\Vert X +Y\right\Vert$ that is a $1$-Lipschitz functional of $(X,Y)$, one notes that $0$ is a median of $\left\Vert X + Y\right\Vert$ and we have~:
  \begin{align*}
    \mathbb{P} \left(\left\Vert X+Y\right\Vert \geq t \right) =
    \left\{\begin{aligned}
      & 1 \text{ if } t=0 \\
      & \frac{1}{2} \text{ if } t \in \left(0,\sqrt{2p}\right]\\
      & 0\text{ if } t >\sqrt{2p}.
    \end{aligned} \right.
  \end{align*} 
  Therefore, with the notation of the observable diameter we introduced in Remark~\ref{rem:diametre_observable_generalise}, we see that~:
  \begin{align*}
    \mathcal R_{(X,Y)} \geq \mathcal R_{\left\Vert (
    X+Y\right\Vert} = \int_0^\infty  \mathbb{P} \left(\left\Vert X + Y\right\Vert \geq t \right) dt = \sqrt{\frac{p}{2}} > 4 \sqrt{2\pi} \geq 2(\mathcal R_X + \mathcal R_Y)  .
  \end{align*}
  For $p$ sufficiently large, it contradicts the inequalities given in Remark~\ref{rem:diametre_observable_generalise} and we see in particular that the random vector $(X,Y)$ has an observable diameter of the same order as the metric diameter which is of order $\sqrt{p}$. The vector $(X,Y)$ is not concentrated in the sense given by Remark~\ref{rem:observable_diameter_gaussienne}.
\end{example}

\subsubsection{Concentration of the product}
 We are now going to look at the concentration of a tensorial product of concentrated vectors (if $E$ is an algebra, that will imply in particular the concentration of a product of vectors). 
Basically, the concentration of products will be shown with two complementary arguments:
\begin{itemize}
  \item products are Lipschitz on balls,
  \item there are few drawings of a concentrated random vector outside of a ball of sufficient size.
\end{itemize}

To present a picture as general as possible, we add in the hypotheses the existence of supplementary norm (or a seminorm) $\|\cdot\|'$ smaller than the norm $\|\cdot\|$ where the concentration happens, allowing us, when it controls the product to sharpen the concentration rate. 
\begin{theorem}\label{the:concentration_des_transformations_lipschitz_multilineairement}
Let us consider three normed vectors spaces $(E,\|\cdot \|_E)$, $(F,\|\cdot \|_F)$ and $(G,\|\cdot \|_G)$ and a product $\cdot : E\times F \to G$
satisfying:
\begin{align*}
  \forall x\in E, \forall y \in F, \|x\cdot y \|_G \leq 
  \left\{\begin{aligned}
    &\|x\|'_E\|y\|_F\\
    &\|x\|_E\|y\|'_F
  \end{aligned}\right.
\end{align*}
where $\|\cdot \|_E'$ and $\|\cdot \|_F'$ are semi-norms respectively endowing $E $ and $F$.
Given two random vectors $X\in E$ and $Y \in F$, if $(X,Y) \propto C e^{-(t/\sigma)^q}$ in $(E\times F, \|\cdot \|_{\ell^\infty} )$ and $(E\times F, \|\cdot \|_{\ell^\infty}' )$, where for any $x,y \in E \times F$, $\|(x,y)\|_{\ell^\infty} = \sup(\|x\|_E,\|y\|_F)$ and $\|(x,y)\|'_{\ell^\infty} = \sup(\|x\|_E',\|y\|_F')$, then there exists two constants $C',c'$, respectively proportional to $C$ and $c$ such that:
\begin{align*}
  X \cdot Y \propto C e^{-(t/\sigma' \mu)^{\frac{q}{2}}}
 + C e^{-(t/\sigma'{}^2)^{\frac{q}{2}}},
\end{align*}
where $\mu = \sup(\mathbb E[\|X\|_E'], \mathbb E[\|Y\|_F'])$

\end{theorem}
{\ajout\begin{proof}
  Let us consider a $1$-Lipschitz (for the norm $\left\Vert \cdot\right\Vert$) function $f: G \rightarrow \mathbb{R}$ and $m_f$, the median of $f(X \cdot Y)$ 
  We wish to bound:
  \begin{align*}
    \mathbb{P}\left(\left\vert f(X \cdot Y) - m_f\right\vert\geq t\right).
  \end{align*}
  The map $(X,Y) \mapsto f(X \cdot Y)$ is not Lipschitz, unless $X$ and $Y$ are bounded. We thus decompose the probability argument into two events, one with bounded $\left\Vert (X,Y)\right\Vert_{\ell^\infty}'$  and the complementary with small probability.
  Since $\left\Vert X\right\Vert_E' \in \mathbb E[\left\Vert X\right\Vert_E'] \pm Ce^{-(\, \cdot \,/\sigma)^q}$ and $\mathbb E[\left\Vert X\right\Vert_E']\leq \mu$, and the same holds for $Y$, we know from Lemma~\ref{lem:conentration_superieure} that:
  \begin{align*}
     &\mathbb P \left((\left\Vert (X,Y)\right\Vert_{\ell^\infty}' \geq t\right) \\
     &\hspace{0.5cm}\leq \mathbb P \left(\left\vert \left\Vert X\right\Vert_E' -\mathbb E[\left\Vert X\right\Vert_E']\right\vert\geq t-\mu\right)+ \mathbb P \left(\left\vert \left\Vert Y\right\Vert_F' - \mathbb E[\left\Vert Y\right\Vert_F']\right\vert \geq t-\mu\right) \\
     &\hspace{0.5cm}\leq 2Ce^{-(t- \mu)^q/\sigma^q}.
  \end{align*}
  Let us introduce 
  $K \geq 2\mu$ and the set
  \begin{align*}
    \mathcal A_K = \left\{(x,y) \in E \times F\ | \ \left\Vert x\right\Vert_E' \leq K  , \left\Vert y\right\Vert_F' \leq K \right\}.
  \end{align*}
  We already know, on the one hand that $K-\mu \geq \frac{K}{2}$ and therefore:
  \begin{align}\label{eq:bound_event_1}
    \mathbb{P} \left(\left\vert f(X \cdot Y) - m_f\right\vert\geq t, (X,Y)\in\mathcal A_K^c\right) 
    \leq \mathbb{P} \left(Z \in \mathcal A_K^c\right)
    \leq 2Ce^{-(K/2)^q/\sigma^q}
  \end{align}
  On the other hand, if we note $\phi:E\times F \to G$ defined for all $(x,y) \in E \times F$ as $\phi(x,y) = x\cdot y$, the mapping $\restriction{f\circ \phi}{\mathcal A_K}$ is $2K$-Lipschitz.
  Therefore, noting $\mathcal C = \{z \in E\times F \ | \ f(\phi(z)) \leq m_f\}$ and for $\mathcal B \subset E$, $\mathcal B \neq \emptyset$, $d(\cdot,\mathcal B) : z \mapsto \inf\{\|z-b\|, b\in \mathcal B\}$, for all $z\in\mathcal A_K$:
  \begin{align*}
    \left\vert f(\phi(z)) - m_f\right\vert\geq t&
    &\Longrightarrow&
    &d(z, \mathcal C)\geq \frac{t}{K^{(p-1)}} \  \text{ or } \ d(z, \mathcal C^c)\geq \frac{t}{K^{(p-1)}}.
  \end{align*}
  Since $d((X,Y), \mathcal C)$ and $d((X,Y), \mathcal C^c)$ are both $1$-Lipschitz functionals of $(X,Y)$ and admit $0$ as a median, we can bound,
  \begin{align}\label{eq:bound_event_2}
    &\mathbb{P}\left(\left\vert f(X \cdot Y) - m_f\right\vert\geq t, (X,Y)\in\mathcal A_K\right)\\
    &\hspace{1.5cm}\leq \mathbb{P}\left( d((X,Y), \mathcal C)\geq \frac{t}{2}\right) + \mathbb P \left(d((X,Y), \mathcal C^c) \geq \frac{t}{2}\right)\nonumber \
    \leq 2 Ce^{-(t/2K\sigma)^q}.
  \end{align}
  Now, if we choose $K = \max \left(2\mu_i,t^{\frac{1}{2}}\right)$, we know that:
  \begin{align*}
    &\left(\frac{t}{\sigma K}\right)^q = \min \left(\left(\frac{t}{\sigma^2}\right)^{\frac{q}{2}}, \left(\frac{t}{ 2\sigma\mu}\right)^{q}\right).
  \end{align*}
  Therefore, putting \eqref{eq:bound_event_1} and \eqref{eq:bound_event_2} together, we obtain
  \begin{align*}
    \mathbb{P}\left(\left\vert f(X\cdot Y) - m_f\right\vert\geq t\right)
    &\leq  C'e^{(t/\sigma'{}^2)^{q/2}}+C'e^{-(t/\sigma'{}^2 \mu)^q}, 
  \end{align*}
  for $\sigma' = \max(2^{\frac{1}{q}},4)\sigma$ and $C' = 2C$. 
\end{proof}}
This theorem is particularly interesting when $\eta_{\|\cdot\|_E'} \ll \eta_{\|\cdot\|_E}$ and $\eta_{\|\cdot\|_F'} \ll \eta_{\|\cdot\|_F}$. 
For instance on $\mathbb R^p$:
\begin{align*}
    \eta_{\Vert \cdot \Vert_\infty} = \log p < p =\eta_{\Vert \cdot \Vert},
  \end{align*}
and on $\mathcal M_{q,n}$:
  \begin{align*}
    \eta_{\Vert \cdot \Vert_\infty}=\log(pn) < \eta_{\Vert \cdot \Vert}=p+n < pn =\eta_{\Vert \cdot \Vert_F}.
\end{align*}
The next examples take advantage of those inequalities.

\begin{example}\label{exe:concentration_covariance_empirique}
   Consider two random vector $Y,Z \in\mathbb{R}^p$ and a random matrix $X \in\mathcal M_{pn}$ satisfying $Y,Z, X\propto 2 e^{-(\cdot/2)^2}$ for the Euclidean norm (the Frobenius norm in $\mathcal M_{pn}$) and such that $\mathbb E [Y]=\mathbb E [Z]=0$ and $\mathbb E [X]=0$. Then,
   \begin{itemize}
      \item $Y \odot Z \propto  Ce^{-c \, \cdot }+Ce^{-c\, \cdot ^2/\log p} $ in $(\mathbb{R}^p, \Vert \cdot \Vert)$ 
      \item $\frac{XX^T}{n } \propto Ce^{-cn\,  \cdot}+Ce^{-cn \, \cdot^2/\bar \gamma}$  in $\left( \mathcal  M_{p,n}, \left\Vert   \cdot\right\Vert _F\right) $
    \end{itemize}  
    where $C \geq 1$ and $c>0$ are two numerical constants, $\gamma = \frac pn$ and $\bar \gamma = \gamma + 1$. One can note in these examples that the Lipschitz concentration preserves the concentration rates better than the linear concentration through the product of random vectors. Multiplying two vectors entry-wise let almost unmodified the concentration rate. The last result states that the sample covariance matrix is a highly concentrated object.
\end{example}

An important consequence to Theorem~\ref{the:concentration_des_transformations_lipschitz_multilineairement} is a result of concentration of bilinear mappings $X^TAY$ with $X,Y\in \mathcal M_{p,n}$ random and concentrated and $A \in \mathcal M_{p}$ deterministic. 
\begin{theorem}\label{the:conc_forme_quadrat_hypo_conc_lipsch}
  Let us consider three integers $q,m,n\in \mathbb{N}$, two random matrices $X, Y \in \mathcal M_{p,n}$, two positive constants $C \geq e$ and $\sigma>0$ and a matrix $A \in \mathcal M_{p}$. 
  If we suppose that $(X,Y) \propto C e^{- (\, \cdot\,/\sigma )^q}$ and $ \| \mathbb E[ X ]\|, \| \mathbb E[ Y ]\|= 0$\footnote{It is possible to relax this hypothesis if one wants to show the concentration of $\tr (BXAY)$ for a given $B \in \mathcal M_n$, then it suffices to suppose that with the notations introduced in the proof, one has $\| \mathbb E[\check X ]\|_\infty, \| \mathbb E[\check Y ]\|_\infty = O(\log(pn)^{1/q})$} then we have~:
  \begin{align*}
    X^TAY \in C'\exp \left( - \left( \frac{\, \cdot \,} { c\sigma^2\left\Vert A\right\Vert_F \log(pn)^{1/q}}\right)^q\right) + C'e^{-  (\,  \cdot \,/c \left\Vert A\right\Vert_F\sigma^2)^{\frac{q}{2}} },
  \end{align*}
  in $(\mathcal M_{n}, \|\cdot \|_F)$, for some $C',c>0$ depending numerically on $C$ and $q$.
\end{theorem}
Note that this is a result of linear concentration, we are not able to apply Theorem~\ref{the:concentration_des_transformations_lipschitz_multilineairement} directly to the random matrix $X^TAY$ but we can still invoke it for $\tr(BX^TAY)$ for any deterministic matrix $B \in \mathcal M_{n}$.
\begin{proof}
  Let us consider a deterministic matrix $B \in \mathcal M_{n}$ satisfying $\|B\|_F \leq 1$ and try to express the concentration of $\tr(BX^TAY)$. For that purpose let us decompose
  \begin{align*}
     A= P_A \Lambda Q_A&
     &\text{and}&
     &B= P_B \Gamma Q_B,
   \end{align*} with $P_A, Q_A \in \mathcal O_p$, $P_B, Q_B \in \mathcal O_n$ orthogonal and $\Lambda \in \mathcal D_p^+ $, $\Gamma \in \mathcal D_n^+ $ diagonal. Let us introduce the matrices $\check Y,\check X  \in \mathcal M_{p,n}$ defined as $\check X= P_A Z Q_B$ and $\check Y = Q_AY P_B$; since $\|P_A \|,\|Q_A \|,\|P_B \|,\|Q_B \| = 1$, we know that $\check X, \check Y \propto C e^{- (\, \cdot\,/\sigma )^q} $ and therefore, we know from Theorem~\ref{the:concentration_des_transformations_lipschitz_multilineairement} that:
  \begin{align*}
    \check X \odot \check Y \propto C'e^{- \, (\cdot \ / c\sigma^2)^{q/2}}+C'e^{-(\, \cdot \,/c\sigma^2)^q/\log (pn)}  \ \ \text{in } \ (\mathcal M_{p,n},\|\cdot \|_F),
  \end{align*}
  for some numerical constants $c, C' >0$. Now, noting $\lambda = (\lambda_1,\ldots,\lambda_p) \in \mathbb R^p$ and $\gamma = (\gamma_1,\ldots,\gamma_p) \in \mathbb R^p$ the vectors such that $\Lambda = \diag(\lambda)$ and $\Gamma = \diag(\gamma)$, we see that:
  \begin{align*}
    \tr(BX^TAY) = \tr(\Gamma \check X^T \Lambda \check Y)  = \sum_{\genfrac{}{}{0pt}{2}{1\leq i \leq p}{1\leq j \leq n}} \lambda_i \gamma_j \check X_{i,j} \check Y_{i,j} = \gamma^T(\check X\odot \check Y)\lambda,
  \end{align*}
  which is a $\|\lambda\|\|\gamma\|$-Lipschitz functional of $\check X \odot \check Y \in (\mathcal M_{p,n}, \| \cdot\|_F)$ (for the spectral norm so in particular for the Frobenius norm). Since $\|\lambda\| = \|A\|_F$ and $\|\gamma\| = \|B\|_F$, we directly deduce the result of the Theorem.


\end{proof}
\begin{remark}\label{rem:Hanson-Write}
  In the previous theorem, the case $n=1$, gives us result similar to Hanson-Wright Theorem (as found in \cite[Theorem 6.2.1]{Ver17} for instance). Noting $Z =X = Y \in \mathbb R^p$, we have indeed:
  \begin{align*}
    Z^TAZ \in \tr(\mathbb E[ZZ^T] A) \pm C'\exp \left( - \left( \frac{\, \cdot \,} { c\sigma^2\left\Vert A\right\Vert_F \log(p)^{1/q}}\right)^q\right) + C'e^{-  (\,  \cdot \,/c \left\Vert A\right\Vert_F\sigma^2)^{\frac{q}{2}} }.
  \end{align*}
  The original Hanson-Wright inequality does not take as hypothesis the concentration of the whole vector $Z=(z_1,\ldots, z_p)$ but just the concentration of each one of its coordinates $z_i$. However, it assumes that the different coordinates of $Z$ are independent which is quite a strong hypothesis and also that their means are equal to zero. The concentration result obtained under these hypotheses is not exactly the same, and relies on a quantity $K$ that could be seen as the maximum tail parameter of the $\{z_i\}_{1\leq i\leq p}$ ($K=\sigma$ in our case). The Hanson-Wright concentration can indeed be written~:
  \begin{align*}
    \mathbb{P}\left(\left\vert Z^TAZ - \mathbb{E}Z^TAZ\right\vert \geq t\right) \leq 2 \exp \left(-c \min\left(\frac{t^2}{K^4 \left\Vert A\right\Vert_F^2},\frac{t}{K^2 \left\Vert A\right\Vert}\right)\right),
  \end{align*}
  where $K=\max \{\left\Vert z_i\right\Vert_{\psi_2}\}_{1\leq i \leq p}$ is the maximum of the Orlicz norms defined as~:
  \begin{align*}
    \left\Vert z\right\Vert_{\psi_2} = \inf\{t>0 : \mathbb{E}\psi_2(\left\vert z\right\vert/t) \leq 1\}
  \text{ \ \ \ with \ \ $\psi_2(x)=e^{x^2} -1$}.
  \end{align*}
  To compare the classical result with the one of Theorem~\ref{the:conc_forme_quadrat_hypo_conc_lipsch}, let us consider $z \in \mathbb R^p$ such that $Z \propto Ce^{-t^2}$. Then $K$ is basically of order $O(1)$ and the tail parameters of the $1$-exponential component is far smaller in the classical result (generally $\|A\| \ll \|A\|_F$). However we will see in Theorem~\ref{the:conc_forme_quadrat} in next subsection, that for $t$ sufficiently large: 
  \begin{align*}
    \mathbb{P}\left(\left\vert Z^TAZ - \mathbb{E}Z^TAZ\right\vert \geq t\right) \leq C e^{-t/c\|A\|},
  \end{align*}
  for some constants $C,c >0$. The tail of the distribution of $Z^TAZ$ is thus exactly the same with the hypothesis of concentration on the whole vector $Z$. More importantly, note that the tail parameter of the $2$-exponential component which gives the order of the standard deviation of $Z^TAZ$ (and all of its centered moments) is very close to our result since $\sqrt{\log(p)}$ is almost a constant.
\end{remark}

\color{black}
In the preamble, we presented the resolvent as the convenient object that one wants to study to get some insight into the spectrum of the sample covariance of a concentrated random vector. Given a matrix $X\in \mathcal{M}_{p,n}$ and a positive real number $z>0$, recall that the resolvent $Q_S$ of the sample covariance matrix $S=\frac1nXX^T$ is defined as~:
\begin{align*}
  Q_{S}(z)= \left(S + z I_p\right)^{-1}.
\end{align*}

We will simply note it $Q$ to lighten the notations. Assuming that $X$ is concentrated, we saw in Example~\ref{exe:concentration_covariance_empirique} that $S$ is concentrated and consequently that $Q = (S+zI_p)^{-1}$ is concentrated as a $\frac 1 {z^2}$-Lipschitz transformation of $S$. That would however imply the appearance of a term $Ce^{-(c n\, \cdot^{q/2}}$ in the concentration function of $Q$ that we can remove if we directly show that $Q$ is a $\frac{2}{\sqrt{z^3n}}$-Lipschitz transformation of $X$.

\begin{proposition}\label{pro:concentration_resovante}
  Given $z>0$, a random matrix $X\in \mathcal{M}_{p,n}$ and a concentration function $\alpha$, we have the implication (recall that $\mathcal M_{p,n}$ are $\mathcal M_p$ both endowed with the Frobenius norm if no other norm is specified for the concentration)
  \begin{align*}
    X\propto \alpha&
    &\Longrightarrow &
    &Q=\left(\frac{1}{n}XX^T+zI_p\right)^{-1}\propto \alpha\left(\frac{\sqrt{nz^3}}{2} \, \cdot \,\right).
  \end{align*}
  In particular~:
  \begin{align*}
    X\overset{\mathbb{E}}{\propto} \alpha&
    &\Longrightarrow &
    &Q \in \mathbb{E}[Q] \pm \alpha\left(\frac{\sqrt{nz^3}}{2} \, \cdot \,\right).
  \end{align*}
 
 \end{proposition} 
 We need a preliminary lemma before giving the proof to control the Frobenius norm of a product~:
\begin{lemma}\label{lem:borne_norme_frob}
  Given $A \in \mathcal{M}_{pn}$ and $B \in \mathcal{M}_{np}$, one has the bound :
  \begin{align*}
    &\left\Vert AB\right\Vert_F \leq \left\Vert A\right\Vert \left\Vert B\right\Vert_F&
    &\text{and}&
    &\left\Vert AB\right\Vert_F \leq \left\Vert A\right\Vert_F \left\Vert B\right\Vert.&
  \end{align*}
  
\end{lemma}
One must be careful that in most cases $\left\Vert AB\right\Vert_F \neq \left\Vert BA\right\Vert_F$, which is why we need to display both inequalities. Recall in passing that the Cauchy-Schwarz inequality gives us directly $\left\Vert AB\right\Vert_F \leq \left\Vert A\right\Vert_F \left\Vert B\right\Vert_F$.
\begin{proof}
  Lemma~\ref{lem:borne_norme_frob} is just a consequence of the computations~:
  \begin{align*}
    \left\Vert AB\right\Vert_F^2 \leq \sum_{j=1}^p\left\Vert AB_{\cdot, j}\right\Vert_2^2 \leq \left\Vert A \right\Vert^2\sum_{j=1}^p \left\Vert B_{\cdot, j}\right\Vert_2^2 
    \leq  \left\Vert A \right\Vert^2 \left\Vert B\right\Vert_F^2 .
  \end{align*}
  The role of $A$ and $B$ could be inverted in the calculus without any problem.
\end{proof}
\begin{proof}[Proof of Proposition~\ref{pro:concentration_resovante}]
  The function $\phi : \mathcal{M}_{p,n} \rightarrow \mathcal M_p$ defined as $\phi(R)=\left(RR^T+zI_p\right)^{-1}$ is $2/z^{3/2}$-Lipschitz. Indeed, given $R,H \in \mathcal{M}_{p,n}^2$~:
  \begin{align*}
     \phi(R+H)-\phi(R) 
    &= \left((R+H)(R+H)^T+ z I_p\right)^{-1} -\left(RR^T+ z I_p\right)^{-1} \\
    &= \phi(R+H) \left((R+H)H +H R^T\right) \phi(R).
  \end{align*}
  Thus $ \left\Vert \phi(R+H)-\phi(R)\right\Vert_F\leq  2\left\Vert H \right\Vert_F/z^{3/2}$ thanks to Lemma~\ref{lem:borne_norme_frob} and to the basic result $\left\Vert \phi(R) R\right\Vert \leq 1/\sqrt{z}$ and $\left\Vert \phi(R)\right\Vert\leq \frac{1} {z}$ enunciated in Lemma~\ref{lem:controle_Q} in the preamble. Now, since $Q(z)=\phi(X/\sqrt{n})$, we recover directly the result of the proposition thanks to the hypothesis on $X$.
  The second implication is just a consequence of Remark~\ref{rem:lien_entre_diff_types_de_concentrations_lipschitz}.
\end{proof}

If we try to characterize geometrically, and roughly speaking, the range of random vectors concerned by Theorem~\ref{the:concentration_vecteur_spherique}, we would describe the set of respectful modifications of the sphere where bounded dilatations or the removals of some parts are tolerated but any cut is forbidden (we have about the same statement considering Theorem~\ref{the:concentration_vecteur_gaaussien}). This represents already a good range of distributions but one might be interested in representing discrete or at least ``discontinuous'' distributions. 

\subsection{Convex Concentration}\label{ssse:concentration_convexe}

\subsubsection{Definition and fundamental examples}

With a combinatorial approach, Talagrand showed in the nineties that it is possible to find a weaker notion of concentration to apprehend the concentration of partly discrete distributions. In these cases, to be concentrated the ``observation'' not only needs to be Lipschitz but also to be quasiconvex, in the sense of the following definition.
\begin{definition}[Quasiconvexity]\label{def:convexite_faible}
  A function $f : E \rightarrow \mathbb{R}$ is said to be quasiconvex if for any real $t\in \mathbb{R}$, the set $\{z\in E \ : \ f(z) \leq t\} = \{f \leq t\}$ is convex.
\end{definition}

\begin{remark}\label{rem:conv_faible_fonct_mon}
  Quasiconvexity concerns of course convex functions, but also any monotonous function supported on $\mathbb{R}$. More generally, given a convex function $f$ and a non decreasing function $g$, the composition $g \circ f$ is quasiconvex.
\end{remark}
 The class of quasiconvex functions is rather interesting in the sense that it is wider than the class of merely convex functions but still verifies the property of the uniqueness of the minimum.

\begin{definition}[Convex concentration]\label{def:conentration_faible}
  Given a random vector $Z\in \left(E, \left\Vert \cdot \right\Vert\right)$ and a concentration function $\alpha$, we say that $Z$ is \textit{convexly $\alpha$-concentrated} if one of the three assertions is verified for any $1$-Lipschitz and quasiconvex function $f : E \rightarrow \mathbb{R}$~:
  \begin{itemize}
     \item $f(Z) \propto \alpha, $ and we will note in that case $ Z \propto_c \alpha$\\
     \item $f(Z) \in m_f \pm \alpha, $ and we will note in that case $ Z \overset{m}{\propto}_c \alpha$\\
     \item $f(Z) \in \mathbb{E}[f(Z)] \pm \alpha, $ and we will note in that case $ Z \overset{\mathbb{E}}{\propto}_c \alpha$,
   \end{itemize} 
   where $m_f$ is a median of $f(X)$.
\end{definition} 
\begin{remark}\label{rem:equi_def_conc_conv}
  It is clear that the concentration of Definition~\ref{def:conc_lipschitzienne_vect_al} implies the convex concentration of Definition~\ref{def:conentration_faible}. Those two notions are equivalent when $E=\mathbb{R}$ since they are then both equivalent to Definition~\ref{def:conc_variable_aleatoire}.
\end{remark}
Once again, when non ambiguous, we will omit the precision ``in $(E, \left\Vert \cdot\right\Vert)$''. We clearly have the implication~:
\begin{align*}
  Z \propto \alpha
  \ \ \Longrightarrow \ \ 
  Z \propto_c \alpha.
\end{align*}
In the case of a $q$-exponential concentration, we have the implication chain~:
\begin{align*}
  Z \overset{\mathbb{E}}{\propto} C e^{-(\, \cdot \, /\sigma)^q}&
  &\Longrightarrow&
  &Z \overset{\mathbb{E}}{\propto}_c C e^{-(\, \cdot \, /\sigma)^q}&
  &\Longrightarrow&
  &Z \in \mathbb{E}Z \pm  e^{-(\, \cdot \,/\sigma)^q}.
\end{align*}

The fundamental example that alone justifies the interest in convex concentration is owed to Talagrand and provides to our study a supplementary setting to the ``smooth'' scenarios given by Theorems~\ref{the:concentration_vecteur_spherique} and~\ref{the:concentration_vecteur_gaaussien}. 

\begin{theorem}[Convex concentration of the product of bounded distributions, \protect{\cite[Theorem 4.1.1]{TAL95}}]\label{the:talagrand}
  Given a random vector $Z \in [0,1]^m$, $ m \in \mathbb{N}$, with independent entries~:
  \begin{align*}
    &Z \overset{m}{\propto}_c 4 e^{-\, \cdot \,^2/4}.
  \end{align*}
\end{theorem}

Considering the example of the preamble, we are tempted to look for a similar theorem where the sets $[0,1]$ are compact sets of $\mathbb{R}^p$ bounded by $K$ and $m$ is taken to be equal to $n$. It is not so interesting however, the issue being that the factor $K$ that would appear in the tail parameter would jeopardize most of the applications since it should be of order $1$ while $\left\Vert Z_i\right\Vert$ is often of order $\sqrt{p}$ when $Z_i \in \mathbb{R}^p$. This can however find some use when considering sparse random vectors.

The interesting setting is the case where $E=\mathbb{R}$, and $m=pn$. Then we do not exactly consider the concentration on $\mathbb{R}^{np}$ but the concentration on $\mathcal{M}_{p,n}$ (endowed with the Frobenius norm). Theorem~\ref{the:talagrand} gives us in that case a convenient tool to build convexly $q$-exponentially concentrated random matrices. In that case, the bound $\left\vert Z_i\right\vert \leq K$ on the entries of a random vector $Z$ is no more an unreachable hypothesis for applications; however we will need in that case an independence between the entries.


Theorems~\ref{the:concentration_vecteur_spherique} and~\ref{the:concentration_vecteur_gaaussien} allow us to track easily concentration properties from a vector $Z \in \mathbb{R}^p$ verifying $Z\sim \sigma_{p-1}$ or $Z\sim \mathcal N(0,I_p)$ to any Lipschitz transformation and even to any uniformly continuous transformation $f(Z)\in \mathbb{R}^q$. Theorem~\ref{the:talagrand} is not so easy to generalize because the convexity (or the quasiconvexity) of a function is only defined for real-valued functions; indeed, most of the transformations between two vector spaces ruin the subtle structure of convexity. We can still slightly relax the hypothesis of independence in the theorem of Talagrand thanks to affine transformations~:
\begin{lemma}\label{lem:transformations_faiblement_convexes}
  Given two vector spaces $E,F$ and a quasiconvex (resp., convex) function $f~:E\rightarrow \mathbb{R}$, for any affine function $g: F \rightarrow E$, the composition $f\circ g$ is also quasiconvex (resp., convex).
\end{lemma}

We can give a supplementary useful proposition that will allow us to keep the properties of convex concentration when thresholding a random vector. But before let us give an immediate preliminary lemma. We present it without proof since it is a direct consequence of Lemma~\ref{lem:conentration_superieure}. 
\begin{lemma}\label{lem:conc_superieure_vect_al}
  Given a random vector $Z$, an exponent $q>0$ and two positive constants $C\geq e$ and $\sigma>0$, if we note $m_{\left\Vert Z\right\Vert}$ a median of $\left\Vert Z\right\Vert$, we have the implication~:
  \begin{align*}
    Z \overset{m}{\propto}_c C e^{- ( \, \cdot\,/\sigma)^q}
    \ \ \Longrightarrow \ \
    \forall t \geq 2m_{\left\Vert Z\right\Vert} \ \ \mathbb{P}\left(\left\Vert Z\right\Vert\geq t\right) \leq Ce^{-\left(t/2\sigma\right)^q}.
  \end{align*}
\end{lemma}
The lemma displays a slight modification of the concentration constants, with here the behavior of the tail only beyond $2 \mathbb{E}\left\Vert Z\right\Vert$. The next proposition allows us to say that the concentration of a $q$-exponentially convexly concentrated random vector occurs under a threshold of order $2m_{\left\Vert Z\right\Vert}$.
\begin{proposition}\label{pro:concentration_thresholdee}
  Given a random vector $Z \in E$ a constant $K \geq 2 m_{\left\Vert Z\right\Vert}$, we introduce the vector $\bar Z=\min(1,\frac{K}{\left\Vert Z\right\Vert}) Z$. If there exist an exponent $q>0$ and two parameters $C\geq e$ and $\sigma>0$ such that $Z \overset{m}{\propto}_c C e^{-(\, \cdot \,/\sigma)^q}$ then $$\bar Z \propto_c 4C e^{-\left(\, \cdot \,/4\sigma\right)^q}.$$ 
\end{proposition}
If $Z \propto C e^{-(\, \cdot \,/\sigma)^q}$, it is possible to see $\bar Z$ as a $1$-Lipschitz transformation of $Z$ that would be naturally concentrated. However the hypothesis of quasiconvexity of the functionals $f(\bar Z)$ required by the convex concentration cannot be extrapolated so simply.

\begin{proof}
  Let us consider a function $f : E\rightarrow \mathbb{R}$ quasiconvex and $1$-Lipschitz. We know from Proposition~\ref{pro:conc_med} that $Z \propto_c 2C e^{-(\, \cdot \,/2\sigma)^q}$, thus introducing $ Z'$, an independent copy of $Z$, we can bound~:
  \begin{align*}
    \mathbb{P}\left(\left\vert f(\bar Z) - f(\bar Z')\right\vert\geq t\right)
    &\leq \mathbb{P}\left(\left\vert f(\bar Z) - f(\bar Z')\right\vert\geq t, \left\Vert  Z\right\Vert\leq K \text{ \ and \ } \left\Vert  Z'\right\Vert\leq K\right) \\
    &\hspace{0.5cm} + \mathbb{P}\left(\left\vert f(\bar Z) - f(\bar Z')\right\vert\geq t, \left\Vert  Z\right\Vert>K \text{ \ or \ } \left\Vert  Z'\right\Vert > K\right) \\
    &\leq \mathbb{P}\left(\left\vert f(Z) - f(Z')\right\vert\geq t\right) + \mathbb{P} \left(\left\Vert  Z\right\Vert>K \text{ \ or \ } \left\Vert  Z'\right\Vert > K\right) \\
    &\leq 2C e^{-(t/2\sigma)^q}+ 2C e^{- (K/2\sigma)^q}.
  \end{align*}
  The last inequality results from the hypothesis on $Z$ and $f$ and Lemma~\ref{lem:conc_superieure_vect_al}.
  By construction $\left\Vert \bar Z\right\Vert, \left\Vert \bar Z'\right\Vert\leq K$, thus $\left\Vert \bar Z - \bar Z'\right\Vert\leq 2K$ and $f$ being $1$-Lipschitz~:
  \begin{align*}
    \text{if $t> 2K$,} \  \ \ \mathbb{P}\left(\left\vert f(\bar Z) - f(\bar Z')\right\vert\geq t\right) \leq \mathbb{P}\left(\left\Vert \bar Z-\bar Z'\right\Vert\geq 2K\right) = 0 .
  \end{align*}
  Now, for any $t \leq 2K$~: $$\exp \left(- \left(\frac{K}{2\sigma}\right)^q\right) \leq   \exp \left(- \left(\frac{t}{4\sigma}\right)^q\right),$$ therefore, if we rejoin the different regimes, we obtain~:
  \begin{align*}
    \forall t>0 \ : \ \ \ \mathbb{P}\left(\left\vert f(\bar Z) - f(\bar Z')\right\vert\geq t\right) \leq 4C \exp \left(- \left(\frac{t}{4\sigma}\right)^q\right),
  \end{align*}
  and we can show exactly in the same manner that~:
  \begin{align*}
    \forall t>0 \ : \ \ \ \mathbb{P}\left(\left\vert f(\bar Z) - \mathbb{E}\left[f(\bar Z)\right]\right\vert\geq t\right) \leq 4C \exp \left(- \left(\frac{t}{4\sigma}\right)^q\right).
  \end{align*}
\end{proof}

\subsubsection{Hanson-Wright-like result for a convexly concentrated random vector}
Contrarily to the linear and Lipschitz concentrations, convex concentration does not seem to be preserved through the product.
To derive any results on the product from a convex concentration hypothesis, we will thus develop the habit of returning on the linear concentration configuration which is weaker but in a sense ``stable'' on algebra, like the Lipschitz concentration.
However, we still have the concentration of the norm as it is a Lipschitz and convex map. That makes a fundamental difference with the linear concentration and that entices a second Hanson-Wright-like result (see Remark~\ref{rem:Hanson-Write}).
It is also a good improvement to \cite[Lemma~8]{ELK09}~:
\begin{theorem}\label{the:conc_forme_quadrat}
  Let us consider two integers $q,m\in \mathbb{N}$ and a random vector $Z \in \mathbb{R}^p$, two positive constants $C \geq e$ and $\sigma>0$ and a matrix $A \in \mathcal M_{p}$. 
  If we suppose that $Z \overset{\mathbb{E}}{\propto}_c C e^{- (\, \cdot\,/\sigma )^q}$ then we have~:
  \begin{align*}
    Z^TAZ \in \tr(A \mathbb{E}[ZZ^T]) \pm 2C e^{- \left( \, \cdot \, / 4 \sigma\left\Vert A\right\Vert  \mathbb{E}\left\Vert Z\right\Vert \right)^q} + 2Ce^{-  (\,  \cdot \,/2 \left\Vert A\right\Vert\sigma^2)^{\frac{q}{2}} }.
  \end{align*}
  It is possible to replace the mean with a median if needed.
\end{theorem}

\begin{proof}
  Let us first assume that $A$ is symmetric nonnegative definite; in this case, $Z^T A Z= \Vert A^{\frac{1}{2}}Z\Vert^2$.
  Theorem~\ref{the:conc_forme_quadrat} is a particular case of Proposition~\ref{pro:conc_puiss_var_al_conc_autour} that gives the concentration of the $r$-power of a random variable $\left\Vert u(Z)\right\Vert$ when $u$ is quasiconvex and $1$-Lipschitz and $r\geq 1$~:
  \begin{align}\label{eq:conc_vvu(Z)vv}   
    \left\Vert u(Z)\right\Vert^r \in \mathbb{E}[\left\Vert u(Z)\right\Vert]^r \pm  Ce^{- \left({\,  \cdot \,}/{2^{r}\sigma \left\Vert u\right\Vert \mathbb{E}[\left\Vert u(Z)\right\Vert]^{r-1}}\right)^q} + Ce^{- \left(\frac{ \,  \cdot \,}{2 \left\Vert u\right\Vert^r  \sigma^r} \right)^\frac{q}{r}}.
  \end{align} 
  The function $z \mapsto \Vert A^{\frac{1}{2}}z\Vert$ is $\Vert A\Vert^{\frac{1}{2}}$-Lipschitz and convex. Therefore~:
  $$\Vert A^{\frac{1}{2}}Z\Vert \in \mathbb{E}[\Vert A^{1/2}Z\Vert] \pm C e^{-  (\, \cdot\,/\sigma \Vert A\Vert^{\frac{1}{2}}\mathbb{E}\Vert A^{1/2}Z\Vert)^q}.$$
  Since $\Vert A^{\frac{1}{2}}Z\Vert \leq  \Vert A^{\frac{1}{2}}\Vert \Vert Z\Vert$, $\mathbb{E}[\Vert A^{1/2}Z\Vert]\leq \Vert A^{\frac{1}{2}}\Vert \mathbb{E}[\Vert Z\Vert]$ and we can then conclude thanks to \eqref{eq:conc_vvu(Z)vv}.
   Now if we consider a general matrix $A \in \mathcal M_p$, let us decompose $A=A_+ - A_-+ A_0$ where $A_+$ is nonnegative symmetric, $A_-$ is non positive symmetric and $A_0$ is antisymmetric. We have clearly $Z^TAZ= Z^TA_+Z -Z^TA_-Z$ and we can conclude thanks to Lemma~\ref{lem:conc_somme}.
\end{proof}

\begin{remark}\label{rem:Hanson-Write_2}
  
  If we compare again this result to the original Hanson-Wright inequality, we can note that the tail parameters of the $1$-exponential component is the same.  The tail parameter of the $2$-exponential component of the classical result is proportional to $\left\Vert A\right\Vert_F$ while in our result, it is proportional (when $\sigma =O(1)$) to $\left\Vert A\right\Vert\mathbb{E}[\Vert Z\Vert]\sim \sqrt{p}\left\Vert A\right\Vert$. Therefore, considering that in most cases, i.e., when $A$ has a high rank and eigenvalues mainly of the same order, $\left\Vert A\right\Vert_F \sim \sqrt{p}\left\Vert A\right\Vert$, we see that the result of Theorem~\ref{the:conc_forme_quadrat} is quite similar to the Hanson-Wright inequality if we do not take into account the hypotheses which are quite different (on the one hand they are stronger because they only require the whole vector $Z$ to be concentrated, on the other hand they are weaker since they do not exploit the independence between the entries).

  It is possible to get a better concentration inequality if we combine Theorem~\ref{the:conc_forme_quadrat} with Theorem~\ref{the:conc_forme_quadrat_hypo_conc_lipsch}, but the expression of the concentration would be very complex (with possibly $4$ concentration rates), we thus prefer not to give it.
\end{remark}

\subsubsection{Transversal Convex concentration}
Although Definition~\ref{def:conc_lipschitzienne_vect_al} is perfectly adapted to the study of the resolvent $Q$ ($=Q_S$), the problem is far less immediate in  the setting of Definition~\ref{def:conentration_faible}, i.e., when $X \propto_c \alpha$.
Indeed in the setting of convex concentration, there does not exist any analogue to Proposition~\ref{pro:concentration_resovante} since it does not seem clear whether a functional of $Q$, $f(Q)$, with $f$ $1$-Lipschitz and quasiconvex can be written $g(X/\sqrt{n})$ with $g$ verifying the same properties.
In this case, the function $\phi$ cannot transfer the quasiconvexity and the study must then be conducted downstream directly with the random variables such as $\tr Q$.

It is interesting to note that $\tr Q=\tr (XX^T/n + zI_p)^{-1}$ stays unmodified if we multiply $X$ on the left or on the right by any orthogonal matrix. More formally, let us introduce the group $\mathcal{O}_{p,n}=\mathcal{O}_p \times \mathcal{O}_n$ where $\mathcal{O}_m$, $m\in \mathbb{N}$, is the orthogonal group of matrices of $\mathcal{M}_{m}$ ; it acts on $\mathcal{M}_{p,n}$ following the formula~:
\begin{align*}
  \text{for} \ \ (U,V) \in \mathcal{O}_{p,n}, \ M \in \mathcal{M}_{p,n} \ : \ \ (U,V) \cdot M = U M V^T.
\end{align*}
The function $f:R \in \mathcal M_{p,n} \mapsto \tr \phi(R)$ is $\mathcal O_{p,n}$-invariant, in the sense that $\forall (U,V) \in \mathcal O_{p,n}$, $\forall R \in \mathcal M_{p,n}$, $f((U,V) \cdot R) = f(R)$.
A result originally owed to Chandler Davis in \cite{DAV57}, and that can be found in a more general setting in \cite{GRA05}, gives us the hint that such an invariance can help us showing that $f$ is quasiconvex. 

To present this theorem, we note $\mathcal{D}_{p,n}^+$ the set of nonnegative diagonal matrices of $\mathcal{M}_{p,n}$~:
$$\mathcal{D}_{p,n}^+=\left\{(M_{i,j})_{\genfrac{}{}{0pt}{2}{1\leq i \leq p}{1\leq j \leq n}} \in \mathcal{M}_{p,n}\ | \ i\neq j \Leftrightarrow M_{i,j} =0 \  \text{and}\  \forall i\in\{1,\ldots, d\}  : M_{i,i}\geq 0\right\},$$where $d=\min(p,n)$.
\begin{theorem}\label{the:convex_matrice_diag}
  If a $\mathcal{O}_{p,n}$-invariant function $f : \mathcal{M}_{p,n}\rightarrow \mathbb{R}$ is quasiconvex on $\mathcal{D}_{p,n}^+$, then it is quasiconvex on the whole set $\mathcal{M}_{p,n}$.
\end{theorem}
The original Davis' theorem is only concerned with symmetric matrices and we could not find any proof of this theorem for the case of rectangle matrices although it is not so different (it was also aiming at proving the \textit{convexity} of $f$ and not its \textit{quasiconvexity} -- but it is actually simpler to treat). We provide in Appendix~\ref{app:davis} a rigorous proof of Theorem~\ref{the:convex_matrice_diag} with the help of some results borrowed from \cite{Bha97}. 

Let us define the subgroup of row permutations $\mathcal P_p=\{U\in \mathcal O_p \ | \ U_{i,j}\in \{0,1\}, 1\leq i,j\leq p\}$ and the subgroup of full permutations~:
\begin{align*}
  \mathcal{P}_{p,n}&=\{(U,V)\in \mathcal{P}_{p} \times \mathcal{P}_{n} \ | UI_{p,n}V^T=I_{p,n}\}\subset \mathcal O_{p,n},
\end{align*}
where $I_{p,n}\in \mathcal{D}_{p,n}$ is a matrix full of ones on the diagonal. 
Given a matrix $A\in \mathcal M_{p,n}$, we note $\diag(A)=(A_{1,1}, \ldots, A_{d,d})$, the vector composed of its diagonal terms (recall that $q = \min(p,n)$). It is tempting to identify the set $\mathcal D_{p,n}^+$ with $\mathbb{R}^d_+$, and the actions of $\mathcal P_{p,n}$ on a diagonal matrix $A$ to the actions of the group of permutation $\mathfrak S_d$ on the vector $\diag(A)$, where we define the action of $\mathfrak S_d$ on $\mathbb{R}^d$ as~:
\begin{align*}
  \forall \tau \in \mathfrak S_d, \forall x \in \mathbb{R}^d \ \ : \ \ \ \tau \cdot x= \left(x_{\tau(1)}, \cdots x_{\tau(d)}\right)&.
\end{align*}
With these considerations in mind, we see that a direct interesting consequence of Theorem~\ref{the:convex_matrice_diag} is that there exists a link between the convex concentration of a matrix $X$ and the convex concentration of the vector of its singular values. Recall that the sequence of singular values verifies $\sigma_1(A) \geq \cdots \sigma_{d}(A) \geq 0$ and for $1\leq i \leq d$, the $i ^{\textit{th}}$ singular value of $M$ can be defined as~:
\begin{align}\label{eq:formule_singuliere}
  \sigma_ {i}(A)= \max_{\genfrac{}{}{0pt}{2}{F \subset \mathbb{R}^n}{{\rm dim} F \geq i}} \min_{\genfrac{}{}{0pt}{2}{x\in F}{\left\Vert   x\right\Vert=1} } \left\Vert  Mx\right\Vert \ \  = \min_{\genfrac{}{}{0pt}{2}{F \subset \mathbb{R}^n}{{\rm dim} F \geq n-i+1} } \max_{\genfrac{}{}{0pt}{2}{x\in F}{\left\Vert  x\right\Vert=1} } \left\Vert  Mx\right\Vert,
\end{align}
where the subsets $F$ of $\mathbb{  R}^d$ on which is computed the optimization are subspaces of $\mathbb{  R}^n$. We introduce the convenient function $\sigma$ mapping a matrix to the ordered sequence of its singular values~:
\begin{align*}
  \sigma \ : \
  \begin{aligned}[t]
    &\mathcal{M}_{p,n} &&\rightarrow && \mathbb{R}^d_+ \\
    &M&&\mapsto&&(\sigma_1(M),\ldots, \sigma_d(M)).
  \end{aligned}
\end{align*}

To formalize this transfer of concentration between a matrix $X$ and $\sigma(X)$ let us introduce a new notion of concentration in a vector space $E$~: the \textit{convex concentration transversally to the action of a group $G$} acting on $E$.
\begin{definition}\label{def:concentrationfaible_transversale}
  Given a normal vector space $E$, a group $G$ acting on $E$, a concentration function $\alpha$ and a random vector $Z\in E$, we say that $Z$ is convexly $\alpha$-concentrated transversally to the action of $G$ and we note $Z\propto^T_G \alpha$ iff for any $1$-Lipschitz, quasiconvex and \emph{$G$-invariant} function $f : E \rightarrow \mathbb{R}$, $f(Z)$ is $\alpha$ concentrated.
\end{definition}

\begin{remark}\label{rem:chaine_implication_type_conc}
  In the setting of Definition~\ref{def:concentrationfaible_transversale}, we have the  induction chain~:$$Z \propto \alpha \ \  \Longrightarrow \ \ Z \propto_c \alpha \ \  \Longrightarrow \ \ Z \propto^T_G \alpha.$$
\end{remark}

Definition~\ref{def:concentrationfaible_transversale} is perfectly adapted to the next theorem whose proof can also be found in Appendix~\ref{app:davis} (it is a similar result to \cite[Corollary 8.21]{Led01} that concerns the eigenvalues of a random symmetric matrix).
\begin{theorem}\label{the:conc_val_sing}
  Given a normal vector space $E$, a concentration function $\alpha$ and a random matrix $X\in \mathcal{M}_{p,n}$, we have the equivalence~:$$\hspace{1cm}X \propto^T_{\mathcal{O}_{p,n}} \alpha \ \Longleftrightarrow \ \sigma(X) \propto^T_{\mathfrak S_d} \alpha,$$
  where $d = \min(p,n)$.
\end{theorem}
Theorem~\ref{the:conc_val_sing} can be rather powerful to set the concentration of symmetric functionals of random singular values. For our present interest, we will prove the concentration of the Stieltjes transform $m_{F}=\frac{1}{p} \tr Q$ of the covariance matrix $S=XX^T/n$ when $X\in \mathcal M_{p,n}$ is convexly concentrated (recall that $F$ stands for the normalized counting measure of the eigenvalues of $S$).
\begin{proposition}\label{pro:conc_faible_tr(Q)}
  Given a random matrix $X\in \mathcal M_{p,n}$, a concentration function $\alpha$ and $z>0$~:$$X \propto_c \alpha \ \ \Longrightarrow \ \ \tr Q \propto 2 \alpha \left(\frac{\sqrt{nz^3}}{8d} \, \cdot \,\right) \ \ \ \text{with} \ \ d=\min(p,n).$$
\end{proposition}
Let us introduce a simple preliminary lemma without proof :
\begin{lemma}\label{lem:tr_convex}
  Given an integer $d\in \mathbb{N}$, if a function $f: \mathbb{R} \rightarrow \mathbb{R}$ is convex, then the function $F : \mathbb{R}^d \rightarrow \mathbb{R}$ defined as $F(x_1,\ldots, x_d)=\sum_{i=1}^d f(x_i)$ is also convex. 
\end{lemma}
\begin{proof}[Proof of Proposition~\ref{pro:conc_faible_tr(Q)}]
  Assuming $X \propto_c \alpha$, we know from Remark~\ref{rem:chaine_implication_type_conc} and Theorem~\ref{the:conc_val_sing} that $\sigma(X) \propto_{\mathfrak S_d}^T \alpha$, where $d=\min(p,n)$. Let us introduce the function~:
  \begin{align*}
    f : s \longmapsto \frac{1}{s^2+z}.
  \end{align*}
  This function verifies $\sigma_i(Q) = f(\sigma_i(X)/\sqrt{n})$ and  $\tr Q = \sum_{i=1}^d f(\sigma_i(X)/\sqrt{n})$. It is also $\frac{2}{z^{3/2}}$-Lipschitz, for $s>0$~:
  \begin{align*}
  \left\vert f'(s)\right\vert= \frac{2s}{(s^2 +z)^2}, \text{ thus } \ f'(s) \leq \frac{2}{ z^{3/2}}.
  \end{align*}
  Therefore, the random variable $\tr Q$ is a Lipschitz and $\mathfrak S_d$-invariant transformation of $\sigma(X)$. However, it is not quasiconvex. The result of convexity can be obtained decomposing $f=g-h$ with $g$ and $h$ both convex, then we will be able to conclude thanks to Lemma~\ref{lem:conc_somme} giving the concentration of a sum of random variables.
  Let us set $h(s)=(\frac{s}{z}-\frac{1}{\sqrt{z}})^2$ if $s\in [0,\sqrt{z}]$ and $h(s)=0$ if $s\geq \sqrt{z}$, and $g=f+h$. We have~:
  \begin{align*}
    &\text{if $s\in [0,\sqrt{z}]$ :   }& &g'{}'(s)=\frac{ 6 s^2-2z}{(s^2+z)^3}+\frac{2}{z^2} \geq 0 & &\text{and} & & h'{}'(s)=\frac{2}{z^2}\geq 0 \\
    &\text{if $s\geq \sqrt{z}$ :   }& &g'{}'(s)=\frac{ 6s^2-2z}{(s^2+z)^3} \geq 0 & &\text{and} & & h'{}'(s)= 0.
  \end{align*}
  Besides, $h$ is $\frac{2}{z^{3/2}}$-Lipschitz, therefore $g$ is $\frac{4}{z^{3/2}}$-Lipschitz. We next introduce as in Lemma~\ref{lem:tr_convex} the functions $G,H : \mathbb{R}^d \rightarrow \mathbb{R}$ defined as~:
  \begin{align*}
    G(s_1, \cdots s_q)= \sum_{i=1}^d g(s_i)&
    &H(s_1, \cdots s_q)= \sum_{i=1}^d h(s_i).
  \end{align*}
  The functions $G$ and $H$ are both $\mathfrak S_d$-invariant and convex from Lemma~\ref{lem:tr_convex}. Besides, $G$ is $\frac{4d}{z^{3/2}}$-Lipschitz and $H$ is $\frac{2d}{z^{3/2}}$-Lipschitz. Therefore, we know from Lemma~\ref{lem:conc_somme} that $$\tr Q= G(\sigma(X)/\sqrt n)-H(\sigma(X)/\sqrt n) \propto 2\alpha \left(\frac{z^\frac{3}{2}\sqrt{n}\, \cdot \,}{8d}\right).$$
\end{proof}
In the case of a convex concentration of $X$, an analogous to Proposition~\ref{pro:concentration_resovante} setting the convex concentration of $Q$ does not seem obvious, even with the help of Theorem~\ref{the:convex_matrice_diag}. 


\subsubsection{From a convex concentration to a linear concentration}
It seems impossible to construct a quasiconvex \textit{and Lipschitz} shift of $\tr A\phi$ supported on the whole vector space $\mathcal M_{p,n}$ as we did in Proposition~\ref{pro:conc_faible_tr(Q)}. On a bounded subset of $\mathcal M_{p,n}$, it is always possible to shift a Lipschitz function with a convex one so that the sum verifies quasiconvex and Lipschitz properties. If we place ourselves in a convex $q$-exponential concentration setting, Proposition~\ref{pro:concentration_thresholdee} helps us treat first the concentration of $\tr AQ$ for a bounded version of a random vector $X$ and then we can take advantage of the contracting behavior of the resolvent (see Lemma~\ref{lem:controle_Q}) to generalize our first result to any unbounded concentrated vector $X$.

\begin{proposition}\label{pro:conc_faible_trAQ}
  Given a random matrix $X\in \mathcal M_{p,n}$, $z>0$, an exponent $q>0$ and two parameters $C\geq e$ and $\sigma>0$~:
  \begin{align*}
    X \overset{\mathbb E}{\propto_c}   \pm C e^{-(\, \cdot \, / \sigma)^q}
    \ \ \ \ \Longrightarrow \ \ \ \  
    Q \in \mathbb{E}Q \pm 2C e^{- \left( \frac{ \sqrt{z^3n}\, \cdot }{4\sigma}\right)^q } \ \ \text{in} \ (\mathcal M_{p,n}, \left\Vert \cdot\right\Vert_F).
  \end{align*}
\end{proposition}
Before proving the result we formulate a preliminary lemma~:
\begin{lemma}\label{lem:borne_trace_norme_spec}
  Given a symmetric nonnegative definite matrix $A \in \mathcal{M}_{p}$ and a matrix $B \in \mathcal M_p$, one has the bound :
  \begin{align*}
    &\vert \tr AB \vert \leq \left\Vert B\right\Vert\tr A .
  \end{align*}
\end{lemma}

\begin{proof}
  There exits a diagonal matrix $\Lambda = \diag(\lambda_i)_{1\leq i\leq p}$ and an orthogonal matrix $U\in \mathcal O_p$ such that $U^TAU=\Lambda$. If we note $u_i$ the $i ^{\textit{th}}$ column of $U$, we can write $A=\sum_{i=1}^p \lambda_i u_iu_i^T$, and for any $i\in \{1,\ldots  p\}$, $\left\Vert u_i\right\Vert=1$, so that~:
  \begin{align*}
    \tr AB = \sum_{i=1}^p \lambda_i \tr(Bu_i u_i^T)= \sum_{i=1}^p \lambda_i u_i^TBu_i \leq \sum_{i=1}^p \lambda_i \left\Vert B\right\Vert \leq \left\Vert B\right\Vert \tr A.
  \end{align*}
  We show the other bound the same way ($\Vert u \Vert \leq 1 \Rightarrow u^T B u \geq -\Vert B \Vert \Vert u \Vert$).
\end{proof}

\begin{proof}[Proof of Proposition~\ref{pro:conc_faible_trAQ}]
  With the function $\phi$ introduced in the proof of Proposition~\ref{pro:concentration_resovante} and given $A\in \mathcal M_p$ verifying $\left\Vert A\right\Vert_F \leq 1$, let us note~:
  \begin{align*}
    &f \ : \ R \ \longmapsto \ \tr A \phi(R).
  \end{align*}
  Given $R,H \in \mathcal M_{p,n}$, let us differentiate~:
  \begin{align*}
    &\nabla \restriction{f}{R} = - \phi(R)A\phi(R) R - R^T\phi(R)A\phi(R) \\
    &\nabla^2 \restriction{f}{R}(H,H) = 2\tr\left(A\phi(R) L \phi(R) L\phi(R)\right) - 2\tr(A\phi(R)HH^T\phi(R)),
  \end{align*}
  with the notation $L=RH^T + HR^T$. Let us suppose first that $A$ is nonnegative symmetric. In that case, we know from Lemma~\ref{lem:borne_trace_norme_spec} ($\left\Vert A\right\Vert\leq \left\Vert A\right\Vert_F\leq 1$) that~:
  \begin{align*}
    \tr(A\phi(R)HH^T\phi(R)) \leq \frac{ 2 }{z^2} \tr HH^T,
  \end{align*}
  and we recognize here the Hessian of the function $g:R \mapsto \frac{1}{z^2}\tr RR^T$ taken in $(H,H)$. If we note $h=f+g$ we know that $$\nabla^2 \restriction{h}{R}(H,H) \geq 2\tr\left(A\phi(R) L \phi(R) L\phi(R)\right)\geq 0.$$ 

  Besides, on the set $\{R \in \mathcal M_{p,n}, \left\Vert R\right\Vert\leq \sqrt{z}\}$, the function $g$ is $\frac{2}{z^{3/2}}$-Lipschitz and convex. If we suppose first that $\left\Vert  X\right\Vert\leq \sqrt{n}$, we know from Lemma~\ref{lem:conc_somme} that the sum $f(X/\sqrt{n})= (h - g)(X/\sqrt{ n})$ is concentrated~:
  \begin{align}\label{eq:conc_trAQ_X_borne}
    \tr A Q \in \tr A\mathbb E[Q] \pm 2Ce^{-(\frac{z^{3/2}\sqrt{n}}{4\sigma} \, \cdot \,)^q}.
  \end{align}

  Now, if we suppose that there exists a constant $K\geq 1$ such that $\left\Vert X\right\Vert\leq z^{\frac{3}{2}} \sqrt{Kn}$, then we have thanks to \eqref{eq:conc_trAQ_X_borne} and Lemma~\ref{lem:conc_transfo_lipschitz}~:
  \begin{align*}
    \tr A Q = \frac{1}{ K} \tr A \left( \frac{XX^T}{Kn} + \frac{z}{ K}I_p\right)^{-1}.
  \end{align*}
  And since $X/\sqrt{Kn}\overset{\mathbb E}{\propto_c}   \pm C e^{-(\sqrt{Kn}\, \cdot \, / \sigma )^q}$ the concentration \eqref{eq:conc_trAQ_X_borne}) and Lemma~\ref{lem:conc_transfo_lipschitz} entail~:
  \begin{align*}
    \tr A Q \in \tr A\mathbb E[Q] \pm 2Ce^{-  (\frac{\sqrt{nz^{3}}}{4\sigma} \, \cdot \,)^q}.
  \end{align*}
  In the general case, we consider the sequences of random matrices $X^{(m)}$ whose entries $X_{i,j}^{(m)}$, $1\leq i \leq p$, $1\leq j \leq n$, are defined as~:
  \begin{align*}
    X_{i,j}^{(m)}=\min\left(1,\frac{\sqrt{nz^3}m}{\left\Vert X\right\Vert_F}\right)X_{i,j}.
  \end{align*}
  By construction, $\left\Vert X^{(m)}\right\Vert_F \leq  \sqrt{nz^{3}}m$. Besides we know from Remark~\ref{rem:borne_esp_Z} that $\mathbb{E}\left\Vert X\right\Vert_F <\infty$ and for $m$ sufficiently large such that $2\mathbb{E}\left\Vert X\right\Vert_F \leq  \sqrt{nz^{3}}m$, we know from Proposition~\ref{pro:concentration_thresholdee} that $X^{(m)} \overset{\mathbb{E}}{\propto} C e^{-(\, \cdot \, / \sigma )^q}$. Thus, as above~:
  $$\tr AQ^{(m)} \in \tr A\mathbb E[Q^{(m)}] \pm 2Ce^{- (\frac{\sqrt{nz^{3}}}{4\sigma} \, \cdot \,)^q}, \ \ \ \text{ with :   } \ \ Q^{(m)}=\phi\left(  \frac{  X^{(m)}}{\sqrt n}\right).$$ If we let $m$ tend to $\infty$, $X^{(m)}$ tends in law to $X$ and thus $\tr A Q^{(m)}$ tends also in law to $\tr AQ$ and we recover the result of the proposition thanks to Proposition~\ref{pro:converg_loi} and Corollary~\ref{cor:moyenne_pivot}.
\end{proof}

We set in Subsections~\ref{sse:conc_vect_al} the comfortable environment where our subsequent results will easily unfold. We notably introduced most of the expressions of concentrated random variables needed in the subsequent section. We are thus now in position to provide results on the spectral distribution of sample covariance matrices of concentrated random vectors.

\section{Spectral distribution of the sample covariance}\label{sec:emp_cov}

\subsection{Setup and notations}\label{sse:set_up}

Reconsidering the example of the preamble of a data matrix $X = (x_1,\ldots,x_n) \in \mathcal M_{p,n}$, let us now list the different hypotheses that are needed to set the concentration and the estimation of the Stieltjes transform $m_F(z) = \frac1p \tr((S- z I_p)^{-1}$ (that will then give us the spectral distribution of the sample covariance matrix $S = \frac1n XX^T$). 
\begin{assump}\label{ass:independance_x_i}
  The random vectors $x_1,\ldots,x_n$ are independent.
\end{assump}
Note that we do not suppose that they are identically distributed. We will then note for every $i\in \{1,\ldots,n\}$:
\begin{align*}
  \Sigma_i = \mathbb E[x_ix_i^T],
\end{align*}
that we will call abusively the \textit{population covariance} of $x_i$, although it is not centered; it appears as the only relevant statistic of $X$ when one wants to estimate the spectral distribution of $\frac{1}{n}XX^T$.

All results presented below (with the exception of Propositions~\ref{pro:controle_delta-tilde_delta}, Theorem~\ref{the:deuxiemme_equi_det} and Corollary~\ref{cor:conc_steilt_trans_autour_tilde_Q_delta}) are valid for any choice of $p$ and $n$, but they will of course gain more value when $p$ and $n$ are sufficiently large for the convergence to arise. This is the reason why we will call this study \textit{quasi-asymptotic}. What we call a \textit{constant} is supposed to be independent of the two quasi asymptotic quantities $p$ and $n$ and a \textit{bounded} quantity is simply a quantity lower than a constant, we will note for instance $C = O(1)$. Given two sequences of reals $u_n$ and $a_n$, we will use the notations $u_n = O(a_{n})$ to specify that there exists a constant $C >0$ such that $|u_n | \leq C a_n$, for all $n,p \in \mathbb N$.

We place ourselves under an hypothesis of convex $q$-exponential concentration~:
\begin{assump}[Concentration of $X$]\label{ass:concentration}
  There exist three constants $C\geq e$, $c>0$ and $q>0$ such that:
  \begin{align*}
    X  \propto_c C e^{-\, \cdot \, ^q/c}.
  \end{align*}
\end{assump}
The parameters of the concentration $C$ and $c$ cannot be preserved throughout the different concentrations of the quantities that will be mentioned in this paper. However to lighten the expression of the result we will abusively keep the notations $C$ and $c$ to designate slight modifications of the original $C$ and $c$ by numerical constants.
{
\begin{remark}
   It is at this point important to recall that, in practice, we wish the above concentration to hold for realistic datasets $(x_1,\ldots, x_n)$ where the $x_i$ may belong different distribution classes, thereby entailing a possibly multi-modal distribution for the $x_i$'s, which has no reason to be concentrated (in particular, if the distance between the statistical means of the classes is larger that $O(\sqrt p)$, we can be sure concentration does not hold). Yet, by considering ``unsorted'' data vectors $x_1,\ldots,x_n$, even possibly with unfixed class sizes, one may still assume the weaker assumption that $X \propto^T_{\mathfrak S_n} C e^{- \cdot\,^q /c}$ where the action of $\mathfrak S_n$ is a permutation of the columns. Indeed, as all the functionals of $X$ considered in this section are actually functionals of $XX^T/n$ (thus invariant permutation of the columns of $X$), all the upcoming results remain valid. 

  Note besides that it is not sufficient to suppose that each $x_i$ is concentrated, since we do not have any satisfactory result allowing us to generalize the concentration to the whole matrix $X$ with an interesting observable diameter -- without further assumptions, the best one can hope is an observable diameter of size $O(\sqrt n)$ in the Lipschitz concentration case (and for the Frobenius norm on $\mathcal M_{p,n}$) \cite{LED05}; Proposition~\ref{pro:concentration_concatenation_vecteurs_independants_lipsch_et_convexe} in the Appendix~\ref{app:side_result_concentration} gives a weaker result in the case of convex concentration.
\end{remark}}
\begin{remark}\label{rem:hypothèse_concentration_avec_theoremes}
  Assumption~\ref{ass:concentration} is in particular verified for the distributions concerned by Theorems~\ref{the:concentration_vecteur_spherique},~\ref{the:concentration_vecteur_gaaussien} and~\ref{the:talagrand}. For $i\in\{1,\ldots, n\}$, if we note $\mu_i$, the law of $x_i$, it  can respect one of the two settings~:
  \begin{itemize}
    \item \text{Setting $1$ :} $\mu_i$ is a pushforward of the canonical normal distribution on $\mathbb{R}^d$, or of the sphere $\mathbb S^{d+1}$ through the mapping of a uniformly continuous function (see Proposition~\ref{pro:conc_fonctionnelles_holderienne}).
    \item \text{Setting $2$ :} $\mu_i$ is an affine pushforward of a product of distributions on $\mathbb{R}$ with support belonging to $\left[-1,1\right]$ (see Lemma~\ref{lem:transformations_faiblement_convexes}).
  \end{itemize}
  Note that $\mu_i$ can also be the sum of two independent distributions, each one coming from a different setting.
  In the literature, it is possible to find plenty of other cases of $q$-exponential concentration with $q \neq 2$ (and still with Euclidean norm!); their presentation goes far beyond the objectives of our paper and we invite the reader to refer to \cite{Led01} for more information. We will see at the end of this paper that our results are still valid on different types of practical data. Thus we come to think that convex $q$-exponential concentration is a general hypothesis that can be adopted in a large range of applications where the data have a satisfactory entropy compared to the diameter of their distribution (the link with entropy is explored once again in \cite{Led01}).
\end{remark}
Assumption~\ref{ass:concentration} gives us access to the results of the propositions and remarks of Subsection~\ref{sse:conc_vect_al} concerning the matrix $X$ but also those concerning the random vectors $x_i$ and the resolvent $Q$. To begin with, we give an immediate bounding result on the true convariance of $X$. Let us note the mean of $x_i$:
\begin{align*}
  m_i = \mathbb E[x_i] \in \mathbb R^p.
\end{align*}
\begin{proposition}\label{pro:covariance_bornee}
  $\forall i \in \{1, \ldots,n\}, \| \Sigma_i - m_i m_i^T \| = O(1)$.
\end{proposition}
\begin{proof}
  Let us consider $u\in \mathbb R^p$ such that $\|u\| \leq 1$. We know that $u^Tx_i \in u^Tm_i C e^{- \, \cdot \,^q/c}$ thus we know from Lemma~\ref{lem:conc_autour_prod} that:
  \begin{align*}
    (u^Tx_i)^2 \in (u^T m_i)^2 \pm C e^{- \, \cdot \,^q/c'} + C e^{- \, \cdot \,^{q/2}/c'{}'},
  \end{align*}
  for some constants $c',c'{}'>0$ depending numerically on $c$. We can then conclude thanks to a closely similar result to Proposition~\ref{pro:carcterisation_ac_moments_conc_q_expo} that:
  \begin{align*}
    \| \Sigma_i - m_i m_i^T \| = \sup_{\|u\|\leq 1} \left\vert\mathbb E \left[ (u^Tx_i)^2 - (u^T m_i)^2\right]\right\vert = O(1).
  \end{align*}
\end{proof}
\color{black}
We can rewrite the first result of Proposition~\ref{pro:conc_faible_trAQ} to get the concentration of the Stieltjes transform of $F$, the spectral distribution of $S$~:
\begin{align*}
  m_{F}(z)= \frac{1}{p}\tr Q(-z)&
  &\text{where}&
  &Q=Q(z)= \left(XX^T/n + z I_p\right)^{-1}.
\end{align*}
We saw in Proposition~\ref{pro:conc_faible_trAQ} that $Q \in \mathbb{E}Q \pm C e^{-(\sqrt{zn} \, \cdot \, )^q/c}$ in $(\mathcal M_{p,n}, \left\Vert \cdot\right\Vert_F)$, for some constants $C\geq e$, $c>0$. Since the linear form $M \in (\mathcal M_{p,n}, \left\Vert \cdot\right\Vert_F) \mapsto \frac{1}{ p}\tr M$ has an operator norm equal to $\frac{1}{p}\left\Vert I_p\right\Vert_F = 1/\sqrt{p}$, the linear concentration of $Q$ implies directly the concentration of the Stieltjes transform with a tail parameter of order $\sqrt{z^3np}$.
\begin{proposition}\label{pro:conc_Steiltjes}
  There exist two numerical constants $C\geq e$ and $c>0$ such that for every $z>0$, $m_{F}(-z) \in \frac{1}{p}\tr \mathbb{E}Q \pm C e^{- (\sqrt{z^3pn} \, \cdot \,)^q/c}$.
\end{proposition}

To complete the lacuna of the concentration notion that does not give any information on the order of a random vector $X$, as concentrated it could be, one needs to give restrictions to the size of the quantities $\left\Vert m_i\right\Vert$, $1\leq i \leq n$. It appears that the means $m_i$ can be of great amplitude ($\|m_i\| = O(\sqrt p)$), as long as their main direction is the same. We will thus decompose each $m_i$ followingly:
\begin{align*}
  m_i = \mathring m_i + s,
\end{align*}
and we place ourselves under the assumption:


\begin{assump}\label{ass:cosmetique}
  $\forall i \in \{1, \ldots, n\} \ : \ \mathbb{E}\left\Vert x_i \right\Vert = O(\sqrt p)$ and $\|\mathring m_i\| = O(1)$\footnote{
  One can also assume that there exit a finite number of signals $s_1,\ldots, s_k$ with $k = O(1)$ such that $\forall i \in \{1,\ldots, n\}$, $\exists a  \in \{1,\ldots, k\}$ satisfying $\mathbb E[x_i] = \mathring m_i + s_a$ (and that all the $s_a$ are sufficiently reprensented).}.
\end{assump}
However, thanks to Corollary~\ref{cor:tao}, it is possible to make a simpler hypothesis concerning directly the deterministic vectors $\mathring m_i$ and $m$ if we suppose that $q\geq 2$ (then $p^{1/q}\leq \sqrt{p}$).
\begin{assbis}
\   $q\geq 2$, $\left\Vert s\right\Vert = O(\sqrt p)$ and $\sup_{1\leq i \leq n} \left\Vert \mathring m_i\right\Vert = O(1)$.
\end{assbis}
Assumption~\ref{ass:cosmetique} (or $3$ bis) allows us to control $\tr \Sigma_i$.
\begin{proposition}\label{pro:borne_tr_Sigma_l}
  For any $l\in\{1,\ldots, n\}$, $\tr \Sigma_i=O(p)$.
\end{proposition}
\begin{proof}
  Since $\left\Vert x_i\right\Vert \in C e^{-\, \cdot \,^q/c}$, we know from Proposition~\ref{pro:carcterisation_ac_moments_conc_q_expo} that~:
  \begin{align*}
      \tr \Sigma_i = \mathbb{E}\left[\left\Vert x_i\right\Vert^2\right] = \left(\mathbb{E}\left\Vert x_i\right\Vert\right)^2 + \mathbb{E}\left(\left\Vert x_i\right\Vert -\mathbb{E}\left\Vert x_i\right\Vert\right)^2 \leq C p,
  \end{align*}  
  for some constant $C>0$.
\end{proof}

 \subsection{Estimation of the spectral distribution of the sample covariance}\label{sse:est_dist_spect_cov_empi}

We know from Proposition~\ref{pro:conc_Steiltjes} that the Stieltjes transform is concentrated in $\mathbb{R}$, we expect now to find a deterministic quantity localizing this concentration. The deterministic equivalent we presented in the preamble can inspire us for a choice of a deterministic equivalent in this more general setting, assuming Assumptions~\ref{ass:independance_x_i},~\ref{ass:concentration} and~\ref{ass:cosmetique} (or $3$ bis). 
Let us note
\begin{align*}
  \tilde Q^{\delta}= Q_{\Sigma^{\delta}} =\left(\Sigma^{\delta} + z I_p\right)^{-1}&
  &\text{with}&
  &\Sigma^{\delta}=  \frac{1}{n}\sum_{i=1}^n \frac{\Sigma_i}{1+\delta_i}&
  &\text{and}&
  &z>0 
 \end{align*}
is a deterministic equivalent for the resolvent $Q(z)$ if $\delta=(\delta_1,\ldots, \delta_n)\in \mathbb{R}^n$ is chosen correctly.

Following the calculus of the preamble in this more general case, one gets~:
\begin{align*}
  &\tilde Q^{\delta} - \mathbb{E}Q =  \mathbb{E}\left[\Delta + \epsilon \right]\  \\
  & \text{with :} \ \ \ \left\{
  \begin{aligned}
    &\Delta = \frac{1}{n}\sum_{i=1}^n\left(\frac{1}{1+x_i^T Q_{-i} x_i/n} - \frac{1}{1+\delta_i}\right)Q_{-i}x_ix_i^T \tilde{Q}_{\delta} \\
    & \epsilon =  \frac{1}{n^2}\sum_{i=1}^n \frac{\mathbb{E}\left[Q_{-i}x_ix_i^TQ\Sigma_i\tilde{Q}_{\delta}\right]}{1+\delta_i},
  \end{aligned}
  \right. 
\end{align*}
where we took advantage of the independence between $Q_{-i}$ and $x_i$ to say that $\mathbb{E}[Q_{-i}\Sigma_i\tilde{Q}_{\delta}/(1+\delta_i)]=\mathbb{E}[Q_{-i}x_ix_i^T\tilde{Q}_{\delta}/(1+\delta_i)]$ (recall that $Q = (\frac{1}{n}XX^T - \frac{1}{n} x_i x_i^T  z I_p)^{-1}$. The two matrices $\Delta_i$ and $\epsilon_i$ will be used several times in our proof arguments; we thus invite the reader to remember their definition. The form of $\Delta_i$ entices us to set
$$\delta= \left(\frac{1}{n}\tr (\Sigma_i\mathbb{E}Q_{-i})\right)_{1\leq l \leq k}$$
to show that $\tilde Q^{\delta}$ is a deterministic equivalent for $Q$. We will show afterwards that the same result holds for $\tilde Q^{\delta'}$ with $\delta' \in \mathbb{R}^+$ chosen as a solution of the system (see Proposition~\ref{pro:definition_delta_recursive} for the validity of this definition)~:
$$\delta'_i=\frac{1}{n}\tr\left(\Sigma_i \left(\frac{1}{n}\sum_{\genfrac{}{}{0pt}{2}{j=1}{j\neq i}}^n \frac{\Sigma_j}{1+\delta'_j}+z I_p\right)^{-1}\right) \ \ \ 1\leq l \leq k.$$

The first choice $\delta$ can seem unsatisfactory because it relies on the computation of $\mathbb{E}Q_{-i}$ which is uneasy to treat. On the contrary, the second option $\delta' $ is much more interesting as it can be approximated by iteration of the fixed point equation as we will see in Proposition~\ref{pro:definition_delta_recursive}. In particular, this second choice reveals that the deterministic equivalent can be chosen in a way that it only depends on $z$ and on the matrices $\Sigma_i$, $1\leq i\leq n$ as will be fully explained in Remark~\ref{rem:loi_spectrale_distr_gaussi}.

\subsubsection{Design of a first deterministic equivalent}
Let us first show the concentration of the random variable $x_i^TQ_{-i}x_i/n$ around its mean $\delta_i$. To simplify the concentration bounds, we introduce the real $1>z_0>0$, and from now on, $z$ is supposed to be greater than $z_0$. To begin with, note that $\forall i \in \{1,\ldots,n\}$, $\delta_i$ is bounded.
\begin{lemma}\label{lem:controle_QC}
  For any $i\in \{1, \ldots, n\}$, $\delta_i = \frac{1}{n}\mathbb E [x_i^T Q_{-i} x_i]\leq \frac{\gamma}{z_0}$.
\end{lemma}
\begin{proof}
  We know from Lemma~\ref{lem:borne_trace_norme_spec} and Assumption~\ref{ass:cosmetique} that : 
  \begin{align}\label{eq:maj_delta'}
    \left\vert \delta_i\right\vert=\frac1n\left\vert \tr \mathbb{E}Q_{-i}\Sigma_i\right\vert  \leq \frac1n\mathbb{E}\left\Vert  Q_{-i}\right\Vert \tr \Sigma_i \leq \frac{\gamma }{z_0}.
  \end{align}
\end{proof}
\begin{proposition}\label{pro:conc_yqy}
  Given $z>z_0$, there exists two numerical constants $C,c>0$ such that~:
  \begin{align*}
    &x_i^TQ_{-i}x_i/n \in \delta_i \pm C e^{-(z_0n \, \cdot \,)^{\frac{q}{2}}/c} + Ce^{-(\sqrt{z_0^3n} \, \cdot \, /\bar{\gamma})^q/c},
  \end{align*}
  where $\gamma=\frac{p}{n}$, and $\bar \gamma= \gamma + 1 \geq \max(\gamma, 1)$. 
\end{proposition}

\begin{proof}
  This proposition looks like Theorem~\ref{the:conc_forme_quadrat} and Proposition~\ref{pro:conc_faible_trAQ} applied with independent random objects, respectively $x_i$ and $Q_{-i}$ that we know to be concentrated thanks to the initial concentration on $X$ given by Assumption~\ref{ass:concentration}. The independence allows us to consider $Q_{-i}$ as deterministic when we bound a probability involving $x_i$ and conversely.
  \begin{align*}
    &\mathbb{P}\left(\left\vert \frac{1}{n}x_i^T Q_{-i}x_i - \frac{1}{n} \tr (\Sigma_i\mathbb{E} Q_{-i})\right\vert \geq t\right)\\
    &\hspace{1cm}\leq \mathbb{E}\left[\mathbb{P}\left(\left\vert x_i^T Q_{-i}x_i - \tr \Sigma_i Q_{-i}\right\vert \geq \frac{nt}{2} \ | \ X_{-x_i}\right)\right] \\
    &\hspace{1.5cm}+\mathbb{P}\left(\left\vert \frac{1}{n}\tr \Sigma_i (Q_{-i} - \mathbb{E}Q_{-i})\right\vert \geq \frac{t}{2}\right) \\
    &\hspace{1cm}\leq  C e^{-(znt)^{\frac{q}{2}}/c} + Ce^{-(zn^{\frac{3}{2}} t/p)^q/c} + Ce^{-((zn)^\frac{3}{2} t/p)^q/c},
  \end{align*}
  for some $C,c>0$. We employed Theorem~\ref{the:conc_forme_quadrat} together with Lemma~\ref{lem:controle_Q} to control the variation of $x_i^TQ_{-i}x_i\ | \ X_{-x_i}$ ; and to control the variation of $\frac{1}{n}\tr \Sigma_i Q_{-i}$, we employed Proposition~\ref{pro:conc_faible_trAQ} thanks to the fact that $\left\Vert \Sigma_i\right\Vert_F\leq \tr \Sigma_i = \mathbb{E}\left\Vert x_i\right\Vert^2\leq p$.

\end{proof}
We are interested in bounding the first centered moments of $\frac{1}{n}x_i^T Q_{-i}x_i$, thus the exponent $m$ is considered to be a constant of the problem and that simplifies the expression of the bound. Assumption~\ref{ass:cosmetique} allows us to simplify the bound~:
\begin{proposition}\label{pro:carcterisation_ac_moments_conc_q_expo_yqy}
  
  Given $l \in \left\{1, \ldots, k\right\}$ and $r>0$, there exists a constant $C\geq e$ such that :
  \begin{align*}
    &\mathbb{E}\left[\left\vert \frac{x_i^TQ_{-i}x_i}{n} - \delta_i\right\vert^r\right] \leq C\left(\frac{\bar \gamma^2}{z_0^3 n}\right)^{r/2}.
  \end{align*}
\end{proposition}

\begin{proof}
  It is a simple consequence of Proposition~\ref{pro:carcterisation_ac_moments_conc_q_expo} applied to the concentration of $\frac{1}{n}x_i^TQ_{-i}x_i$ given by Proposition~\ref{pro:conc_yqy}~:
  \begin{align*}
    \mathbb{E}\left[\left\vert \frac{x_i^TQ_{-i}x_i}{n} - \frac{1}{n} \tr (\Sigma_i Q_{-i})\right\vert^r\right] \leq \left(\frac{C\bar \gamma^2}{z_0^3 n}\right)^{r/2} +\frac{C}{(z_0n)^r},
  \end{align*}
  and we recover the bound of the proposition thanks to~:
  \begin{align*}
    \left(\frac{C\bar \gamma^2}{z_0^3 n}\right)^{r/2} \geq \frac{C}{(z_0n)^r}  \hspace{1cm} \text{(recall that $z_0\leq 1$)}.
  \end{align*}
\end{proof}

We then just need a last lemma to be able to bound the spectral norm of $\Sigma_i \tilde Q$. It is actually for this purpose that we had to assume that the $\mathring m_i = \mathbb E[x_i]$ were bounded.
{\ajout \begin{lemma}\label{lem:borne_Ci_Q}
  $\sup_{1\leq i \leq n} \| \Sigma_i \tilde Q \| \leq \frac{C\bar \gamma}{ z_0}$.
\end{lemma}
\begin{proof}
  Let us note $\mathring \Sigma_i = \Sigma_i - (m_i+s)(m_i+s)^T\geq 0$ (as a symmetric matrix), $m = \frac{1}{n}\sum_{i=1}^n \frac{m_i}{\sqrt{1+\delta_i}}$ and $\nu = \frac{1}{n}\sum_{i=1}^n \frac{1}{\sqrt{1+\delta_i}}$, we can decompose:
  \begin{align*}
    \frac{1}{n}\sum_{i=1}^n\frac{\Sigma_i}{1+\delta_i}
    &= \frac{1}{n}\sum_{i=1}^n\frac{\mathring \Sigma_i}{1+\delta_i} + \frac{1}{n}\sum_{i=1}^n\frac{m_i m_i^T}{1+\delta_i} - mm^T + (m+\nu s)(m + \nu s)^T,
  \end{align*}
  where we note that $\frac{1}{n}\sum_{i=1}^n\frac{m_i m_i^T}{1+\delta_i} - mm^T \geq 0$ (as a covariance matrix). Therefore:
  \begin{align*}
    Q_{-s}^\delta \equiv \left( \frac{1}{n}\sum_{i=1}^n\frac{\mathring \Sigma_i}{1+\delta_i} + \frac{1}{n}\sum_{i=1}^n\frac{m_i m_i^T}{1+\delta_i} - mm^T + z I_p\right)^{-1} \geq 0.
  \end{align*}
  And we can employ the Schur formula to set:
  \begin{align*}
    (m+\nu s)^T\tilde Q^\delta (m+\nu s)  = \frac{ (m+\nu s)^T\tilde Q_{-s}^\delta (m+\nu s)}{1 + (m+\nu s)^T\tilde Q_{-s}^\delta (m+\nu s)} <1.
  \end{align*}
  Thus, since (i) $\frac{1}{\nu} = O(\sqrt{\bar \gamma})$ (thanks to Lemma~\ref{lem:controle_QC}), (ii) $\|\mathring m \| = O(1)$ (see Assumption~\ref{ass:cosmetique} or $3$ bis) and (iii) $\|\tilde Q^\delta \| \leq O (1/z_0)$ (see Lemma~\ref{lem:controle_Q}):
  \begin{align*}
    s^T\tilde Q^\delta s \leq \frac{2}{\nu^2}(m+\nu s)^T\tilde Q^\delta (m+\nu s)  + \frac{2}{\nu^2} m^T\tilde Q^\delta m  = O (\bar \gamma /z_0).
  \end{align*}
 The same way, $m_i^T\tilde Q^\delta s, m_i^T\tilde Q^\delta m_i = O(\bar \gamma /z_0)$, which allows us to conclude:
  \begin{align*}
      \| \Sigma_i \tilde Q\| \leq \| \mathring \Sigma_i \tilde Q\| + \| (m_i+s)^T \tilde Q(m_i+s)\| = O \left( \frac{\bar \gamma}{z_0}\right).
  \end{align*}
\end{proof} }

To show the concentration of the Stieltjes transform around $\frac{1}{p}\tr \tilde Q^\delta$ (we already know from Proposition~\ref{pro:conc_Steiltjes} that it concentrates around $\frac{1}{p}\tr \mathbb{E}Q$) but also to bound $\left\Vert \delta-\delta'\right\Vert$, it is important to control $\Vert \mathbb{E}Q - \tilde Q^{\delta}\Vert$.

\begin{proposition}\label{pro:borne_EQ_m_tQ}
    $\left\Vert \mathbb{E}Q - \tilde Q^{\delta} \right\Vert = O \left(a_{n,p}  \right)$, with $a_{n,p} \equiv \frac{\bar \gamma^{3/2}}{z_0^4}\sqrt{\frac{ \log(n)}{n}}$ (recall that $\gamma = \frac{p}{n}$ and $\bar \gamma = 1+\gamma$).
\end{proposition} 
{\ajout \begin{proof}
  It is sufficient to bound for any vectors $u,v\in \mathbb{R}^p$ of unit norm the quantity $\vert u^T (\mathbb{E}Q - \tilde Q^{\delta})v\vert$.
  Recall from the heuristic approach displayed at the beginning of this subsection the matrices $\Delta$ and $\epsilon$ verifying~:
  \begin{align*}
    u^T(\mathbb{E}Q - \tilde Q^{\delta})u = u^T \Delta v
    + u^T \epsilon v.
  \end{align*}
  Let us compute:
  \begin{align*}
    \left\vert u^T\epsilon v\right\vert
    &= \frac{1}{n}   \left\vert\sum_{i=1}^n\mathbb E \left[ \frac{ u^T(Q -Q_{-i} ) \Sigma_i\tilde Q^\delta  v }{1 + \delta_i}\right]\right\vert
      = \frac{1}{n^2}   \left\vert\sum_{i=1}^n\mathbb E \left[ \frac{ u^TQ x_ix_i^TQ_{-i}  \Sigma_i\tilde Q^\delta  v }{1 + \delta_i}\right]\right\vert \\
      & \leq \frac{1}{\sqrt n} \sqrt{   \frac{1}{n}\sum_{i=1}^n\mathbb E \left[ u^TQ x_ix_i^TQu D_i\right] \frac{1}{n}\sum_{i=1}^n\mathbb E \left[ v^T\tilde Q^\delta \Sigma_iQ_{-i}  \Sigma_i\tilde Q^\delta  v \right] },
    \end{align*}
    with $D_i = \frac{1}{n}x_i^TQ_{-i}x_i$. We thus know from Lemma~\ref{lem:borne_Ci_Q} that:
    \begin{align*}
       \left\vert u^T\epsilon v\right\vert \leq \frac{\bar \gamma}{\sqrt{nz_0^3}}  \sqrt{\frac{1}{n}\mathbb E[u^TQXDXQu]} \leq \frac{\sqrt{ \mathbb E[\| D \|]}\bar \gamma}{z_0^{2} \sqrt n}  .
     \end{align*} 
    We know from Lemma~\ref{pro:conc_yqy} that:
    \begin{align*}
       D_i \in \delta_i \pm Ce^{-(\sqrt{z_0^3n} \, \cdot \, /\bar{\gamma})^{q/2}/c},
     \end{align*} 
    we can then conclude with an analog to Corollary~\ref{cor:tao} that 
    $$\mathbb E[\|D\|] = \mathbb E \left[\sup_{i\in[n]}|D_i|\right] \leq \mathbb E[\|\delta\|_{\infty}] + O \left(\log(n)^{2/q}\sqrt{\frac{\bar \gamma}{nz_0^{3}}}\right) = O \left(\sqrt{\frac{ \bar \gamma}{z_0^{3}}}\right),$$
    and thus $u^T\epsilon v = O\left(\frac{\bar \gamma^{3/2}}{z_0^4\sqrt{n}}\right)$. 

    Besides, with Cauchy Schwartz inequality, we can bound:
\begin{align*}
      u^T\Delta v
      &= \frac{1}{n}  \left\vert \sum_{i=1}^n \mathbb E \left[    \frac{ u^T Q_{-i} x_i x_i^T\tilde Q^\delta v \left(\frac{1}{n} x_i^T Q_{-i} x_i - \delta_i \right)  }{\left(1 + \frac{1}{n} x_i^T Q_{-i} x_i\right) \left(1 + \delta_i\right)} \right] \right\vert\\
      &= \frac{1}{n}  \sqrt{ \mathbb E \left[  u^T Q X  E  X^T Qu\right] \sum_{i=1}^n\mathbb E\left[\frac{ v \tilde Q^\delta x_i x_i^T\tilde Q^\delta v }{1 + \delta_i} \right]},
    \end{align*}
    with $E =  \diag \left( \frac{1}{n} x_i^T Q_{-i} x_i - \delta_i \right)^2_{1\leq i \leq n}$. As previously, since 
    \begin{align*}
    E_i \in 0 \pm C e^{-(z_0^2n^2 \, \cdot \,)^{\frac{q}{4}}/c} + Ce^{-(z_0^3n \, \cdot \, /\bar{\gamma})^{q/2}/c},
  \end{align*}
  we can show that:
    \begin{align*}
      \mathbb E[\|E\|] \leq O \left( \frac{\log(n)^{4/q}}{z_0^2n^2}\right) +   \left(\log(n)^{2/q} \frac{\bar \gamma}{z_0^3 n}\right) = \left(\log(n)^{2/q} \frac{\bar \gamma}{z_0^3 n}\right).
    \end{align*}
    Therefore $u^T\Delta v = O(\frac{\bar \gamma\sqrt{\mathbb E[\|E\|]}}{z_0}) =  O(\frac{\bar \gamma^{3/2}}{z_0^2}\sqrt{\frac{ \log(n)}{n}})$.

  Putting the bounds on $u^T \Delta v$ and on $u^T \epsilon v$ together, we find the results of the proposition.
\end{proof}}
\color{black}
This last proposition together with Proposition~\ref{pro:conc_faible_trAQ} and Lemma~\ref{lem:comportement_pivot_equivalent_deterministe} implies that $\tilde Q^\delta$ is a deterministic equivalent for $Q$.
\begin{proposition}\label{pro:premier_equivalent_deterministe}
  $Q \in \tilde Q^\delta \pm Ce^{-(a_{n,p}\, \cdot \,)^q/c}$ in $(\mathcal  M_{p,n}, \left\Vert \cdot\right\Vert)$ (for some constant $C=O(1)$, $a_{n,p}$ is defined in Proposition~\ref{pro:borne_EQ_m_tQ}).
\end{proposition}
\begin{proof}
  We know from Proposition~\ref{pro:conc_faible_trAQ} that $Q \in \mathbb{E}Q \pm C e^{-(\sqrt{zn} \, \cdot \, )^q/c}$ in $(\mathcal M_{p,n}, \left\Vert \cdot\right\Vert_F)$.
  But since $\Vert \mathbb{E}Q - \tilde Q^\delta\Vert \leq O(a_{n,p})$, and $1/\sqrt{zn} = O(a_{n,p})$, Lemma~\ref{lem:comportement_pivot_equivalent_deterministe} entails the result of the proposition.
\end{proof}

  

  
\subsubsection{A second deterministic equivalent}

Let us introduce a substitute for $\delta$ that we note $\delta'$ and that will only depend on the means and covariances of the laws $\mu_l$ and on the cardinality coefficients $\frac{1}{n}$ of the different classes.
\begin{proposition}[Definition of $\delta'$]\label{pro:definition_delta_recursive}
  The system of equations~:
  \begin{align*}
    \forall i \in \{1,\ldots, n\} \ :&
    &\delta'_i=\frac{1}{n}\tr\left(\Sigma_i \left(\frac{1}{n}\sum_{\genfrac{}{}{0pt}{2}{j=1}{j \neq i}}^n \frac{\Sigma_j}{1+\delta'_j}+z I_p\right)^{-1}\right)   
  \end{align*}
  admits a unique solution in $\mathbb{R}_+^n$ that we note $\delta'$. 
\end{proposition}
\begin{proof}
  The scheme of the proof follows the formalism of standard interference functions as presented in \cite{YAT95}. Following the ideas of Yates, we introduce the function :
  \begin{align*}
    I \ \ : \ \ 
    \begin{aligned}[t]
      & \ \ \ \mathbb{R}_+^n&&\longrightarrow&& \mathbb{R}_+^n \\
      &(\delta'_i)_{1\leq l \leq k}&&\longmapsto&& \frac{1}{n}\tr \left(\Sigma_i \tilde Q_{-i}^{\delta'}\right),
    \end{aligned}
  \end{align*}
  where $\tilde Q^\nu_{-i}=(\Sigma^\nu_{-i}+zI_p)^{-1}$, and $\Sigma^\nu_{-i}= \frac{1}{n}\sum_{\genfrac{}{}{0pt}{2}{j = 1}{j \neq i}}^n \frac{\Sigma_j}{1+\nu_j}$ for $\nu=(\nu_1,\ldots, \nu_n)\in \mathbb{R}^n$. Given $\nu,\eta \in \mathbb{R}^n$, we note $\nu\leq \eta$ iff $\forall i \in \{1,\ldots, n\}$, $\nu_i\leq \eta_i$. The function $I$ is increasing in the sense that~:
  \begin{align*}
    \nu\leq \eta \ \ \ \Longrightarrow  \ \ \ I(\nu) \leq I(\eta).
  \end{align*}
  This is simply due to the fact that for any $i\in \{1, \ldots, n\}$, $\Sigma_i$ is symmetric nonnegative definite and $\nu\mapsto \tilde Q^\nu$ is increasing (with the classical order relation defined on the set of symmetric matrices see \cite[Section~V]{Bha97} for more details). Besides, we know from Lemmas~\ref{lem:controle_Q} and~\ref{lem:borne_trace_norme_spec} that for any $\nu\in \mathbb{R}^n$, $I(\nu) \leq \frac{\tr \Sigma_i}{nz}$. Therefore the vector 
  $$\nu_0=\left(\frac{\tr \Sigma_1}{nz}, \cdots ,\frac{\tr \Sigma_n}{nz}\right)$$
  verifies $I(\nu_0)\leq \nu_0$. Then the monotonicity of $I$ implies that the sequence $(I^n(\nu_0))_{n\geq 0}$ is decreasing and since $I$ takes its values in $\mathbb{R}_+^n$, it converges to a fixed point $\delta'\geq 0$. 

  Let us now show the uniqueness of this fixed point. Let us suppose that there exists $\nu \neq \delta'$ such that $I(\nu)=\nu$. Given $i\in \{1, \cdots , n\}$, we have the identity~: 
  \begin{align*}
     \nu_i - \delta'_i
    & = \frac{1}{n} \tr \left(\frac{1}{n}\sum_{\genfrac{}{}{0pt}{2}{j=1}{j\neq i} }^n \frac{\gamma_j - \delta'_j}{(1 + \delta'_j)(1 + \gamma_j)}\Sigma_i \tilde Q^{\delta'}_{-i} \Sigma_j\tilde Q^\nu_{-i}\right),
  \end{align*}
  and if we introduce the vector $\varepsilon=(\varepsilon_l)_{1\leq l \leq k}$ defined as $\varepsilon_l=\frac{\left\vert \nu_i - \delta'_i\right\vert}{\sqrt{(1+\nu_i)(1+\delta'_i)}}$~:
  \begin{align}\label{eq:maj_gamma-delta}
    \left\vert \varepsilon_i\right\vert
    &=\frac{1}{n} \tr \left(\frac{1}{n}\sum_{\genfrac{}{}{0pt}{2}{j=1}{j\neq i} }^n \frac{\Sigma_i^{\frac 12} \tilde Q^{\delta'}_{-i} \Sigma_j^{\frac 12}}{\sqrt{(1 + \delta'_j)(1 + \delta'_i)}}\frac{\Sigma_j^{\frac 12} \tilde Q^\nu_{-i} \Sigma_i^{\frac 12}}{\sqrt{(1 + \gamma_j)(1 + \nu_i)}} \left\vert \varepsilon_j\right\vert\right) \nonumber\\
    & \leq \sqrt{\frac{1}{n}\frac{\tr \left(\Sigma_i \tilde Q^{\delta'}_{-i} \Sigma_{-i}^{\delta'}\tilde Q^{\delta'}_{-i}\right)}{1+\delta'_i} } \ \ \sqrt{\frac{1}{n}\frac{\tr \left(\Sigma_i \tilde Q^\nu_{-i} \Sigma_{-i}^\nu\tilde Q^\nu_{-i}\right) }{1+\nu_i} } \ \ \left\Vert \varepsilon\right\Vert_\infty.
  \end{align}
  Recall that by definition of the resolvent $\tilde Q^\sigma=Q_{\Sigma_{-i}^\sigma}$, we have the identity
  $$\tilde Q^\sigma_{-i} \Sigma_{-i}^\sigma + z \tilde Q^\sigma_{-i} = I_p.$$
  Thus, for any $\sigma\in \mathbb{R}_+^n$ such that $I(\sigma)=(\frac{1}{n}\tr (\Sigma_i \tilde Q_{-i}^\sigma))_{1\leq l \leq k}=\sigma$ and:
  \begin{align}\label{eq:borne_pour_delta_prime}
    \frac{1}{n}\tr \left(\Sigma_i \tilde Q_\sigma \Sigma_{-i}^\sigma\tilde Q_{-i}^\sigma\right) 
    &= \sigma_l- \frac{z}{n}\tr \left(\Sigma_i (\tilde Q_{-i}^\sigma)^2\right) \ \ <  \ 1+\sigma_l  .
  \end{align}
  Therefore, we know from the inequality~\eqref{eq:maj_gamma-delta}, true for every $i\in \{1, \ldots, n\}$, that $\left\Vert \varepsilon\right\Vert_\infty=0$, in other words $\delta'=\nu$, the fixed point is unique.
\end{proof}
\begin{proposition}\label{pro:controle_delta-tilde_delta}
  If $p =O(n)$, $\frac{1}{z_0} =O(1)$, $q\geq2$ and $n$ is large enough, there exists a constant $C>0$ such that~: 
  \begin{align*}
    \left\Vert \delta- \delta'\right\Vert_\infty = O \left(\sqrt{\frac{\log n}{n}}\right)&
    &\text{and}&
    &\left\Vert \tilde Q^\delta - \tilde Q^{\delta'}\right\Vert = O \left(\sqrt{\frac{\log n}{n}}\right).
  \end{align*}
\end{proposition}
\begin{proof}
 The bounds $p =O(n)$ and $\frac{1}{z_0} =O(1)$ implies that $\bar \gamma = O(1)$ and in particular $a_{p,n} = \log(n)^{1/2} \sqrt{\frac{\bar \gamma}{z_0^3 n}} = O (\sqrt{\log(n) /n})$.
  Let us employ as in the proof of Proposition~\ref{pro:definition_delta_recursive} a vector $\varepsilon=\delta-\delta'$. Given $i\in \{1,\ldots, n\}$, we compute~:
  \begin{align*}
    \left\vert \varepsilon_i\right\vert 
    &=\frac{1}{n}\left\vert \tr \Sigma_i (\mathbb{E}Q_{-i}-\tilde Q_{-i}^{\delta'}) \right\vert\\
    &\leq \frac{\tr \Sigma_i}{n}\left\Vert \mathbb{E}Q_{-i}-\tilde Q_{-i}^{\delta}\right\Vert + \frac{1}{n}\left\vert \tr \Sigma_i (\tilde Q_{-i}^\delta-\tilde Q_{-i}^{\delta'}) \right\vert \\
    &\leq C \sqrt{\frac{\log n}{n}} + \sqrt{\frac{1}{n}\frac{\tr \left(\Sigma_i \tilde Q_{-i}^{\delta'} \Sigma_{-i}^{\delta'}\tilde Q_{-i}^{\delta'}\right)}{1+\delta'_i} } \ \ \sqrt{\frac{1}{n}\frac{\tr \left(\Sigma_i \tilde Q^\delta \Sigma_\delta\tilde Q^\delta\right) }{1+\delta_i} } \ \ \left\Vert \varepsilon\right\Vert_\infty,
  \end{align*}
  where we employed for the first inequality the result of Lemma~\ref{lem:borne_trace_norme_spec} and in the second inequality Propositions~\ref{pro:borne_EQ_m_tQ} and \ref{pro:borne_tr_Sigma_l} and the intermediate result \eqref{eq:maj_gamma-delta} of the proof of Proposition~\ref{pro:definition_delta_recursive}. We already know from \eqref{eq:borne_pour_delta_prime} that~:
  \begin{align*}
    \frac{1}{n}\frac{\tr \left(\Sigma_i \tilde Q_{-i}^{\delta'} \Sigma_{-i}^{\delta'}\tilde Q_{-i}^{\delta'}\right)}{1+\delta'_i} \leq \frac{\delta'}{1+\delta'}<1.
 \end{align*} 
 Besides, we know from Lemma~\ref{lem:controle_Q} that $\Vert (\tilde Q_{-i}^\delta)^{\frac 12}\Sigma_{-i}^\delta (\tilde Q_{-i}^\delta)^{\frac12}\Vert =\Vert \Sigma_{-i}^\delta\tilde Q_{-i}^\delta \Vert \leq 1$ and we can infer from Lemma~\ref{lem:borne_trace_norme_spec} and Proposition~\ref{pro:borne_EQ_m_tQ} that:
\begin{align*}
    \frac{1}{n}\frac{\tr \left(\Sigma_i \tilde Q_{-i}^\delta \Sigma_{-i}^\delta\tilde Q_{-i}^\delta\right) }{1+\delta_i} 
    &\leq \frac{1}{n}\frac{\tr \left((\tilde Q_{-i}^\delta)^{\frac12}\Sigma_i(\tilde Q_{-i}^\delta)^{\frac12} \right) }{1+\delta_i} \\
    &\leq \frac{1}{n}\frac{ \tr \left(\Sigma_i \mathbb E[Q]\right)  + \tr \Sigma_i \Vert \tilde Q_{-i}^\delta - \mathbb E[Q]\Vert }{1+\delta_i} \\
    &\leq \frac{\delta_i}{1+\delta_i} + O \left(\sqrt{\frac{\log n}{n}}\right), 
     \end{align*}
  and we know from \eqref{eq:maj_delta'} that $\frac{\delta_i}{1+\delta_i}\leq 1- \frac{z_0}{\bar \gamma}$. Therefore since $\frac{z_0}{\bar \gamma}$ is bounded from below, if $n$ is large enough~:
  \begin{align*}
    \frac{1}{n}\frac{\tr \left(\Sigma_i \tilde Q_{-i}^\delta \Sigma_{-i}^\delta\tilde Q_{-i}^\delta\right) }{1+\delta_i} <1 - \frac{z_0}{2\bar \gamma}&
    &\text{and thus}&
    &\left\Vert \varepsilon\right\Vert_\infty \leq  \frac{C \sqrt{\frac{\log n}{n}}}{1 - \frac{z_0}{2\bar \gamma}} = O \left(\sqrt{\frac{\log n}{n}}\right).
  \end{align*}
  Now, let us bound the spectral norm of the difference $\tilde Q^\delta -\tilde Q^{\delta'}$~:
  \begin{align*}
    \left\Vert \tilde Q^\delta -\tilde Q^{\delta'}\right\Vert
    &\leq \frac{1}{n}\sum_{i=1}^n \left\vert \delta_i-\delta'_i\right\vert\frac{\left\Vert\tilde Q^{\delta'} \Sigma_i \tilde Q^{\delta}\right\Vert}{(1+\delta_i)(1+\delta'_i)}\\
    &\leq\left\Vert \delta - \delta'\right\Vert_{\infty}  \sqrt{\left\Vert\tilde Q^{\delta'} \Sigma^{\delta'} Q^{\delta'}\right\Vert \left\Vert\tilde Q^{\delta} \Sigma^{\delta} Q^{\delta}\right\Vert}\\
    &\leq  \frac{\left\Vert \delta - \delta'\right\Vert_{\infty}}{z_0}
    \ \  = O \left(\sqrt{\frac{\log n}{n}}\right).
  \end{align*}
\end{proof}
We now have all the elements to set the linear concentration of $Q$ around the second deterministic equivalent $\tilde Q^{\delta'}$ with the same arguments as those given to justify Proposition~\ref{pro:premier_equivalent_deterministe}.
\begin{theorem}\label{the:deuxiemme_equi_det}
  If $p = O(n)$, $\frac{1}{z_0}= O(1)$ and $n$ is large enough, there exist two numerical constants $C\geq e$ and $c>0$ such that~:
  $$Q \in \tilde Q^{\delta'} \pm Ce^{-(\sqrt{n/\log n}\, \cdot \,)^q/c} \ \ \text{in} \ \left(\mathcal M_{p,n}, \left\Vert \cdot\right\Vert\right).$$
\end{theorem}
As for Proposition~\ref{pro:conc_Steiltjes}, we can directly deduce that $\frac{1}{p}\tr \tilde Q^\delta$ is a pivot of the Stieltjes transform. This time the linear form $M \mapsto \frac{1}{p}\tr M$ is seen as a $1$-Lipschitz transformation from $(\mathcal   M_{p,n}, \left\Vert \cdot\right\Vert)$ to $\mathbb{ R}$ (see Lemma~\ref{lem:borne_trace_norme_spec}).
\begin{corollary}[Estimation of the Stieltjes transform]\label{cor:conc_steilt_trans_autour_tilde_Q_delta}
  In the setting of Theorem~\ref{the:deuxiemme_equi_det}, for $z>0$~: 
  \begin{align*}
    m_F(-z)\in\frac{1}{p}\tr \tilde Q^{\delta'}(z) \pm Ce^{-(\sqrt{n/\log n}\, \cdot \,)^q/c}.
  \end{align*}
\end{corollary}
\begin{remark}[Central limit theorem for covariance matrices]\label{rem:loi_spectrale_distr_gaussi}
Let us define~: $$X^\mathcal{N}=(\mathring \Sigma_1x_1^{\mathcal{N}} + m_1, \cdots ,\mathring \Sigma_nx_n^{\mathcal{N}}+ m_n),$$ where $x_1^{\mathcal{N}}, \cdots x_n^{\mathcal{N}}$ are independent Gaussian vectors with zero mean and unit variance entries and for every $i\in \{1,\ldots, n\}$, $\mathring \Sigma_i = \Sigma_i -m_im_i^T$ is the classical population covariance of $x_i$.

From the definition of $\delta'$ and $\tilde Q^{\delta'}$ we can remark that the deterministic equivalent $\tilde Q^{\delta'}$ is the same for a resolvent $Q$ constructed with the sample covariance of $X$ or of the matrix  $X^\mathcal{N}$. This implies that the asymptotic spectral distribution of the sample covariance matrix of $X$ strictly depends on the means and the covariances of the laws $\mu_l$, $1\leq l \leq k$, but not at all on the intrinsic distribution of those laws. In that sense, the Gaussian case describes all the possible asymptotic spectral distributions of sample covariances of any concentrated data respecting our basic assumptions. 
\end{remark}
\subsection{Illustration of the results}


\begin{figure}[H]
\centering
\begin{tabular}{cc}
\begin{tikzpicture}   
  \begin{axis}[ y label style={at={(-0.08,0.5)}}, width = 0.55\textwidth,xlabel=eigenvalues $\lambda$, ylabel=spectral density $dF(\lambda)$, legend cell align = left, legend style={draw=none},] 
    \addplot [blue,ybar,fill, fill opacity=0.3, bar width = 0.01,ybar legend] coordinates {(0.01,10.45) (0.03,4.25) (0.05,3.2) (0.07,2.65) (0.09,2.15) (0.11,2.1) (0.13,1.8) (0.15,1.7) (0.17,1.55) (0.19,1.45) (0.21,1.4) (0.23,1.15) (0.25,1.2) (0.27,1.1) (0.29,1.05) (0.31,0.95) (0.33,0.85) (0.35,0.9) (0.37,0.85) (0.39,0.75) (0.41,0.7) (0.43,0.75) (0.45,0.65) (0.47,0.65) (0.49,0.6) (0.51,0.55) (0.53,0.55) (0.55,0.45) (0.57,0.55) (0.59,0.4) (0.61,0.4) (0.63,0.35) (0.65,0.45) (0.67,0.35) (0.69,0.25) (0.71,0.25) (0.73,0.15) (0.75,0.2) (0.77,0.15) (0.79,0.1)} \closedcycle;
    \addplot[sharp plot,black!60, line width = 1,domain=1:40,samples=100,line legend] coordinates {(0.005,10.458) (0.015,5.987) (0.025,4.597) (0.025,4.597) (0.035,3.851) (0.045,3.367) (0.045,3.367) (0.055,3.018) (0.065,2.751) (0.065,2.751) (0.075,2.538) (0.085,2.362) (0.085,2.362) (0.095,2.214) (0.105,2.086) (0.105,2.086) (0.115,1.974) (0.125,1.876) (0.125,1.876) (0.135,1.787) (0.145,1.708) (0.145,1.708) (0.155,1.635) (0.165,1.569) (0.165,1.569) (0.175,1.508) (0.185,1.452) (0.185,1.452) (0.195,1.399) (0.205,1.351) (0.205,1.351) (0.215,1.305) (0.225,1.261) (0.225,1.261) (0.235,1.221) (0.245,1.182) (0.245,1.182) (0.255,1.146) (0.265,1.111) (0.265,1.111) (0.275,1.078) (0.285,1.047) (0.285,1.047) (0.295,1.017) (0.305,0.988) (0.305,0.988) (0.315,0.96) (0.325,0.934) (0.325,0.934) (0.335,0.908) (0.345,0.883) (0.345,0.883) (0.355,0.859) (0.365,0.836) (0.365,0.836) (0.375,0.814) (0.385,0.792) (0.385,0.792) (0.395,0.771) (0.405,0.75) (0.405,0.75) (0.415,0.73) (0.425,0.711) (0.425,0.711) (0.435,0.692) (0.445,0.673) (0.445,0.673) (0.455,0.655) (0.465,0.637) (0.465,0.637) (0.475,0.62) (0.485,0.603) (0.485,0.603) (0.495,0.586) (0.505,0.57) (0.505,0.57) (0.515,0.554) (0.525,0.538) (0.525,0.538) (0.535,0.522) (0.545,0.507) (0.545,0.507) (0.555,0.491) (0.565,0.476) (0.565,0.476) (0.575,0.461) (0.585,0.446) (0.585,0.446) (0.595,0.431) (0.605,0.416) (0.605,0.416) (0.615,0.401) (0.625,0.387) (0.625,0.387) (0.635,0.372) (0.645,0.357) (0.645,0.357) (0.655,0.342) (0.665,0.327) (0.665,0.327) (0.675,0.312) (0.685,0.297) (0.685,0.297) (0.695,0.281) (0.705,0.265) (0.705,0.265) (0.715,0.248) (0.725,0.231) (0.725,0.231) (0.735,0.213) (0.745,0.195) (0.745,0.195) (0.755,0.175) (0.765,0.152) (0.765,0.152) (0.775,0.128) (0.785,0.097) (0.785,0.097) (0.795,0.053) (0.805,0.0) (0.805,0.0) (0.815,0.0) (0.825,0.0)};
     \legend{empirical distribution, prediction from $\frac{1}{p}\tr \tilde Q^{\delta'}$} 
  \end{axis} 
\end{tikzpicture}
 \begin{tikzpicture}    
  \begin{axis}[width = 0.55\textwidth,xlabel=eigenvalues $\lambda$, legend cell align = left, legend style={draw=none},] 
    \addplot [blue,ybar,fill, fill opacity=0.3, bar width = 0.03,ybar legend] coordinates {(0.025,8.52) (0.075,3.08) (0.125,2.12) (0.175,1.52) (0.225,1.2) (0.275,0.8) (0.325,0.56) (0.375,0.2) (0.425,0.0) (0.475,0.0) (0.525,0.0) (0.575,0.04) (0.625,0.06) (0.675,0.08) (0.725,0.12) (0.775,0.12) (0.825,0.1) (0.875,0.12) (0.925,0.1) (0.975,0.1) (1.025,0.12) (1.075,0.08) (1.125,0.1) (1.175,0.1) (1.225,0.1) (1.275,0.08) (1.325,0.08) (1.375,0.08) (1.425,0.04) (1.475,0.1) (1.525,0.04) (1.575,0.06) (1.625,0.06) (1.675,0.04) (1.725,0.02) (1.775,0.06)} \closedcycle;
    \addplot[sharp plot,black!60, line width = 1,domain=1:40,samples=100,line legend] coordinates {(0.012,8.602) (0.038,4.705) (0.062,3.444) (0.062,3.444) (0.088,2.743) (0.112,2.273) (0.112,2.273) (0.138,1.926) (0.162,1.652) (0.162,1.652) (0.188,1.427) (0.213,1.235) (0.213,1.235) (0.238,1.066) (0.262,0.912) (0.262,0.912) (0.288,0.768) (0.312,0.628) (0.312,0.628) (0.338,0.482) (0.362,0.311) (0.362,0.311) (0.388,0.0) (0.412,0.0) (0.413,0.0) (0.438,0.0) (0.462,0.0) (0.462,0.0) (0.488,0.0) (0.513,0.0) (0.512,0.0) (0.538,0.0) (0.562,0.0) (0.562,0.0) (0.588,0.042) (0.612,0.069) (0.612,0.069) (0.638,0.084) (0.662,0.094) (0.662,0.094) (0.688,0.101) (0.712,0.105) (0.712,0.105) (0.737,0.108) (0.762,0.11) (0.762,0.11) (0.788,0.112) (0.812,0.112) (0.812,0.112) (0.838,0.112) (0.862,0.112) (0.862,0.112) (0.888,0.111) (0.912,0.11) (0.912,0.11) (0.938,0.109) (0.962,0.108) (0.962,0.108) (0.987,0.107) (1.012,0.105) (1.012,0.105) (1.037,0.104) (1.062,0.102) (1.062,0.102) (1.088,0.1) (1.112,0.098) (1.112,0.098) (1.138,0.096) (1.162,0.094) (1.162,0.094) (1.187,0.092) (1.212,0.09) (1.212,0.09) (1.237,0.088) (1.262,0.086) (1.262,0.086) (1.287,0.084) (1.312,0.081) (1.312,0.081) (1.338,0.079) (1.362,0.077) (1.362,0.077) (1.388,0.074) (1.412,0.072) (1.412,0.072) (1.437,0.069) (1.462,0.067) (1.462,0.067) (1.487,0.064) (1.512,0.062) (1.512,0.062) (1.537,0.059) (1.562,0.057) (1.562,0.057) (1.588,0.054) (1.612,0.051) (1.612,0.051) (1.638,0.048) (1.662,0.045) (1.662,0.045) (1.687,0.041) (1.712,0.038) (1.712,0.038) (1.737,0.034) (1.762,0.03) (1.762,0.03) (1.787,0.025) (1.812,0.019) (1.812,0.019) (1.838,0.011) (1.862,0.0)};
  \end{axis} 
\end{tikzpicture}
\end{tabular}

\caption{Spectral distribution of the empirical covariance matrix of two Gaussian populations with $p=n=1000$. On the left, the whole population belongs to one class with covariance $\Sigma = T +9 T^2$ where $T$ is a symmetric Toeplitz matrix with first row equal to $(0.1,0.01,\ldots, 0.1^p)$; on the right the sample contains two classes of data, $n_1 = 100$ samples have a covariance $\Sigma_1 = 10 T$ and the $n_2=900$ samples left have a covariance $\Sigma_2 = 10 T^2$. One can note that although the covariance of the population of the second sample is $\frac{100}{1000} \Sigma_1 + \frac{900}{1000} \Sigma_2 =  T + 9 T^2 = \Sigma$, the spectral distributions of the sample covariances are different.} \label{fig:donnees_artif}
\end{figure}
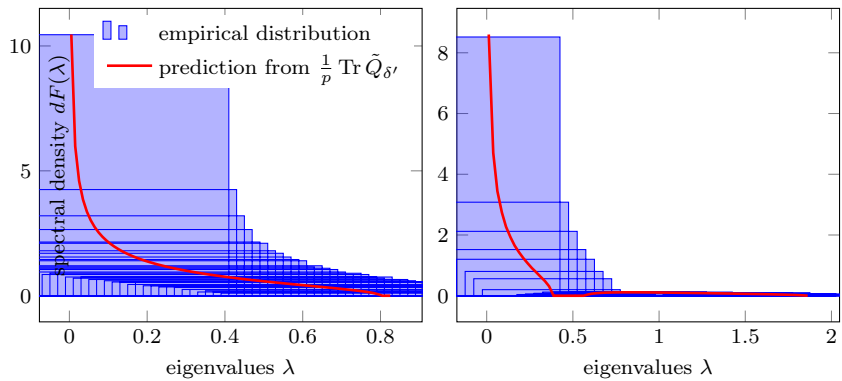
As it was exposed in the preamble, the design of the deterministic equivalent allows us to estimate the spectral distribution $F$ of the sample covariance matrix $S = \frac 1 n XX^T$ thanks to the approximation of the Stieltjes transform $m_F(z) = \frac 1 p \tr Q(z)$ with $\frac 1 p \tr \tilde Q(z)$. The interested reader can find in the abundant literature of random matrices a complete description of this inference, we can cite for instance~\cite{Baisil09}. We just give here some figures to illustrate our results. Figure~\ref{fig:donnees_artif} reveals that the multi-class setting (right display) differs from the uni-class setting (left display) at least when the $\delta_i$, $1\leq i\leq n$ are different from one another. Indeed although the population covariances of the samples from which are drawn the two histograms are rigorously the same, the histogram of the spectral distribution of the empirical covariance matrices are completely different. We drew the histogram from Gaussian data, but we obtain of course the same histograms for Bernouilli populations or any convexly concentrated data.

\begin{figure}[ht]
\centering
\begin{tabular}{cc}
 \begin{tikzpicture}
  \begin{axis}[y label style={at={(-0.08,0.5)}}, width = 0.55\textwidth, height = 0.57\textwidth, xlabel=eigenvalues $\lambda$, ylabel=$dF(\lambda^{0.1})$, legend cell align = left, legend style={draw=none},xtick={0,0.5,1,1.5,2},xmax =2.4, ymax = 5,ytick={0,1,2,3,4}, xticklabels={$0$,$0.5^{10}$,$1^{10}$,$1.5^{10}$, $2^{10}$}] 
    \addplot [smooth, black!60,ybar,fill, fill opacity=0.3, bar width = 0.02,ybar legend] coordinates
    {(0.05,0.077) (0.1,2.143) (0.15,3.265) (0.2,2.194) (0.25,0.816) (0.3,0.51) (0.35,0.357) (0.4,0.332) (0.45,0.638) (0.5,1.097) (0.55,2.296) (0.6,1.939) (0.65,1.429) (0.7,1.097) (0.75,0.638) (0.8,0.383) (0.85,0.255) (0.9,0.153) (0.95,0.153) (1.0,0.077) (1.05,0.051) (1.1,0.051) (1.15,0.0) (1.2,0.051)}
    \closedcycle;
    \addplot [smooth, blue,ybar,fill, fill opacity=0.3, bar width = 0.02,ybar legend] coordinates
    { (0.525,1.5) (0.575,2.6) (0.625,3.0) (0.675,3.3) (0.725,2.7) (0.775,2.3) (0.825,1.5) (0.875,0.8) (0.925,0.8) (0.975,0.5) (1.025,0.3) (1.075,0.2) (1.125,0.2) (1.175,0.1) (1.225,0.0) (1.275,0.1) (1.325,0.0) (1.375,0.0) (1.425,0.0) (1.475,0.0) (1.525,0.1)}
    \closedcycle;
    \addplot[smooth, black!60, line width = 1,domain=1:40,samples=100,line legend] coordinates 
    {(2.275,0.0) (1.575,0.0) (1.55,0.0) (1.525,0.0) (1.5,0.0) (1.475,0.0) (1.45,0.0) (1.425,0.0) (1.4,0.0) (1.375,0.0) (1.35,0.0) (1.325,0.0) (1.3,0.0) (1.275,0.0) (1.25,0.195) (1.225,0.0) (1.2,0.158) (1.175,0.0) (1.15,0.221) (1.125,0.219) (1.1,0.165) (1.075,0.006) (1.05,0.381) (1.025,0.314) (1.0,0.367) (0.975,0.561) (0.95,0.607) (0.925,0.632) (0.9,0.769) (0.875,1.048) (0.85,1.272) (0.825,1.518) (0.8,1.871) (0.775,2.284) (0.75,2.636) (0.725,2.929) (0.7,3.151) (0.675,3.31) (0.65,3.386) (0.625,3.306) (0.6,3.101) (0.575,2.755) (0.55,2.189) (0.525,1.148) (0.5,0.395) (0.49,0.1) (0.48,0.02) (0.47,0) (0.46,0) (0.45,0) (0,0)
    };  
    \legend{spectrum of $\Sigma$, spectrum of $S$, prediction from $\frac{1}{p}\tr \tilde Q^{\delta'}$} 
  \end{axis}
\end{tikzpicture} 
\begin{tikzpicture}   
  \begin{axis}[ width = 0.55\textwidth,height = 0.57\textwidth, xbar, xlabel=eigenvalues $\lambda$, legend cell align = right, legend style={draw=none},xtick={0,0.5,1,1.5,2}, ymax = 5,ytick={0,1,2,3,4}, xmax =2.4, xticklabels={$0$,$0.5^{10}$,$1^{10}$,$1.5^{10}$, $2^{10}$}] 
    \addplot [black!60,ybar,fill, fill opacity=0.3, bar width = 0.017,ybar legend] coordinates 
    {(0.2,0.128) (0.25,1.224) (0.3,3.929) (0.35,2.577) (0.4,2.704) (0.45,2.5) (0.5,2.194) (0.55,1.607) (0.6,1.097) (0.65,0.689) (0.7,0.434) (0.75,0.255) (0.8,0.179) (0.85,0.077) (0.9,0.077) (0.95,0.077) (1.0,0.051) (1.05,0.026) (1.1,0.051) (1.15,0.051) (1.2,0.0) (1.25,0.026) (1.3,0.026) (1.35,0.0) (1.4,0.0) (1.45,0.0) (1.5,0.0) (1.55,0.0) (1.6,0.026) };
    \addplot [blue,ybar,fill, fill opacity=0.3, bar width = 0.017,ybar legend] coordinates 
    {(0.425,0.5) (0.475,2.9) (0.525,3.3) (0.575,3.2) (0.625,3.0) (0.675,1.9) (0.725,1.5) (0.775,1.2) (0.825,0.5) (0.875,0.4) (0.925,0.3) (0.975,0.3) (1.025,0.1) (1.075,0.1) (1.125,0.2) (1.175,0.2) (1.225,0.2) (1.275,0.0) (1.325,0.0) (1.375,0.1) (1.425,0.0) (1.475,0.0) (1.525,0.0) (1.575,0.0) (1.625,0.0) (1.675,0.0) (1.725,0.0) (1.775,0.0) (1.825,0.0) (1.875,0.0) (1.925,0.0) (1.975,0.0) (2.025,0.0) (2.075,0.0) (2.125,0.0) (2.175,0.0) (2.225,0.1) }
    \closedcycle;
    \addplot[smooth,black!60, line width = 1,domain=1:40,samples=100,line legend] coordinates
    {(2.275,0.0) (2.25,0.0) (2.225,0.0) (2.2,0.0) (2.175,0.0) (2.15,0.0) (2.125,0.0) (2.1,0.0) (2.075,0.0) (2.05,0.0) (2.025,0.0) (2.0,0.0) (1.975,0.0) (1.95,0.0) (1.925,0.0) (1.9,0.0) (1.875,0.0) (1.85,0.0) (1.825,0.0) (1.8,0.0) (1.775,0.0) (1.75,0.0) (1.725,0.0) (1.7,0.0) (1.675,0.0) (1.65,0.0) (1.625,0.0) (1.6,0.0) (1.575,0.0) (1.55,0.0) (1.525,0.0) (1.5,0.0) (1.475,0.0) (1.45,0.0) (1.425,0.0) (1.4,0.0) (1.375,0.0) (1.35,0.0) (1.325,0.0) (1.3,0.0) (1.275,0.0) (1.25,0.185) (1.225,0.0) (1.2,0.183) (1.175,0.189) (1.15,0.0) (1.125,0.146) (1.1,0.198) (1.075,0.205) (1.05,0.236) (1.025,0.245) (1.0,0.0) (0.975,0.103) (0.95,0.312) (0.925,0.356) (0.9,0.377) (0.875,0.269) (0.85,0.591) (0.825,0.697) (0.8,0.787) (0.775,1.019) (0.75,1.261) (0.725,1.596) (0.7,1.923) (0.675,2.295) (0.65,2.697) (0.625,3.088) (0.6,3.428) (0.575,3.68) (0.55,3.795) (0.525,3.71) (0.5,3.329) (0.475,2.408) (0.45,1) (0.425,0.5) (0.4,0.1) (0.39,0) (0.38,0) (0.15,0)};
  \end{axis} 
\end{tikzpicture}
\end{tabular}
\caption{Distribution of the (pushforward of the) spectral distributions of the population covariance matrix $\Sigma$ and sample covariance matrix $S$ of two datasets constructed from the MNIST database (\cite{MNIST}) through the map $\lambda\mapsto \lambda^{0.1}$ ; the population covariance and the mean are empirically computed from a larger and independent datasets of $50,000$ elements. {\bf (Left):} images of $2$'s generated from a generative adversarial network (GAN) trained on the MNIST database (see Figure~\ref{fig:image_de_2}); $p=784$ and $n=200$. {\bf (Right):} convolution neural network (CNN) features of those images; $n=200$ and $p=784$ (for fair comparison, we took the number of features equal to the number of pixels in the original images). The CNN is described in \cite{SED20}: it is trained on the MNIST database and achieves $99.1\%$ accuracy to discriminate the $10$ digits.} \label{fig:donnees_reelles}
\end{figure}

The range of validity of our results seems to be rather wide. Indeed, Figure~\ref{fig:donnees_reelles} shows that our predictions are rather good for data $x_i$ taken as ``raw'' (untreated) MNIST handwritten digits \cite{MNIST} (left display) as well as for $x_i$ arising from feature extractions of the same handwritten digits by a convolutional neural network (CNN) (right display). In the CNN, these features are given by the penultimate layer preceding the classification layer that will attribute a score for each class. To be precise, we generated the MNIST images with a generative adversarial network (GAN) as presented in \cite{SED20}: we here used a small sample of these images. This technique allows us to obtain a sufficiently large dataset ($50,000$ elements for data vectors of size $784$) so that the population covariance matrix of the digits can be accurately estimated. Yet, one must be careful that we drew the distribution of a push-forward of the spectral distribution dilated around zero; otherwise the distribution would be huddled against $0$ and at the same time sparsely spread up to high values.
The spectral distribution is more concentrated for the sample covariance of the refined feature dataset and satisfies our prediction better than the generated digits. This suggests that, from a concentration of measure perspective, the treatment of the signal by the network neurons tends to ``normalize'' the data.   
\begin{figure}[ht]
    \label{fig:image_de_2}
    \centering
    \includegraphics[width = 1\textwidth]{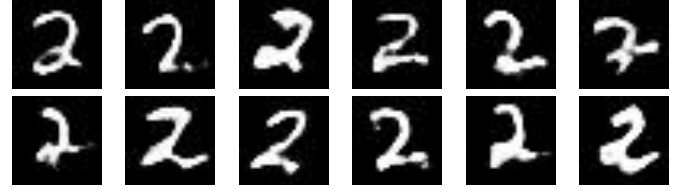}
\caption{Sample of GAN generated images of the MNIST database (see \cite{SED20} for a description of the generative adversarial network).}
\end{figure}

\newpage

\appendix
\addappheadtotoc
\begin{appendices}

\section{Side results of the concentration of random vectors}\label{app:side_result_concentration}
\subsection{Proof of Proposition~\ref{pro:tao}}
  
  Given $\varepsilon\in (0,1)$, a set $A \subset \mathcal B_H$ is said to be an $\varepsilon$-net of $\mathcal B_H$ if $x,y \in A \Rightarrow \left\Vert x-y\right\Vert \geq \varepsilon$. We consider here $N_{1/2}$, a maximal $\frac{1}{2}$-net of $\mathcal B_H$ with respect to inclusion.
  We know that the balls of radius $\frac{1}{4}$ centered on the points of $N_{1/2}$ are all disjoint by hypothesis, and their volume is equal to $ \mathcal V_{\mathcal B_H}/4^p$ (where $\mathcal V_{\mathcal B_H}$ is the volume of $\mathcal B_H$). Since they all belong to the ball of radius $2$ and centered at the origin, we know that their number cannot exceed $8^{{\rm dim}(H)}$.

  Besides, given a drawing of $Z$, there exists $f_0 \in \mathcal B_{H}$ such that $Z-\tilde Z = f_0(Z-\tilde Z)$ (since $\mathcal B_H$ is compact). Then there exists $f \in \mathcal B_H$ such that $\left\Vert f-f_0\right\Vert_*$ is bounded by $\frac{1}{2}$ (otherwise $f_0$ could be added to $N_{1/2}$). Furthermore : 
  \begin{align*}
    \left\Vert Z - \tilde{Z}\right\Vert - f\left(Z - \tilde{Z}\right) 
    &\leq\left\vert f\left(Z - \tilde{Z}\right)  -f_0 \left(Z - \tilde{Z}\right)\right\vert \\
    &\leq \left\Vert f-f_0\right\Vert_*\left\Vert Z - \tilde{Z}\right\Vert\ \leq \ \ \frac{1}{2} \left\Vert Z-\tilde{Z}\right\Vert.
  \end{align*}
  Therefore~:
  \begin{align*}
    \left\Vert Z - \tilde{Z}\right\Vert \leq 2\sup\left\{ f\in N_{1/2} \ : \ f\left(Z - \tilde{Z}\right) \right \}
  \end{align*}
  and this inequality being true for any drawing of $Z$, we have then by hypothesis~:
  \begin{align*}
    \forall t > 0 \ : \ \mathbb{P}\left(\left\Vert Z-\tilde{Z}\right\Vert \geq t\right) \leq \sum_{N_{1/2}} \alpha\left(\frac{t}{2}\right)\leq 8^{{\rm dim}(H)} \alpha\left(\frac{t}{2}\right).
  \end{align*}

\subsection{Linear concentration of the product}

We place ourselves in an algebra $\mathcal{A}$ endowed with an algebra norm $\left\Vert \cdot \right\Vert$ (verifying $\left\Vert xy\right\Vert \leq \left\Vert x\right\Vert \left\Vert y\right\Vert$). To simplify the result we place ourselves in the exponential concentration setting, the reader is required to adapt the proof for generalization if needed.
\begin{proposition}\label{pro:concentration_lineaire_XY}
  Given two random vectors $X,Y \in \mathcal A$ and three parameters $C\geq e$ and $\sigma,q >0$, if $X$ and $Y$ follow the same concentration $X \in \tilde X \pm Ce^{-(\, \cdot \,/\sigma)^q}$ and $Y \in \tilde Y \pm Ce^{-(\, \cdot \,/\sigma)^q}$, then $XY$ is also concentrated~:
  \begin{align*}
    XY \in \tilde X \tilde Y \pm C \exp\left(-\left(\frac{c \,\cdot \,}{\sigma^2 \eta_{\Vert \cdot \Vert}^{\frac{2}{q}}} \right)^{\frac{q}{2}} \right) + C \exp \left(-\left( \frac{c \, \cdot}{\sigma( \Vert \tilde X \Vert + \Vert \tilde Y \Vert) }\right)^q  \right),
  \end{align*}
  where $c$ is a numerical constant independent of $C$ and $\sigma$.
\end{proposition}
\begin{proof}
  As for Lemma~\ref{lem:conc_autour_prod}, we employ the identity~:
  \begin{align*}
    XY- \tilde X\tilde Y = (X- \tilde X) (Y- \tilde Y) + \tilde Y(X- \tilde X ) + \tilde  Y (Y -\tilde Y).
  \end{align*}
  For any linear function $u$ with an operator norm bounded by $1$, we have~:
  \begin{align*}
    u(XY- \tilde X\tilde Y) \leq \Vert X- \tilde X\Vert \Vert Y- \tilde Y\Vert + u_{\tilde Y}(X- \tilde X ) + u_{\tilde  Y }(Y -\tilde Y)
  \end{align*}
  where for any $z \in \mathcal A$ we defined $u_z : x \mapsto u(zx)$. To conclude, we just have to note that $u_z$ is a linear function the operator norm of which is bounded by $\left\Vert z\right\Vert$ (since $\left\Vert u(zx)\right\Vert \leq \left\Vert u\right\Vert \left\Vert z\right\Vert \left\Vert x\right\Vert$).
\end{proof}

\subsection{Concatenation of convexly concentrated random vectors}\label{ssa:conc_conc_conv}
Through the characterization with the centered moments given by Proposition~\ref{pro:carcterisation_ac_moments_conc_q_expo}, the $q$-exponential concentration allows to explore the concentration of a random vector constructed as a concatenation of $p$ independent random vectors $(Z_1,\ldots,Z_p)$. This approach is indifferently adapted to the Lipschitz or convex concentration. We use the index $(c)$ under the sign $\propto$ to specify that the proposition is valid in both settings (for Lipschitz \textit{or} convexly concentrated random vectors).

\begin{proposition}\label{pro:concentration_concatenation_vecteurs_independants_lipsch_et_convexe}
  Given $p$ normed vector spaces $(E_1, N_1),\ldots,(E_p,N_p)$, consider $p$ independent random vectors $(Z_1,\ldots, Z_p) \in E=E_1 \times \cdots \times E_p$ verifying for any $i\in \{1,\ldots,p\}$ that $Z_i \propto_{(c)} C e^{-(\, \cdot \, / \sigma)^q}$, for two given parameters $C\geq e$, $\sigma>0$.
  The space $E$ can then be seen as a normed vector space endowed with the norm $\left\Vert \cdot\right\Vert_{\ell_1}$ defined as~:
  \begin{align*}
    \forall z=(z_1,\ldots, z_p) \in E_1 \times \cdots \times E_p \ \ : \ \ \ \left\Vert z\right\Vert_{\ell_1}= N_1(z_1) + \ldots + N_p(z_p).
  \end{align*} 
  Then the concatenation $Z=(Z_1,\ldots, Z_p)$ is $q$-exponentially concentrated in $(E,\left\Vert \cdot\right\Vert_{\ell_1})$ with an observable diameter lower than $p\sigma e^{\frac{1}{q}}$~:
  $$Z \propto_{(c)} C e^{-(\,\cdot \, / p\sigma)^q/e}.$$
\end{proposition}

\begin{proof}
  Let us consider a function $f: E \rightarrow \mathbb{R}$, $1$-Lipschitz (resp. $1$-Lipschitz and quasiconvex), and $Z'$ an independent copy of $Z$. We plan to employ the characterization with the centered moments given by Proposition~\ref{pro:carcterisation_ac_moments_conc_q_expo}. Given $i\in\{0, \ldots, p\}$, we note $Z^{(i)} = (Z_1,\ldots Z_i, Z_{i+1}',\ldots, Z_p')$ (with this notation : $Z^{(0)}=Z$ and $Z^{(p)}=Z'$). For any $r\geq \max(q,1)$, let us exploit the convexity of $t \mapsto t^r$ to bound~:
  \begin{align*}
    \mathbb{E}\left[\left\vert f(Z)-f(Z')\right\vert^r\right]
    &\leq p^{r-1} \sum_{i=1}^p\mathbb{E}\left[\ \left\vert f\left(Z^{(i-1)}\right)-f\left(Z^{(i)}\right)\right\vert^r \right].
  \end{align*}
  Therefore, since for any $(z_1,\ldots, z_p)\in \mathbb{R}^p$ and for any $i\in \{1,\ldots, p\}$, $z\mapsto f(z_1,\ldots, z_{i-1},z,z_{i+1},\ldots z_p)$ is Lipschitz (resp. Lipschitz and quasiconvex), we can employ Proposition~\ref{pro:carcterisation_ac_moments_conc_q_expo} to bound~:
  \begin{align*}
    \mathbb{E}\left[\left\vert f(Z)-f(Z')\right\vert^r\right]
    \leq C p^r \left(\frac{r}{q}\right)^\frac{r}{q} \sigma^r.
  \end{align*}
  If $q\leq r \leq 1$, the concavity of $t \mapsto t^r$ allows us to write thanks to Jensen's inequality~:
  \begin{align*}
    \mathbb{E}\left[\left\vert f(Z)-f(Z')\right\vert^r\right]
    \leq \left(\mathbb{E}\left[\left\vert f(Z)-f(Z')\right\vert\right]\right)^r 
    \leq C^r p^r \left(\frac{1}{q}\right)^\frac{r}{q} \sigma^r 
    \leq \ p^r \left(\frac{r}{q}\right)^\frac{r}{q} \sigma^r 
  \end{align*}
  since $r\leq 1\leq C^q$.
  The last implication of Proposition~\ref{pro:carcterisation_ac_moments_conc_q_expo} then gives us the desired result: $f(Z) \propto C e^{-(\, \cdot \, / p\sigma)^q/e}$.
\end{proof}


\section{Davis theorem for rectangle matrices}\label{app:davis}

\subsection{Proof of Theorem~\ref{the:convex_matrice_diag}}
Let us first present basic notions to set the theorem. Given a vectorial space $E$ and a group $G$ acting on $E$, for any subset $A\subset E$, we note $G\cdot A=\{g\cdot a, g\in G, a\in A\}$. We say that a set $T$ is \textit{transversal} if $G\cdot T=E$ and we say that a function $f$ is $G$-invariant if $\forall x\in E, \forall g \in G, f(g\cdot x)=f(x)$. In the same vein, we say that a set $A\subset E$ is $G$-\textit{invariant} if $G \cdot A=A$ and that it is $G$-\textit{invariant in $T$} if $A\subset T$ and $G\cdot A \cap T=A$. Given $U \subset T$, we note $L_T^G(U)$ the smallest convex subset of $T$ containing $U$ and $G$-invariant in $T$. We give here an adaptation of one of the result of Grabovsky and Hijab to the case of quasiconvex functions.
\begin{theorem}[cf. \cite{GRA05}, Theorem 4] \label{the:grabowsky}
  Let us consider a vector space $E$, a group $G$ acting on $E$, and a \emph{convex} and \emph{transversal} subset $T\subset E$. 
  We suppose that for any $U\subset  T$, the set $G \cdot  L_{T}^G(U)$ is convex ; we call this property the \emph{convexity conservation of $G$ from T}. Then any $G$-invariant function $f$ is quasiconvex iff its restriction to $T$ is quasiconvex.
\end{theorem}
Intuitively, this theorem states that the quasiconvexity (or the mere convexity as in \cite{GRA05}) of any function $G$-invariant is ``transversal'' to the action of $G$ when $G$ preserves the convexity from $T$. A curious reader might be interested in simplifying the  convexity conservation property as we presented it from a transversal subset $T\subset E$ to a mere conservation of the convexity of any convex subset $U\subset E$ (i.e., for any convex set $U \subset E$, $f(U)$ is convex). This would be indeed an hypothesis more than sufficient for the result of the theorem. However, in practice, and in particular for the applications we want to consider, it cannot be verified. 

\begin{proof}
  Let us note note $\restriction{f}{T}$ the restriction of $f$ on $T$. Given any $t\in \mathbb{R}$, we know that the set $\{\restriction{f}{T}\leq t\}=\{x\in T, f(x)\leq t\}$ is convex. Then the set $\{f\leq t\} = G\cdot \{\restriction{f}{T}\leq t\} = G \cdot L_T^G(\{\restriction{f}{T}\leq t\})$ is also convex thanks to the convexity conservation of $G$ from T. 
\end{proof}
To simplify the application of Theorem~\ref{the:grabowsky}, Grabowsky and Hijab provide us with a useful property~:
\begin{proposition}[Convexity conservation, \cite{GRA05}, Theorem 3]\label{pro:conv_conservation}
  With the notations of Theorem~\ref{the:grabowsky}, if for any $ x,  y \in  T$, the set $G \cdot  L_{{T}}^G(\{ x, y\})$ is convex, then $G$ conserves the convexity from $T$.
\end{proposition}
\begin{proof}
  Let us consider $U \subset T$, and two points $x,y \in G \cdot L_T^G(U)$. There exists $x^*,y^*\in T$ such that $x,y \in G\cdot \{x^*,y^*\}$. We know by hypothesis that the set $G \cdot L_{{T}}^G(\{ x^*, y^*\})$ is convex, and moreover, it contains $x$ and $y$. Therefore, any element of the segment $[x,y]$ is also in $L_{{T}}^G(\{ x^*, y^*\}) \subset L_{{T}}^G(U)$.
\end{proof}

In the case of a matrix concentration, Theorem~\ref{the:grabowsky} can be applied to the transversal set of nonnegative diagonal matrices $\mathcal{D}_{p,n}^+$ for the action of the group $\mathcal{O}_{p,n}$.
The transversal character of $\mathcal{D}_{p,n}^+$ is a consequence of the singular value decomposition. Indeed, for any matrix $M\in \mathcal{M}_{p,n}$, there exists $(U,V)\in \mathcal{O}_{p,n}$ such that $$M=U\Sigma V^T \ \ \ \text{with}  \ \ \Sigma = \text{Diag}_{p,n}(\sigma_i(M))_{1\leq i \leq  d},$$ where $ d=\min(p,n)$, $\sigma_i(M)$ is the $i^{\textit{th}}$ singular value of $M$ and the notation $\text{Diag}_{p,n}(a_i)$ represents an element of $\mathcal{D}_{p,n}^+$ having the values $a_i$ on the diagonal.

To prove Theorem~\ref{the:convex_matrice_diag}, let us characterize the sets $L(X)=L_{\mathcal{D}_{p,n}^+}^{\mathcal{O}_{n,p}}(\{X\})$ when $X\in \mathcal{D}_{p,n}^+$. We know that $\mathcal{D}_{p,n}^+$ is invariant under the action of the subgroup of permutations $\mathcal{P}_{p,n}$. Given a subset $U$ of a vector space, we note $\conv(U)$ the convex hull of $U$.
  \begin{proposition}\label{pro:chract_L(X)=conv(PIX)}
    Given $X\in \mathcal D_{p,n}$, $L(X)=\conv(\mathcal{P}_{p,n} \cdot \{X\})$.
  \end{proposition}
  \begin{proof}
    We know from the uniqueness of the singular value decomposition that for any $U \subset \mathcal{D}_{p,n}^+$, $(\mathcal{O}_{p,n} \cdot U) \cap \mathcal{D}_{p,n}^+ = \mathcal{P}_{p,n} \cdot U$. Consequently, since the convexity is stable under the action of $\mathcal {P}_{p,n}$ (they are linear transformations), $L(X)=\conv(\mathcal{P}_{p,n} \cdot \{X\})$.
  \end{proof}

  Here the tools of \textit{majorization} as presented for instance in \cite{Bha97} are perfectly adapted to the description of $\conv(\mathfrak{S}_{p,n} \cdot \{x\})$ that we identify with $\conv(\mathcal{P}_{p,n} \cdot \{X\})=L(X)$. Given a vector $x=(x_1,\ldots x_d)\in \mathbb{R}^ d$, let us note $x^\downarrow=(x_1^\downarrow, \cdots x_d^\downarrow)$, a decreasing ordered version of $x$ ($x_1^\downarrow\geq \cdots x_d^\downarrow$ and $\exists \sigma \in \mathfrak S_d \ | \ x^\downarrow = \sigma \cdot x$).
\begin{definition}\label{def:majorization}
  Given two vectors $x,y \in \mathbb{R}^ d$, $ d\in \mathbb{N}$, we say that $y$ is majorized by $x$ and we note $y\prec x$ iff :
  \begin{align*}
    \forall k \in \{1,\ldots  d\} \  : \ \  \sum_{i=1}^k y_i^\downarrow \leq \sum_{i=1}^k x_i^\downarrow
    &&\text{and}&&
    \sum_{i=1}^ d x_i = \sum_{i=1}^ d y_i.
  \end{align*}
  
\end{definition}

Majorization offers a complete characterization of $\conv(\mathfrak S_d \cdot \{x\})$, $x\in \mathbb{R}^ d$~:
\begin{theorem}[\cite{Bha97}, Theorem~II.1.10]\label{the:ensemble_x_prec_y}
  Given a vectors $x\in \mathbb{R}^{ d}$ :
  \begin{align*}
    \left\{y\in \mathbb{R}^ d,y\prec x\right\} = \conv \left(\mathfrak S_d\cdot\left\{ x \right\}\right).
  \end{align*}
  
\end{theorem} 

Majorization appears to be the pertect tool to control the singular decomposition of a sum of matrices as we will see in Theorem~\ref{the:in_triangul_val_sing}. This is the core argument to justify the convexity consevation of $\mathcal O_{p,n}$ from $\mathcal D_{p,n}^+$ that we need to prove Theorem~\ref{the:convex_matrice_diag}. Let us first give an intermediate result that we originally owe to Schur and whose proof can be found in \cite{Mar11} or \cite{Bha97}.

\begin{proposition}[Schur's Theorem, B.1. in \cite{Mar11}]\label{pro:Schur_s_theo}
  Given a symmetric matrix $S \in \mathcal M_p$, we have the majorization $\diag(S)\prec \sigma(S)$.
\end{proposition}
This proposition entails a kind of triangular inequality for the set of singular values.
\begin{theorem}[\cite{Bha97}, Exercise II.1.15]\label{the:in_triangul_val_sing}
  Given two matrices $A,B \in \mathcal M_{p,n}$, $\sigma(A+B) \prec \sigma(A) + \sigma(B)$.
\end{theorem}
Of course, it is important that the vectors $\sigma(A)$ and $\sigma(B)$ are both ordered when we sum them.
\begin{proof}
  Given a symmetric matrix $S\in \mathcal M_q$, there exists $(U_S,V_S) \in \mathcal O_{p,n}$ such that $U_AAV_A^T= \diag_{p,n} \sigma(S)$, thus~:
  \begin{align*}
    \sum^k\diag(U_SSV_S^T) = \sum^k \sigma(S),
   \end{align*}
  where given $x\in \mathbb{R}^ d$ and $k\in \{1,\ldots  d\}$, $\sum^k x=\sum_{i=1}^k x_i^\downarrow$. Besides, since $\sigma(USV^T)=\sigma(S)$, we know from Proposition~\ref{pro:Schur_s_theo} that~:
  \begin{align}\label{eq:indent_trace_tronqu\'{e}e_sup}
    \sum^k \sigma(S)
    =\sup_{(U,V)\in \mathcal O_{p,n}}\sum^k\diag(USV^T).
  \end{align}
  Now, if we suppose that we are given a general matrix $A\mathcal M_{p,n}$ and $(U_A,V_A)\in \mathcal O_{p,n}$ such that $U_AAV_A^T=\diag_{p,n}\sigma(A)$. If we introduce the matrices~:
  \begin{align*}
    \tilde A = \left(\begin{array}{cc} (0)&A\\A^T& (0)\end{array}\right)\in \mathcal M_{p+n},&
    &\text{and}&
    &P = \left(\begin{array}{cc} U&(0)\\(0)& V\end{array}\right) \in \mathcal M_{p+n},
   \end{align*} 
  we have the identity~:$$P \tilde A P^T = \left(\begin{array}{cc} (0) & D\\ D&(0)\end{array}\right) \in \mathcal M_{p+n}, \ \ \  \ \ \text{with} \ \ D=\diag_{p,n}(\sigma(A)).$$
  Depending on the relation between $p$ and $n$, we introduce the invertible matrices~:
  \begin{align*}
    \text{if $ d=n$ :    } Q= \left(\begin{array}{ccc} I_d&(0)&I_d\\(0)&I_d& (0)\\-I_d&(0)&I_d\end{array}\right)&
    &\text{and if $ d=p$ :    } Q= \left(\begin{array}{ccc} I_d&I_d&(0)\\-I_d&I_d& (0)\\(0)&(0)&I_d\end{array}\right),&
  \end{align*}
  then if $ d=n$, $QP\tilde A(PQ)^T=\diag(\sigma(A),0\cdots 0, -\sigma(A))$ and if $ d=p$, $QP\tilde A(PQ)^T=\diag(\sigma(A), -\sigma(A),0\cdots 0)$. Thus in both cases, we obtain a diagonalisation of $\tilde A$ that allows us to generalize the identity \eqref{eq:indent_trace_tronqu\'{e}e_sup}) for any matrix $A\in \mathcal M_{p,n}$ and with $0\leq k \leq  d$.
  The supremum of a sum being lower than the sum of a supremum, for any pair of matrices $A,B\in \mathcal M_{p,n}$~:
  \begin{align*}
    \sigma(A+B) \prec \sigma(A)+\sigma(B).
  \end{align*}
\end{proof}
Now that the picture is clearer, we can prove Theorem~\ref{the:convex_matrice_diag}~:

\begin{proof}[Proof of Theorem~\ref{the:convex_matrice_diag}]
  To employ Theorem~\ref{the:grabowsky}, let us show the convexity conservation property of $\mathcal O_{p,n}$ from $\mathcal D_{p,n}^+$. Inspired by Proposition~\ref{pro:conv_conservation}, we consider two non-negative diagonal matrices $X,Y \in \mathcal D_{p,n}^+$, and we note $L(X,Y)=L^{\mathcal D_{p,n}^+}_{\mathcal O_{p,n}}(\{X,Y\})$. We want to show that $\mathcal O_{p,n} \cdot L(X,Y)$ is convex.

  Noting $x=\diag X$ and $y=\diag Y$, let us first show that $L(X,Y)=K(x,y)$, with~: $$K(x,y)=\{Z\in \mathcal D_{p,n}^+, \diag Z \prec \lambda x+(1-\lambda)y, \ 1\leq \lambda\leq 1\}.$$ 
  
  We know from Theorem~\ref{the:ensemble_x_prec_y} that for any $U\in \mathcal D_{p,n}^+$, $L(U)=\{\diag Z\prec \diag U\}$ and therefore, since for any $t\in [0,1]$, $tX+(1-t)Y \in L(X,Y)$, we obtain the first inclusion $K(x,y) \subset L(X,Y)$. 

  To prove the converse inclusion, let us show that $K(x,y)$ is convex (we already know that $\mathcal O_{p,n}\cdot K(x,y) \cap \mathcal D_{p,n}^+=\mathcal P_{p,n} \cdot K(x,y)=K(x,y)$ by definition of the relation $\prec$). 
  
  We consider $A,B\in K(x,y)$, $t\in[0,1]$ and we set $C=tA+(1-t)B$. Therefore, We know that there exists $t_z,t_w\in [0,1]$ such that $\diag A\prec \lambda x + (1-\lambda) y$ and $\diag B\prec \mu x + (1-\mu) y$, therefore : $$\diag C \prec \left(t \lambda+(1-t)\mu\right)\ x \ \  + \ \  (t(1-\lambda)+(1-t)(1-\mu)) \ y \in [x,y].$$
  
  In conclusion, since $X,Y \in K(x,y)$ and $K(x,y)$ is $\mathcal {P}_{p,n}$-invariant and convex we recover the second inclusion $K(x,y) \subset L(X,Y)$.

  Thus we are left to show that $\mathcal O_{p,n} \cdot K(x,y)$ is convex.
  We consider this time $A,B\in \mathcal O_{p,n} \cdot K(x,y)$, $t\in [0,1]$, and we introduce $C=tA +(1-t) B$. We know from Theorem~\ref{the:in_triangul_val_sing} that~:
  \begin{align*}
    \sigma(C) \prec t \sigma(A) + (1-t) \sigma(B),
  \end{align*}
  and as we saw before, that implies that $\diag \sigma(C) \in K(x,y)$. We can then conclude with the relations $$C \in \mathcal O_{p,n}\cdot \{\diag \sigma(C)\} \subset \mathcal O_{p,n}\cdot K(x,y)= \mathcal O_{p,n}\cdot L(X,Y).$$ 
  
  We can apply Proposition~\ref{pro:conv_conservation} to get the hypothesis of Theorem~\ref{the:grabowsky} that entails Theorem~\ref{the:convex_matrice_diag} in our setting.
\end{proof}
\subsection{Proof of Theorem~\ref{the:conc_val_sing}}
  
  Let us first show the Lipschitz character of $\sigma$, it is a well known result that can be found for instance in \cite{Gol96}~:
\begin{lemma}[Theorem 8.1.15 in \cite{Gol96}]\label{lem:lipschitz_val_singul}
  The function $\sigma$ is $1$-Lipschitz. 
\end{lemma}
\begin{proof}
  Given $M,H\in \mathcal{ M}_{p,n}$, and $i\in \{1,\ldots  d\}$ (where as before, $ d=\min(p,n)$, we know from formula \eqref{eq:formule_singuliere}) that~:
  \begin{align*}
    \lambda_i(A+H) \
    &=\min_{\genfrac{}{}{0pt}{2}{F \subset \mathbb{R}^n}{{\rm dim} F \geq n-i+1}} \max_{\genfrac{}{}{0pt}{2}{x\in F}{\left\Vert x\right\Vert=1}} \left\Vert (A +H)x\right\Vert \\
    &\leq \min_{\genfrac{}{}{0pt}{2}{F \subset \mathbb{R}^n}{{\rm dim} F \geq n-i+1}} \left(\max_{\genfrac{}{}{0pt}{2}{x\in F}{\left\Vert x\right\Vert=1} } \left\Vert Ax\right\Vert +\max_{\genfrac{}{}{0pt}{2}{x\in F}{\left\Vert x\right\Vert=1}} \left\Vert Hx\right\Vert\right) \\
    &\leq \min_{\genfrac{}{}{0pt}{2}{F \subset \mathbb{R}^n}{{\rm dim} F \geq n-i+1}} \max_{\genfrac{}{}{0pt}{2}{x\in F}{\left\Vert x\right\Vert=1}} \left\Vert Ax\right\Vert +\max_{\genfrac{}{}{0pt}{2}{x\in \mathbb R^n}{\left\Vert x\right\Vert=1}} \left\Vert Hx\right\Vert \ \ \leq \ \ \lambda_i(A) + \lambda_1(H),
  \end{align*}
  and the same way, we can show that $\lambda_i(A+H) \geq \lambda_i(A) - \lambda_n(H)$. Therefore, we get~:
  \begin{align*}
    \left\vert  \lambda_i(A+H)-\lambda_i(A)\right\vert  \leq \max(\lambda_1(H), \lambda_n(H)) \leq \left\Vert   H\right\Vert.
  \end{align*}

\end{proof} 
\begin{proof}[Proof of Theorem~\ref{the:conc_val_sing}]
  It is a simple corollary of Theorem~\ref{the:convex_matrice_diag}. If $X\propto^T_{\mathcal{O}_{p,n}} \alpha$, and given a $1$-Lipschitz, convex and $\mathfrak S_n$-invariant function $f:\mathbb{R}^ d \rightarrow \mathbb{R}$, one can introduce the function $\tilde F$ defined as~:
  \begin{align*}
    \tilde F :
    \begin{aligned}[t]
      &\mathcal{M}_{p,n}&&\longrightarrow&&\mathbb{R}  \\
      &M&&\longmapsto&&f(\sigma(M)) .
    \end{aligned}
  \end{align*}
  The function $F$ is $1$-Lipschitz thanks to Lemma~\ref{lem:lipschitz_val_singul} and $\mathcal{O}_{p,n}$-invariant because of the uniqueness of the singular decomposition. Besides we can identify the set $\mathcal{D}_{p,n}^+$ with $\mathbb{R}^ d$ and introduce a function $\tilde f : \mathcal{D}_{p,n}^+ \rightarrow \mathbb{R}$ verifying $\tilde f(\text{diag}_{p,n}(x))=f(x)$ for $x\in \mathbb{R}^ d$. In that case, since $f$ is $\mathfrak S_ d$-invariant and convex, $\tilde f$ is also convex and since $\tilde f = \restriction{F}{\mathcal D_{p,n}^+}$, Theorem~\ref{the:convex_matrice_diag} allows us to set that $F$ is also convex. Therefore, the random variable $f(\sigma(X))=F(X)$ is $\alpha$-concentrated by hypothesis on $X$.

  Reciprocally, let us suppose that we are given a random matrix $X\in \mathcal M_{p,n}$ such that $\sigma(X) \propto_{\mathfrak S_d}^{T} \alpha$, and let us consider a $1$-Lipschitz, convex and $\mathcal O_{p,n}$-invariant function $F : \mathcal M_{p,n}\rightarrow \mathbb{R}$. The restriction $\restriction{F}{\mathcal D_{p,n}^+}$ is also $1$ Lipschitz, convex and $\mathcal {P}_{p,n}$-invariant.
  Thus, with the same identification as before between $\mathcal D_{p,n}^+$ and $\mathbb{R}^ d$, we can assert by hypothesis that $F(X)= \restriction{\tilde F}{ \mathcal D_{p,n}^+}(\sigma(X))$ is $\alpha$-concentrated (we defined $ \restriction{\tilde F}{ \mathcal D_{p,n}^+}(x)= \restriction{ F}{ \mathcal D_{p,n}^+}(\text{diag}_{p,n}(x))$, it is a $\mathfrak S_d$-invariant function).
\end{proof}

\end{appendices}

\bibliographystyle{alpha} 

 \bibliography{biblio,bibli_covariance} 

\begin{thebibliography}{SLTC20}

\bibitem[Ada11]{ADA11}
Radoslaw Adamczak.
\newblock On the marchenko-pastur and circular laws for some classes of random
  matrices with dependent entries.
\newblock {\em Electronic Journal of Probability}, 16:1065--1095, 2011.

\bibitem[Bha97]{Bha97}
Rajendra Bhathia.
\newblock {\em Matrix Analysis}.
\newblock Springer, Graduate texts in mathematics, 1997.

\bibitem[BS09]{Baisil09}
Zhidong Bai and Jack~W. Silverstein.
\newblock {\em Spectral Analysis of large dimensional Random Matrices}.
\newblock Springer, 2009.

\bibitem[BZ08]{Bai08t}
Zhidong Bai and Wang Zhou.
\newblock Large sample covariance matrices without independence structures in
  columns.
\newblock {\em Statistica Sinica}, 18:425--442, 2008.

\bibitem[CB16]{COU16}
R.~Couillet and F.~{Benaych-Georges}.
\newblock Kernel spectral clustering of large dimensional data.
\newblock {\em Electronic Journal of Statistics}, 10(1):1393--1454, 2016.

\bibitem[CT15]{CHA15}
Djalil Chafaï and Konstantin Tikhomirov.
\newblock On the convergence of the extremal eigenvalues of empirical
  covariance matrices with dependence.
\newblock {\em arXiv preprint arXiv:1509.02231}, 2015.

\bibitem[Dav57]{DAV57}
Chandler Davis.
\newblock All convex invariant functions of hermitian matrices.
\newblock {\em Archiv der Mathematik 8}, pages 276--278, 1957.

\bibitem[{El }09]{ELK09}
N.~{El Karoui}.
\newblock Concentration of measure and spectra of random matrices: applications
  to correlation matrices, elliptical distributions and beyond.
\newblock {\em The Annals of Applied Probability}, 19(6):2362--2405, 2009.

\bibitem[{El }10]{ELK10}
N.~{El Karoui}.
\newblock The spectrum of kernel random matrices.
\newblock {\em The Annals of Statistics}, 38(1):1--50, 2010.

\bibitem[GH05]{GRA05}
Yury Grabovsky and Omar Hijab.
\newblock A generalization of the chandler davis convexity theorem.
\newblock {\em Advances in Applied Mathematics 34}, pages 192--212, 2005.

\bibitem[GL96]{Gol96}
Gene~H. Golub and Charles F.~Van Loan.
\newblock {\em Matrix computations}.
\newblock Johns Hopkins University press, 1996.

\bibitem[Gro79]{Gro79}
Mikhail Gromov.
\newblock Paul lévy’s isoperimetric inequality.
\newblock {\em Preprint IHES}, 1979.

\bibitem[Gro99]{gro99}
Mikhail Gromov.
\newblock Metric structures for riemannian and non-riemannian spaces.
\newblock In {\em Progress in Math. 152}. Birkhäuser, Boston, 1999.

\bibitem[HZS11]{HUA06}
Guang-Bin Huang, Qin~Yu Zhu, and Chee-Kheong Siew.
\newblock Extreme learning machine: theory and applications.
\newblock {\em IEEE Transactions on Information Theory}, 2011.

\bibitem[KC17]{KAM17}
Abla Kammoun and Romain Couillet.
\newblock Subspace kernel clustering of large dimensional data.
\newblock {\em (submitted to) Journal of Machine Learning Research}, 2017.

\bibitem[LCB98]{MNIST}
Y.~LeCun, C.~Cortes, and C.~Burges.
\newblock {The MNIST database of handwritten digits}, 1998.

\bibitem[Led01]{Led01}
Michel Ledoux.
\newblock {\em The Concentration of Measure Phenomenon}.
\newblock Mathematical Surveys and Monographs, Number 89, 2001.

\bibitem[Led05]{LED05}
Michel Ledoux.
\newblock {\em The concentration of measure phenomenon}, volume~89.
\newblock American Mathematical Soc., 2005.

\bibitem[LLC18]{LOU17b}
Cosme Louart, Zhenyu Liao, and Romain Couillet.
\newblock A random matrix approach to neural networks.
\newblock {\em Annals of Applied Probability}, 28:1190--1248, 2018.

\bibitem[MOA11]{Mar11}
Albert~W. Marshall, Ingram Olkin, and Barry~C. Arnold.
\newblock {\em Inequalities: Theory of Majorization and Its Applications}.
\newblock Springer series in statistics, 2011.

\bibitem[MP67]{MAR67}
V.~A. Mar\u{c}enko and L.~A. Pastur.
\newblock {Distribution of eigenvalues for some sets of random matrices}.
\newblock {\em Math USSR-Sbornik}, 1(4):457--483, 1967.

\bibitem[MS86]{Mil86}
Vitali~D. Milman and Gideon Schechtman.
\newblock {\em Asymptotic Theory of Finite Dimensional Normed Spaces}.
\newblock Springer-Verlag Berlin Heidelberg, 1986.

\bibitem[Ouv09]{Ouvr09}
Jean-Yves Ouvrard.
\newblock {\em Probabilités, tome II}.
\newblock Cassini, 2009.

\bibitem[PP07]{PAJ07}
Alain Pajor and Leonid Pastur.
\newblock On the limiting empirical measure of the sum of rank one matrices
  with log-concave distribution.
\newblock {\em arXiv preprint arXiv:0710.1346}, 2007.

\bibitem[SB95]{SIL95}
J.~W. Silverstein and Z.~D. Bai.
\newblock {On the empirical distribution of eigenvalues of a class of large
  dimensional random matrices}.
\newblock {\em Journal of Multivariate Analysis}, 54(2):175--192, 1995.

\bibitem[SLTC20]{SED20}
Mohamed El~Amine Seddik, Cosme Louart, Mohamed Tamaazousti, and Romain
  Couillet.
\newblock Random matrix theory proves thatdeep learning representations of
  gan-databehave as gaussian mixtures.
\newblock {\em ICML (submitted)}, 2020.

\bibitem[Tal94]{TAL94}
Michel Talagrand.
\newblock The supremum of some canonical processes.
\newblock {\em The Johns Hopkins University Press}, pages Vol. 116, No. 2, pp.
  283--325, 1994.

\bibitem[Tal95]{TAL95}
Michel Talagrand.
\newblock {\em Concentration of Measure and Isoperimetric Inequalities in
  product spaces}.
\newblock Publications math\'ematiques de l'I.H.E.S., tome 81, 1995.

\bibitem[Tao11]{Tao11}
Terence Tao.
\newblock Topics in random matrix theory.
\newblock Department of Mathematics, UCLA, Los Angeles, 2011.

\bibitem[Ver17]{Ver17}
Roman Vershynin.
\newblock High-dimensional probability.
\newblock University of Michigan, 2017.

\bibitem[Yat95]{YAT95}
R.~D. Yates.
\newblock A framework for uplink power control in cellular radio systems.
\newblock {\em Journal on Selected Areas in Communications}, 13(7):1341--1347,
  1995.

\end{thebibliography}
\end{document}